\documentclass[11pt]{article}
\usepackage{amsfonts,amsmath}
\usepackage{latexsym}
\usepackage{amsthm}
\usepackage{amscd}
\usepackage{commath}
\usepackage{epsfig}
\usepackage{amssymb}
\usepackage{caption}
\usepackage{enumitem}
\usepackage{mathtools}
\usepackage{booktabs}
\usepackage[table]{xcolor}
\usepackage[toc,page]{appendix}
\usepackage{array}
\newcolumntype{P}[1]{>{\centering\arraybackslash}p{#1}}
\usepackage{soul}

\usepackage[english]{babel}

\usepackage[letterpaper,top=2cm,bottom=2cm,left=2.5cm,right=2.5cm,marginparwidth=1.5cm]{geometry}

\usepackage[colorlinks=true, allcolors=blue,bookmarksnumbered]{hyperref}
\hypersetup{
	citecolor = red!80!black
}
\usepackage{color,soul}
\setstcolor{red}
\usepackage{float,subcaption}
\usepackage{graphicx}
\usepackage{stackengine}
\usepackage{tikz}
\usetikzlibrary{intersections}
\usepackage[export]{adjustbox}
\usepackage{array}

\usepackage{microtype}
\usepackage[percent]{overpic}

\usepackage{lineno}

\newtheorem{defn}{Definition}[section]

\newtheorem{theorem}[defn]{Theorem}
\newtheorem{prop}[defn]{Proposition}
\newtheorem{lem}[defn]{Lemma}
\newtheorem{cor}[defn]{Corollary}
\theoremstyle{remark}
\newtheorem{remark}[defn]{Remark}
\newtheorem{question}[defn]{Question}

\newtheorem{example}[defn]{Example}
\newtheorem{claim}{Claim}

\usepackage[capitalise]{cleveref}
\crefname{model}{Model}{Models}
\crefname{submodel}{Model}{Models}
\crefname{algorithm}{Algorithm}{Algorithms}
\crefname{equation}{}{}
\crefname{problem}{Problem}{Problems}
\crefname{thm}{Theorem}{Theorems}
\crefname{cor}{Corollary}{Corollaries}
\Crefname{cor}{Corollary}{Corollaries}
\crefname{proposition}{Proposition}{Propositions}
\crefname{prop}{Proposition}{Propositions}
\crefname{defn}{Definition}{Definitions}
\crefname{lem}{Lemma}{Lemmas}
\crefname{conj}{Conjecture}{Conjectures}
\crefname{question}{question}{questions}
\Crefname{question}{Question}{Questions}
\AddToHook{env/cor/begin}{\crefalias{defn}{cor}}
\AddToHook{env/lem/begin}{\crefalias{defn}{lem}}
\AddToHook{env/thm/begin}{\crefalias{defn}{thm}}
\AddToHook{env/prop/begin}{\crefalias{defn}{prop}}
\AddToHook{env/theorem/begin}{\crefalias{defn}{theorem}}
\AddToHook{env/example/begin}{\crefalias{defn}{example}}
\AddToHook{env/problem/begin}{\crefalias{defn}{problem}}
\AddToHook{env/question/begin}{\crefalias{defn}{question}}
\AddToHook{env/remark/begin}{\crefalias{defn}{remark}}
\AddToHook{env/defn/begin}{\crefalias{defn}{defn}}

\numberwithin{equation}{section}
\numberwithin{figure}{section}

\newcommand{\bb}{\begin{equation}}
	\newcommand{\ee}{\end{equation}}

\newcommand{\origin}{\mathbf{o}}

\newcolumntype{?}{!{\vrule width 1.5pt }}
\newlength\savedwidth

\newcommand{\tightoverset}[2]{%
	\mathop{#2}\limits^{\vbox to -.5ex{\kern-0.75ex\hbox{$#1$}\vss}}}

\newcommand{\dfn}[1]{\textbf{\textit{#1}}}

\newcommand\HH{\mathbb{H}} 

\newcommand{\R}{\mathbb{R}}

\newcommand\ST{\,;\;}
\newcommand\stdp{\,:\;}

\newcommand\Vor{\mathrm{Vor}}  

\newcommand{\mynorm}[1]{\left| #1 \right|}
\newcommand{\lip}[0]{\mathrm{Lip}}
\newcommand{\corona}[1]{\widehat{\partial #1}}
\newcommand{\restrict}[2]{{
		\left.\kern-\nulldelimiterspace 
		#1 
		\vphantom{\big|} 
		\right|_{#2} 
}}

\newcommand{\tv}[1]{||#1||_{\mathrm{TV}}}
\newcommand{\myprob}[1]{\mathbb P \left[ #1 \right]}

\newcommand{\edge}[2]{\mathcal{E}_{#1}^{#2}}
\newcommand{\edgeone}[1]{\mathcal{E}_1^{#1}}
\newcommand{\edgezero}[1]{\mathcal{E}_0^{#1}}
\newcommand{\floor}[1]{\lfloor #1 \rfloor}

\def\rlabel #1 #2{\begin{equation} \label{#1} #2 \end{equation}}

\def\rproof{\begin{proof}}
	
	\def\Qed{\end{proof}}

\newcommand{\rtree}[1]{\mathcal{T}^{(#1)}}

\newcommand{\treedist}[1]{\mathrm{d}_{\mathcal{T}^{(#1)}}}

\newcommand{\bs}[1]{\boldsymbol{#1}}

\makeatletter 
\tikzset{ 
	reuse path/.code={\pgfsyssoftpath@setcurrentpath{#1}} 
} 
\tikzset{even odd clip/.code={\pgfseteorule}, 
	protect/.code={ 
		\clip[overlay,even odd clip,reuse path=#1] 
		(current bounding box.south west) rectangle (current bounding box.north east)
		; 
}} 
\makeatother 
\usetikzlibrary{3d,perspective,quotes,angles}

\makeatletter
\let\@fnsymbol\@alph
\makeatother
\title{\textsc{Convergence towards Ideal Poisson--Voronoi tessellations with a focus on Diestel--Leader graphs}}
\author{
	Matteo \textsc{D'ACHILLE}\thanks{Institut Élie Cartan de Lorraine, CNRS, Universit\'e de Lorraine, F-57070 Metz, France\newline $_{}$\hfill  \href{mailto:matteo.d-achille@univ-lorraine.fr}{\texttt{matteo.d-achille@univ-lorraine.fr}}}
	\hspace{4pt}\&
	Ali \textsc{KHEZELI}\thanks{School of Mathematics, Institute for Research in Fundamental Sciences, Tehran, Iran\newline $_{}$\hfill  \href{mailto:alikhezeli@gmail.com}{\texttt{alikhezeli@gmail.com}}} }
\date{}

\begin{document}
\maketitle

\begin{abstract}
We provide necessary and sufficient conditions for convergence towards a unique IPVT on any proper pointed measured metric space. The conditions are that the volume function, when composed with $\log$, is regularly varying and that the limit of the uniform probability measure on a large ball exists in the horocompactification. As an application we prove convergence towards a unique IPVT for higher rank symmetric spaces, which solves an open problem of \cite{MiMe23}. Versions of the general theorem are provided for graphs and edge-measured graphs, where a natural parameter $\xi$ appears. We prove independence on $\xi$ in a specific sense under mild assumptions, which answers an open problem of~\cite{IPVT}. As a main example, we show that the latter holds for the IPVT of Diestel--Leader graphs. We also focus on further properties of this example, in particular, that its IPVT cells are distinguishable, providing the first Cayley graph with this property.
\end{abstract}

\tableofcontents

\section{Introduction and Main Results}
	\paragraph{Ideal Poisson--Voronoi Tessellations} 
	Poisson--Voronoi tessellations are among the most widely studied topics of random geometry. Recently, low-intensity limits of these tessellations have been studied, which have found significant interest and applications. In~\cite{bhupatiraju}, the limits of some observable are studied when the underlying space is a regular tree. \cite{BCP} considered marginals of low-intensity Poisson--Voronoi {tessellations} on the hyperbolic plane, and used it to bound from above the Cheeger constant of closed hyperbolic surfaces in the large-genus limit. It was also conjectured in~\cite{BCP} that, in the low-intensity limit, a random tessellation of the hyperbolic plane is obtained, which is now called an \dfn{ideal Poisson--Voronoi tessellation (IPVT)}.\footnote{It is referred to as ``pointless Poisson--Voronoi tessellation'' in~\cite{BCP}.} This conjecture for the hyperbolic plane was proved rigorously in~\cite{IPVT} as a consequence of a general convergence theorem for a sequence of Voronoi diagrams in an arbitrary metric space. In~\cite{IPVT}, the authors also considered the $d$-dimensional real hyperbolic space, $d\geq 2$, for which they discovered a surprisingly simple Poissonian description of the zero cell (i.e., the size-biased typical cell); and provided further results on the IPVT of infinite regular trees (see \Cref{fig:ipvtt2} for a portrait of a low-intensity Poisson--Voronoi tessellation of the infinite 3-regular tree in three samples). 

    \begin{figure}[!hbtp]
		 \includegraphics[width=.32\linewidth]{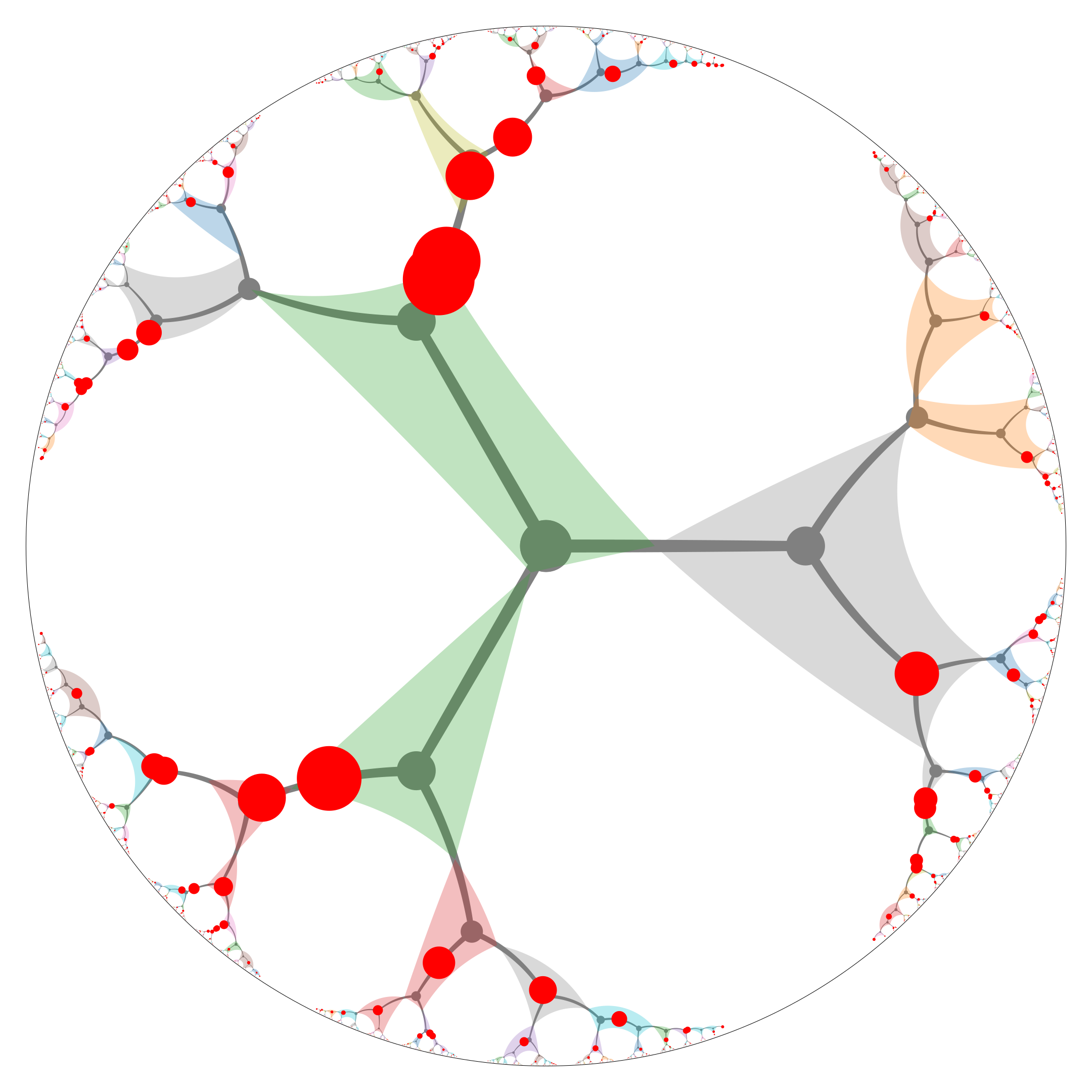}
		 \includegraphics[width=.32\linewidth]{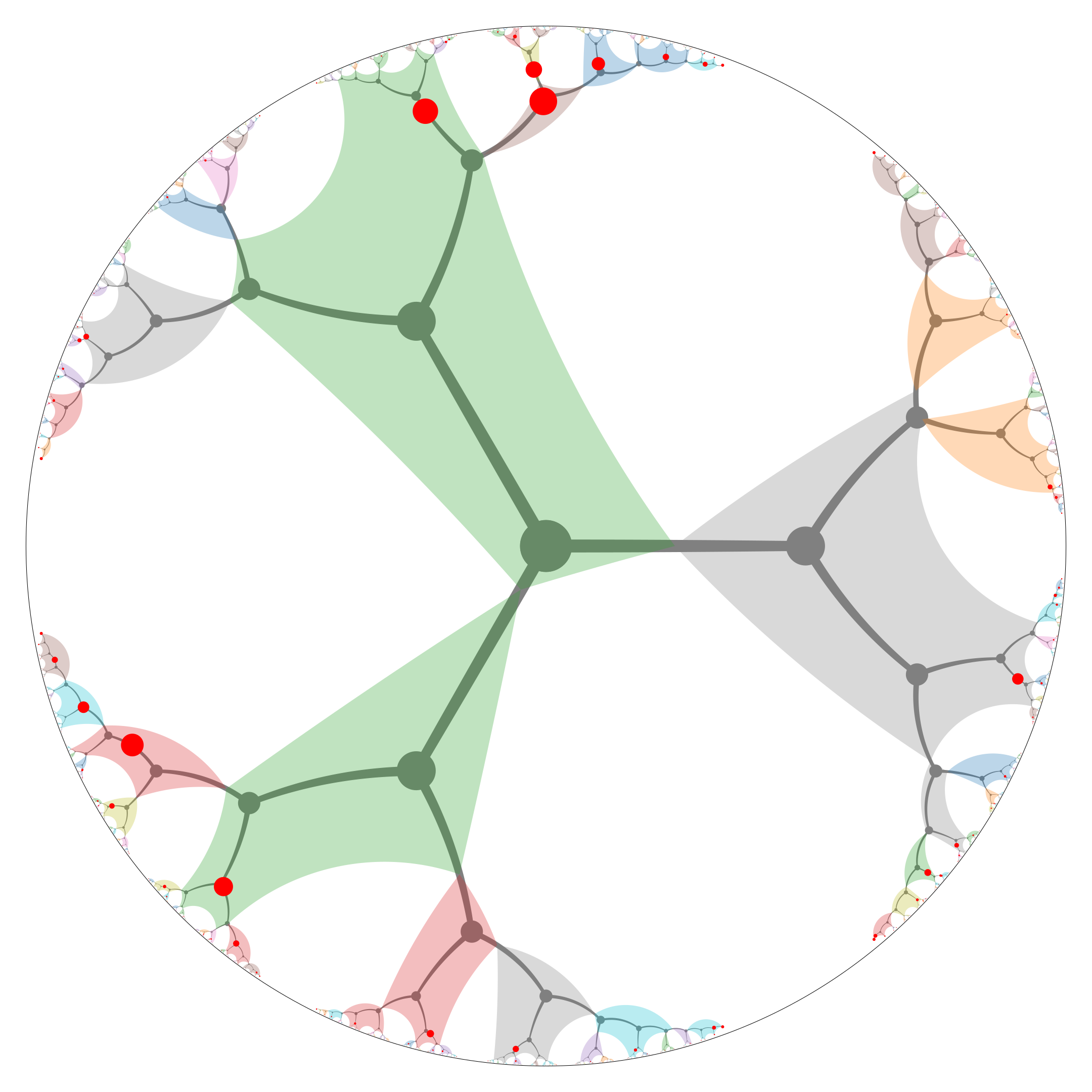}
		 \includegraphics[width=.32\linewidth]{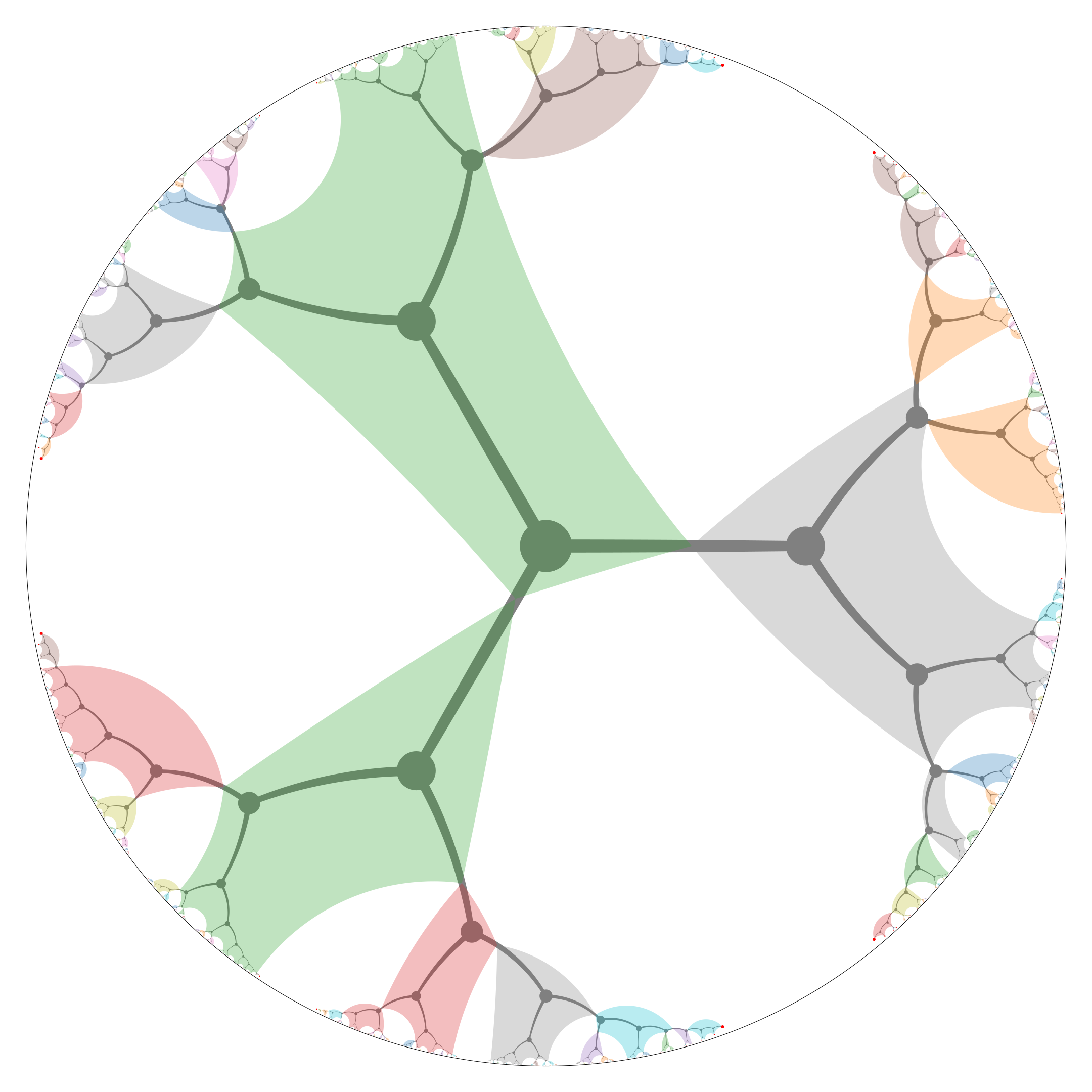}
		\caption{Portraits of Poisson--Voronoi tessellations of $\rtree{2}$ of intensities $2^{-1}, 2^{-3}$ and $2^{-60}$ embedded in the disk model of $\HH^2$. The nuclei are represented by hyperbolic discs of constant radius. The three samples are coupled similarly to the coupling in \Cref{thm.ioxi}. An animation is available at \href{https://alikhezeli.github.io/gallery/treeIPVT.html}{https://alikhezeli.github.io/gallery/treeIPVT.html}, which shows a limiting periodic behavior (see \Cref{thm.ctfemg}) and the independence on $\xi$ mentioned in \Cref{thm.ioxi}.}
		\label{fig:ipvtt2}
	\end{figure}
	
	In order to illustrate further the general convergence theorem {for diagrams} provided in~\cite[Theorem 2.3]{IPVT}, the first author constructed the IPVT of the Cartesian product of two hyperbolic planes endowed with the $L^{1}$ metric in~\cite{IPVTH2H2}. 
	More recently, in~\cite{AT26} all face volume densities  of the IPVT of hyperbolic space of any dimension have been obtained explicitly via low-intensity limits, thanks to a new Blaschke--Petkantschin formula for the finite intensity Poisson--Voronoi tessellation.
	IPVTs and related constructions have found numerous and possibly surprising applications, which have led to the resolution of some open problems. These include some cases of Gaborieau's fixed price conjecture in measured group theory~\cite{MiMe23, BB25, AK25} and constructing examples which have the sparse unique infinite cluster property in percolation theory~\cite{GR25,AGKRW25}.

\paragraph{Convergence Towards Ideal Voronoi Diagrams}	
The approach of~\cite{IPVT} is based on a deterministic convergence criterion for a sequence of Voronoi diagrams in an arbitrary proper metric space $E$ (Theorem~2.3 of~\cite{IPVT}). This criterion assumes that the nuclei points of the diagram, ordered by increasing distance from an arbitrary point $\origin\in E$ (called the \textit{origin}), converge in the \textit{horocompactification} $\bar E$ of $E$ towards a sequence of points in the boundary $\partial E\coloneqq \bar E\setminus E$ (see \Cref{subsec:boundary}). Also, it assumes that the distances of the nuclei to $\origin$, shifted by the first one\footnote{There are other natural choices for shifting the distances as well.  
		For example, in~\cite{MiMe23} the distances are shifted by a constant. But the choices are not much different in considering convergence or tightness, see \Cref{subsec:nuclei}.} (called therein \emph{proto-delays}), converge in $\mathbb R$. With these assumptions, a limiting point process is obtained on the \textit{extended boundary} $\partial E\times \mathbb R$. A point of the extended boundary is usually denoted by a pair $(\theta,\delta)$, where $\delta\in\mathbb R$ is called the \textit{delay} and $\theta\in\partial E$ . Then, if a certain nondegeneracy condition (mentioned in this paper after \Cref{def:voronoi}) is satisfied, the Voronoi diagrams converge (in the product Fell topology) to \emph{ideal Voronoi diagrams}. 

        On the other hand, in~\cite{MiMe23} the authors use subsequential limits of the intensity measures (i.e., the volume measure, scaled and shifted appropriately) to obtain a measure on the extended boundary and to construct an ideal tessellation\footnote{{We emphasize that for the random limit diagram to be a \emph{tessellation} one usually means a Voronoi diagram which is a.s.~locally-finite, normal and face-to-face (see~\cite{bookScWe08}). These properties have been proved in \cite[Proposition 3.7]{IPVT} for the $d$-dimensional real hyperbolic space, $d\geq 2$. Both in \cite{MiMe23} and in \cite{IPVTH2H2} (and in subsequent literature) the authors refer to the limit Voronoi diagram of a Poisson point process as IPVT. The term \textit{weak tessellation} is used in the present paper to stress that the cells may have large overlaps in general, and that no face-to-face property is meaningful.}} in the special case of symmetric spaces. The existence of a convergent subsequence is shown by a precompactness argument based solely on a certain volume growth condition in~\cite[Proposition~3.3]{MiMe23}, which is extended in the present work in \Cref{prop:precompact}. 
	The convergence of the intensity measures is also proved in Theorem~3.6 of~\cite{MiMe23}, however, the convergence of low-intensity Poisson--Voronoi diagrams was not studied there (see \cite[Section 1.1]{MiMe23} and also \cite[Page 6]{WMC}). We prove this convergence here using an application of Theorem~2.3 of~\cite{IPVT} on a class of CAT(0) spaces (see \Cref{cor:cat0}):

\begin{prop}[\textsc{Convergence to IPVT on Symmetric Spaces}]
		\label{prop:symmetric}
		In symmetric spaces, low-intensity Poisson--Voronoi diagrams converge to a unique IPVT.
	\end{prop}
	
	{The proof is given in \Cref{subsec:symmetric}.}
	The idea of subsequential convergence has been used in the further works \cite{GR25,AGKRW25,AK25,BB25}. In the present paper, we focus on the \emph{full convergence} of low-intensity Poisson--Voronoi diagrams, which is more difficult and depends crucially on the geometry of the underlying space. Our first main theorem is the following. 
	{Here, and throughout the paper, unless explicitly mentioned {otherwise}, we let $B_r(\origin)$ (resp. $S_r(\origin)$) be the closed ball (resp. sphere) of radius $r$ centered at $\origin$.}
	Given a measure $\mu$ on $E$, denote by $F(r)\coloneqq \mu(B_r(\origin))$ its volume. We say that $F$ is \textbf{log-regularly-varying} if $F \circ \log$ is regularly-varying; i.e., for all $r_0\in\mathbb R$, the limit of $F(r+r_0)/F(r)$ as $r\to +\infty$ exists (see \Cref{def:reg} for further details). By well-known results about regularly-varying functions, the limit is then equal to $e^{br_0}$ for some $b\in\mathbb R$, which is called the \textbf{index} of $F$.

    \begin{theorem}[\textsc{Full Convergence of Nuclei}]\label{thm.scfc}
		Let $(E,\origin,\mu)$ be a proper pointed measured metric space and, for $\lambda >0$, let $X^{(\lambda)}$ be a Poisson point process with intensity measure $\lambda\mu$. Then, $X^{(\lambda)}$ converges weakly (after a suitable shift) to a nontrivial point process $(\Theta_i,\Delta_i)_{i\geq 1}$ in the extended boundary if and only if:
		\begin{enumerate}[label=(\roman*)]
			\item \label{thm.scfc-i} the volume function $F$ is log-regularly-varying with positive index (say $b$), and
			\item \label{thm.scfc-ii} the probability measure $F(r)^{-1} \restrict{\mu}{B_r(\origin)}$, regarded as a probability measure on $\bar E$, converges weakly to a probability measure on $\partial E$ (say $\nu_0$) as $r\to\infty$. 
		\end{enumerate}  
		
		If these conditions hold, then $(\Theta_i)_{i\geq 1}$ are i.i.d.~points on $\partial E$ with distribution $\nu_0$ and $(\Delta_i)_{i\geq 1}$ is an independent Poisson point process on $\mathbb R$ with intensity measure proportional to the measure $\beta$ defined by $\beta((-\infty,r])\coloneqq e^{br}$. If, in addition, the nondegeneracy condition provided just after \Cref{def:voronoi} holds a.s., then the corresponding Poisson--Voronoi diagrams also converge to the Poisson--Voronoi diagram of $(\Theta_i,\Delta_i)_{i\geq 1}$.
	\end{theorem}
	
	In this theorem, the intensity measure of the limiting point process has a product form, see~\Cref{eq:nu}.
	The proof of \Cref{thm.scfc} is given in \Cref{subsec:fullConvergence}. The topology used for the convergence in the statement will be defined in \Cref{subsec:precompact}, which in particular, prevents the delays escaping to $-\infty$.
	\Cref{thm.scfc} is useful for proving convergence because verifying convergence in $\bar E$ is simpler than in the extended boundary, and the theorem takes care of the delay part. In particular, \Cref{thm.scfc} can be used to obtain a simple proof of full convergence in the well-known cases of hyperbolic spaces and their products (\cite{IPVT} and~\cite{IPVTH2H2}). \Cref{thm.scfc} and the analogous version for graphs given in \Cref{thm.ctfg,thm.ctfemg} are  useful tools for proving convergence, for example in the case of products of two or more metric spaces, such as  $\HH^2 \times \HH^2$, $\HH^2 \times \R$, regular trees, product of regular trees, and Diestel--Leader graphs. These will be discussed in \Cref{sec:ex}.
	
\paragraph{Convergence for Graphs} On the other hand, IPVTs of graphs have also been of great interest; e.g., regular trees \cite{bhupatiraju,IPVT}, products of regular trees \cite{MiMe23,AGKRW25,M26} and Cayley graphs \cite{AK25,BB25}. In this case, one might replace the Poisson point process of nuclei with a Bernoulli point process on the vertices. We prove in \Cref{prop:bernoulli} that this does not affect the limiting behavior.	
	However, the volume function $F$ of graphs cannot be log-regularly-varying due to the natural singularities at integers. Hence, by \Cref{thm.ctfg}, full convergence of nuclei cannot hold. The strongest hope is having a one-parameter family of limiting point processes, as in the following Theorem:
	
	\begin{theorem}[\textsc{Full Convergence for Graphs}]\label{thm.ctfg}
		Assume that $E=G$ is a graph, $\mu$ is any measure on the vertices of $G$ and $\lambda_n\downarrow 0$ such that $F(n)^{-1}\leq \lambda_n \leq F(n-1)^{-1}$. Then, $X^{(\lambda_n)}$ converges weakly (after a suitable shift) to a nonempty point process on the extended boundary (say $(\Theta_i,\Delta_i)_{i\geq 1}$) if and only if:
		\begin{enumerate}[label=(\roman*)]
			\item $\restrict{F}{\mathbb N}$ is log-regularly-varying with positive index (say $b$),
			\item the probability measure $F(n)^{-1} \restrict{\mu}{B_n(\origin)}$ converges weakly to a probability measure on $\partial G$ (say $\nu_0$) as $n\to\infty$, and
			\item $\lim_{n \to \infty} \lambda_n F(n) = e^{b\xi}$ for some $\xi$. 
		\end{enumerate}
		If these conditions hold, then $(\Theta_i)_{i\geq 1}$ are i.i.d. points on $\partial G$ with distribution $\nu_0$ and $(\Delta_i)_{i\geq 1}$ is an independent Poisson point process on $\mathbb Z$ with intensity measure $\beta^{(\xi)}(i)\coloneqq (p-1)p^{i-1+\xi}$, where $p\coloneqq e^b$. In addition, the corresponding Poisson--Voronoi diagrams also converge to the IPVT corresponding to $(\Theta_i,\Delta_i)_{i\geq 1}$, {which we denote by ${\rm IPVT}_{\xi}(G)$}, \underline{without} any non-degeneracy condition.
	\end{theorem}
	
	\noindent
	The proof \Cref{thm.ctfg} is given in \Cref{subsec:graphs}.

\paragraph{Convergence for Edge Measure-Graphs.} We also provide an analogue of \Cref{thm.ctfg} for edge-measured graphs, which are regarded as a union of unit-length segments equipped with the Lebesgue measure on each edge. 
	For this, let $\alpha^{(n)}$ be the probability measure on $G^2$ which represents the uniform measure on the set of edges in $B_n(\origin)$ which are not contained in $B_{n-1}(\origin)$. 
	We assume that $\alpha^{(n)}$ is symmetric with respect to $(x,y)\mapsto(y,x)$.
	For simplicity, we first state the result in the bipartite case: 
	
	\begin{theorem}[\textsc{Full Convergence  for Edge-Measured Bipartite Graphs}]\label{thm.ctfemg}
		Let $E$ be the edge-measured graph corresponding to a bipartite graph $G$. Consider the measure $\alpha^{(n)}$ on $G^2$, defined above, and let $\lambda_n\downarrow 0$ such that $F(n)^{-1}\leq \lambda_n \leq F(n-1)^{-1}$. Then, $X^{(\lambda_n)}$ converges weakly (after a suitable shift) to a nonempty point process $(\Theta_i,\Delta_i)_{i\geq 1}$ on the extended boundary if and only if:
		\begin{enumerate}[label=(\roman*)]
			\item \label{thm.ctfemg-item1} $\restrict{F}{\mathbb N}$ is log-regularly-varying with positive index $b$,
			\item \label{thm.ctfemg-item2} the measure $\alpha^{(n)}$ converges weakly as $n\to\infty$ to a probability measure $\alpha$ on $(\partial G)^2$,
			\item \label{thm.ctfemg-item3} and $\lim_{n\to \infty} \lambda_n F(n) = e^{b\xi}$ for some $\xi$.
		\end{enumerate}
		If these conditions hold, then $(\Delta_i)_{i\geq 1}$ is a Poisson point process on $\mathbb R$ with intensity measure 
		$$
		\beta^{(\xi)}\coloneqq \sum_{j\in \mathbb Z} (p-1)p^{j+\xi} \mathrm{Leb}_{[j,j+1]} \; ,
		$$ where $p\coloneqq e^b$ and $\mathrm{Leb}_{[j,j+1]}$ is the Lebesgue measure on $[j,j+1]$.
		In addition, conditionally on $(\Delta_i)_{i\geq 1}$, the points $(\Theta_i)_{i\geq 1}$ are independent points on $\partial E$ such that $\Theta_i$ has distribution $\nu_{\{\Delta_i\}}$, where $\{\cdot\}$ denotes the fractional part and, for all $t\in[0,1)$, $\nu_t$ is a measure on $\partial E$ determined by $\alpha$. In addition, the corresponding Poisson--Voronoi diagrams also converge to the IPVT of $(\Theta_i,\Delta_i)_{i\geq 1}$, which we denote by ${\rm IPVT}_{\xi}(E)$.
	\end{theorem}

Heuristically, $\nu_t$ is obtained from $\alpha$ by the operation which, given an edge $(x,y)$, returns a point on the edge with distance $t$ from $x$. This heuristic is formalized on the boundary of $G$ in \Cref{subsec:partialEgraphs} and the precise formula of $\nu_t$ is given in~\eqref{eq:thm:criterion-edge:nu}. 

	\Cref{thm:criterion-edge} provides an extended version of \Cref{thm.ctfemg} by allowing weights on the edges, and also covers the non-bipartite case. 
	In this case, the uniform measure on the \textit{same-level edges} (i.e., the edges whose end points have the same distance to $\origin$) are also assumed to converge to a probability measure on $(\partial G)^2$, say $\alpha_0$. This probability measure will be used in \Cref{thm.ioxi} below.

\paragraph{Independence on $\xi$.} Despite the fact that the nuclei process cannot fully converge in the case of graphs, \cite{bhupatiraju} observed the remarkable fact that, in the case of edge-measured regular trees, if one forgets the nuclei and considers only the restriction of the tessellations on the vertices, then the limiting distributions do not depend on the parameter $\xi$. This fact is highly nontrivial and is proved in that case of edge-measured regular trees by extensive calculations. It is asked in~\cite[Page 46]{IPVT} whether this independence on $\xi$ can hold in Cayley graphs other than trees as well. We answer this problem affirmatively by the following Theorem and provide a simple proof using a coupling technique. In the statement,  $G^{1/2}$ denotes the set of vertices and midpoints of the edges of $G$:
	
	\begin{theorem}[\textsc{Independence on $\xi$ (Short Version)}]\label{thm.ioxi}
		Let $E$ be the edge-measured graph corresponding to a graph $G$. Assume that $\restrict{F}{\mathbb N}$ is log-regularly-varying with index $b>0$ and that the limiting measures $\alpha$ and $\alpha_0$ defined above exist. If $\alpha_0=0$ and  $\alpha$ is supported on the diagonal of $(\partial G)^2$, then the distribution of {$\restrict{{\rm IPVT}_{\xi}(E)}{(G^{1/2})}$} does not depend on $\xi$. In addition, low-intensity Poisson--Voronoi diagrams on $E$, restricted to $G^{1/2}$, converge weakly. Moreover, there exists a coupling of ${\rm IPVT}_{\xi}(E)$ for all $\xi\in\mathbb R$ that is continuous and 1-periodic in $\xi$ such that $\forall \xi: \restrict{{\rm IPVT}_{\xi}(E)}{(G^{1/2})} = \restrict{{\rm IPVT}_{0}(E)}{(G^{1/2})}$ a.s.
	\end{theorem}

	\Cref{thm:xi} provides an extended version of \Cref{thm.ioxi}.
	The condition of being supported on the diagonal ensures that the measures $\nu_t$ in \Cref{thm.ctfemg} do not depend on $t$, which gives a nice decomposition of the intensity measure (see~\eqref{eq:thm:xi-beta}) and implies that the nuclei of the limiting point process are independent from the delays. See also \Cref{rem:off-diagonal}.
	{In fact, this condition is not too strong and is expected to hold in many examples. Assuming $\alpha_0=0$, this condition means that \textit{thin $1\times n$ rectangles} are rare, see \Cref{lem:rectangle}.
		In particular, the next proposition applies the last theorem to some examples. In the statement, ${\rm DL}(p,q)$ denotes the Diestel--Leader graph with parameters $p$ and $q$, which will be defined shortly, and $\rtree{p}$ denotes the $(p+1)$-regular tree.
		\begin{cor}
			\label{cor:xi-examples}
			The following graphs satisfy the independence on $\xi$ mentioned in \Cref{thm.ioxi}:
			\begin{enumerate}[label=(\roman*)]
				\item The Diestel--Leader graph ${\rm DL}(p,q)$, which in the case $p=q$, is a Cayley graph of the lamplighter group $\mathbb{Z}_{p} \wr \mathbb{Z}$.
				\item The direct product $\rtree{p}\times \rtree{p}$, which is a Cayley graph of the product of two free groups if $p$ is even.\footnote{Note that if $p$ is odd, $\rtree{p}\times \rtree{p}$ is the Cayley graph of $(\mathbb{Z}_2 * \cdots * \mathbb{Z}_2)^2$ where duplicated edges are removed.}
			\end{enumerate}
			In contrast, $\rtree{p}\times \rtree{q}$, when $p>q\geq 1$, does not satisfy independence on $\xi$.
		\end{cor}
		The proof will be given in \Cref{subsec:regularTree} for products of trees, and at the end of \Cref{subsec:dl-conv} for Diestel--Leader graphs.
	}

\paragraph{IPVT of the Diestel--Leader Graph.} 
	{Diestel--Leader graphs ${\rm DL}(p,q)$ are a family of transitive graphs which have interesting properties in various aspects, like amenability, behavior of the simple random walk, and non-quasi-invariance with Cayley graphs (when $p\neq q$), to name a few {(see \emph{e.g.}, \cite{woess2005,bookLyPe16})}. In this paper, we describe their IPVT and obtain further interesting properties. In particular, we have claimed in \Cref{cor:xi-examples} that they satisfy independence on $\xi$, mentioned above. We will also show in \Cref{prop:distinguishable} that the cells of the IPVT are \textit{distinguishable} (in the directed version), in contrast to regular trees and the hyperbolic space (see \cite{M26}), which provide the first known Cayley graphs satisfying this property.
		
		We recall briefly the definition of ${\rm DL}(p,q)$, see \Cref{sec:dl} for further details. Let $T=\rtree{p}$ be the $(p+1)$-regular tree and $T'=\rtree{q}$. Fix ends $\omega_0$ of $T$ and $\omega'_0$ of $T'$. The vertices of ${\rm DL}(p,q)$ are $(x,x')\in T\times T'$ such that $d_{\omega_0}(x)+d_{\omega'_0}(x')=0$, where $d_{\omega_0}$ denotes the horofunction (or \textit{height function}) corresponding to $\omega_0$. The edges are also defined naturally, see \Cref{def:dl}. We will describe the horoboundary $\partial {\rm DL}(p,q)$ in \Cref{prop.gbdl,prop.gbdl-conv,prop:DL-top}. Roughly speaking, as sets, one can write $\partial {\rm DL}(p,q) = \overline T\sqcup\overline {T'}\sqcup\mathbb Z$ (see also Figure~\ref{fig:DLboundary}).} {By proving sharp estimates on the volume growth (see Equations~\eqref{eq:dl-growth1}  to~\eqref{eq:dl-growth2edge}) and calculating the limit of the uniform measure on large balls, we prove:}
	
	\begin{theorem}[IPVT of ${\rm DL}(p,q)$]
		\label{thm.dl-convergence}
		The Diestel--Leader graph ${\rm DL}(p,q)$ satisfies the conditions of \Cref{thm.ctfg}. If $p\geq q$, then $\restrict{F}{\mathbb N}$ is log-regularly-varying with index $\ln p$. In addition, the measure $\nu_0$ on $\partial {\rm DL}(p,q)$ has the following description:
		\begin{enumerate}[label=(\roman*)]
			\item If $p>q$, then $\nu_0$ is concentrated on $\partial T\subseteq \partial {\rm DL}(p,q)$ and is the harmonic measure on $\partial T$. In this case, {${\rm IPVT}_{\xi}({\rm DL}(p,q)) = \pi^{-1}({\rm IPVT}_{\xi'}(T))$ for some $\xi'$}, where $\pi:{\rm DL}(p,q)\to T$ denotes the projection on the first coordinate.
			\item If $p=q$, then $\nu_0$ is the average of the harmonic measures on $\partial T$ and $\partial T'$. 
		\end{enumerate}
		In addition, the edge-measured version of ${\rm DL}(p,q)$ satisfies \Cref{thm.ctfemg} with the same growth index $\ln p$, and the measures $\nu_t$ in \Cref{thm.ctfemg} are identical to the measure $\nu_0$ of the vertex-measured version.
	\end{theorem}
	\noindent

	We will also prove further geometric properties of ${\rm IPVT}({\rm DL}(p,q))$. In \Cref{prop:intersection}, {using the mass transport principle on a specific subset of nuclei}, we show that two cells can have empty, finite or infinite intersection (or adjacency). Also, in \Cref{prop:dl-biinfinite}, we show that the IPVT cells do not contain bi-infinite geodesics.

	\noindent \textbf{Structure of the paper}. 
	{In \Cref{sec:ipvt}, we define the horoboundary and the IPVT of an arbitrary rooted measured metric space. \Cref{sec:full} studies precompactness of Poisson--Voronoi tessellations, and also proves the full convergence theorem towards a unique IPVT (\Cref{thm.scfc}). In \Cref{sec:graphs}, we prove extended versions of the full convergence theorems for graphs (\Cref{thm.ctfg,thm.ctfemg}) and independence on $\xi$ (\Cref{thm.ioxi}). In \Cref{sec:ex}, the general theorems are used to prove convergence to IPVT in various settings, like CAT(0) spaces, symmetric spaces, and products. It also studies further properties of the IPVT of trees, like one-endedness and the description of the zero-cell via a percolation. The properties of the IPVT of Diestel--Leader graphs are provided in \Cref{sec:dl}. Finally, problems and future research directions are discussed in \Cref{sec:problems}.
	}
	
	\medskip
	\noindent \textbf{Acknowledgements}. We thank Nicolas Curien and Mikolaj Fraczyk for useful comments on an earlier version of this manuscript and for suggesting Question \ref{q:mf}.  M.D'A. is grateful to the IHES for kind hospitality in the spring of 2025 when part of this work has been done. Both authors acknowledge the conference ``Probability and Geometry in, on and of non-Euclidean space'' at CIRM Luminy, where this work was started. The work of M.D'A.~has been partially supported by the ERC Consolidator Grant SuperGRandMA (Grant No.~101087572) and by the ANR project LOUCCOUM (ANR-24-CE40-7809).

\section{IPVT of Measured Metric Spaces}
	\label{sec:ipvt}	
	
	{In this section, we provide the general definition of IPVT and the horoboundary. These definitions build on the abstract deterministic framework of \cite[Section 2]{IPVT}, on ideas of~\cite{MiMe23}, and classical notions from stochastic geometry, in particular, point processes of closed subsets}.

	\subsection{Boundary and Extended Boundary}
	\label{subsec:boundary}
	
	Let $(E,d)$‌ be a boundedly-compact (sometimes called proper or Heine-Borel, see~\cite{WJ87}) metric space; i.e., every bounded closed subset of $E$‌ is compact. Fix an element $\origin\in E$‌ and call it the \dfn{origin}.
	A‌ \dfn{shifted distance function} is a function on $E$‌ of the form $d(\cdot, x)-t$, where $x\in E$‌ and $t\in \mathbb R$. Let $\mathrm{SD}(E)$ denote the space of shifted distance functions, which is a subset of the space $\lip_1(E)$ of 1-Lipschitz functions on $E$. The latter is equipped with the topology of uniform convergence on compact sets, which is equivalent to pointwise convergence. Consider the equivalence relation on $\lip_1(E)$ by regarding every $f\in\lip_1(E)$ equivalent to $f+t$ for all $t\in\mathbb R$. Let $\mathrm{SD}(E)/\mathbb R$‌ and $\lip_1(E)/\mathbb R$ denote the quotients of $\mathrm{SD}(E)$‌ and $\lip_1(E)$ by this equivalence relation and equip them with the quotient topology. For every $f\in \lip_1(E)$, the function $f(\cdot)-f(\origin)$ is the unique function in the equivalence class of $f$‌ whose value at $\origin$ is zero. So, $\lip_1(E)/\mathbb R$ is homeomorphic to $\lip_1^0(E,\origin)$, where the latter is the set of $f\in\lip_1(E)$ such that $f(\origin) = 0$. For brevity of notations, we use the symbols $\lip_1, \lip_1^0$, etc when $E$ and $\origin$ are known.

    The mapping $E\to \mathrm{SD}$ defined by $x\mapsto (y\mapsto d(x,y))$ induces a homeomorphism from $E$ to $\mathrm{SD}(E)/\mathbb R$, which is denoted by $\iota$ (the greek letter \emph{iota}). The \dfn{horocompactification} of $E$ is the closure of $\iota(E)=\mathrm{SD}(E)/\mathbb R$‌ in $\lip_1(E)/\mathbb R$ and is denoted by $\bar E$ (notice that this is indeed a compact set since we have assumed that $E$‌ is boundedly-compact). The \dfn{horoboundary} (also called \dfn{Gromov boundary} in \cite{IPVT}) of $E$‌ is the boundary of $\iota(E)$‌ and is denoted by $\partial E$. So, $\bar E$‌ is a disjoint union of $E$‌ and $\partial E$. 
	
	\smallskip
	The \dfn{extended horoboundary} of $E$‌ is the boundary of $\mathrm{SD}(E)$ in $\lip_1$ and is denoted by $\corona{E}$. Let also $\overline{\mathrm{SD}}(E)=\mathrm{SD}(E)\cup\corona{E}$ denote the closure of $\mathrm{SD}(E)$ in $\lip_1$. Elements of $\corona{E}$ are pointwise limits of functions of the form $d(\cdot, x_n)-t_n$ such that $d(x_n,\origin)\to\infty$ for sequences $(x_{n})_{n\geq 1}$ and $(t_{n})_{n\geq 1}$ as $n \to \infty$, and are called \dfn{horofunctions}. When the origin $\origin$ is fixed, for every $\theta\in\partial E$, there exists a unique $d_{\theta}\in\corona{E}$ which represents $\theta$ and satisfies $d_{\theta}(\origin)=0$.
	The terminology of \textit{extended horoboundary} is chosen because every element of $\corona{E}$ can be uniquely represented as $d_{\theta}+\delta$, where $\theta\in\partial E$‌ is a boundary point and $\delta\in\mathbb R$‌ is called the \dfn{delay}. In fact, this gives a homeomorphism between $\corona{E}$ and $\partial E\times \mathbb R$. So, following \cite[Section 2.1]{IPVT}, elements of the extended horoboundary $\corona{E}$ are denoted by pairs $(\theta,\delta)$‌ in this paper (when the origin is known).

	\begin{remark}[\textsc{Non-Quasi-Isometry-Invariance}]
		\label{rem:totallydisconnected}
		The horoboundary is not invariant under quasi-isometries. For example, if $E$ is the vertex set of a graph, then $\partial E$ is totally disconnected (since the space of integer-valued elements of $\lip_1^0(E)$ is totally disconnected). When $E$ is the graph of a mosaic of the hyperbolic plane $\mathbb H^2$ by congruent regular $p$-gons in which each vertex has degree $q$, $\frac{1}{p}+\frac{1}{q}<1/2$, then $\partial E$ is different from $\partial \mathbb H^2$ (the latter being homeomorphic to the unit circle $\mathbb{S}^{1}$). This might seem counter intuitive since a regular mosaic is quasi-isometric to $\mathbb H^2$. As a result, the horoboundary is not quasi-isometry-invariant. 
		\\
		As another example, the horoboundary of a Cayley graph depends crucially on the choice of the generating set of the group. For instance, if $G_1$ and $G_2$ are the Cayley graphs of $\mathbb Z^2$ corresponding to the generating sets $\{(1,0),(0,1)\}$ and $\{(1,0),(0,1),(1,1)\}$ respectively, then $\partial G_1$ has 4 non-isolated points, while $\partial G_2$ has 6 non-isolated points.
	\end{remark}

    \subsection{Limits of the Poisson--Voronoi Diagram}
	\label{subsubsec:ipvt}
	
	Now, assume a boundedly-finite (i.e., Radon) Borel measure $\mu$ is given on $E$ such that $\mu(E)=\infty$. Throughout, we will make  use of the \dfn{volume function}: it is the non-decreasing, right-continuous non-negative function defined by 
	\[F(r)\coloneqq\mu(B_r(\origin)),\] 
	where $B_r(\origin)$ is the closed $d$-ball of radius $r$ centered at $\origin$. 
	\begin{defn}[\textsc{Nuclei Process}]
		\label{def:PPP}
		For each $0<\lambda<\infty$, the \textbf{nuclei process} is a marked Poisson point process $X^{(\lambda)}$ with intensity measure $\lambda \mu$ and with uniform i.i.d.~marks\footnote{We don't use a symbol for the marks because they are only used to break the ties shortly.} in the unit interval $[0,1]$. Given the origin $\origin$, sort the points of $X^{(\lambda)}$‌ by increasing distance from $\origin$; i.e., $X^{(\lambda)}_1$ is the point nearest to $\origin$, $X^{(\lambda)}_2$ is the second nearest point to $\origin$, and so on. If there are points with equal distance to $\origin$, sort by marks to break the ties.
	\end{defn}
	
	In fact, one can forget the marks in the unordered version. Given the nuclei process $X^{(\lambda)}$, our goal is to study the limit as $\lambda\downarrow 0$ of its Voronoi diagram, whose definition we recall now.
	
	\begin{defn}[\textsc{Voronoi Cells and Diagram}]
		\label{def:voronoi}
		Given a sequence $\Phi \coloneqq (f_i)_{i\geq 1}$ in $\corona{E}$, the \dfn{Voronoi cells} $(C_i)_{i\geq 1}$ are defined by 
		$$
		C_i\coloneqq\left\{y\in E \ST f_i(y)\leq f_j(y), \; \forall j \geq 1 \right\} \; .
		$$
		The \dfn{Voronoi diagram} of $\Phi$ is $\Vor(\Phi)\coloneqq(C_i)_{i\geq 1}$.	
	\end{defn}
	Note that Voronoi cells can overlap in general. The \dfn{strict Voronoi cells} are disjoint sets defined by the strict inequality, namely $C_i^{\circ}\coloneqq \{y\in E \ST f_i(y)< f_j(y), \forall j\neq i\}$.\footnote{Another option is to break the ties using the marks of the points, as in~\cite{GR25,AGKRW25}, which leads to a partition of $E$. But we don't use this definition here.}. The Voronoi diagram (or just $\Phi$) is called \dfn{nondegenerate} if $C_i=\overline{C_i^{\circ}}, \forall i$.
	Note that the strict Voronoi cells are open provided that $(f_i)_i$ is \textbf{admissible} in the sens of \cite{MiMe23,IPVT}; i.e., $f_{i}(\origin)\to +\infty$. Admissibility also implies that the cells $(C_i)_{i \geq 1}$ are locally-finite (i.e., for every radius $T\geq 0$, only finitely many cells intersect $B_T(\origin)$). In this case, we say that the Voronoi diagram is a \textbf{(weak) tessellation} in the sense of a locally-finite family of closed subsets of $E$ that cover $E$ (but may have large overlaps).

    To study limits of Voronoi diagrams, {as in~\cite{AK25},} we use the notion of \dfn{point processes of closed subsets of $E$}. Such a point process is a random locally-finite unordered collection of elements of $\mathcal F(E)$, where $\mathcal F(E)$ denotes the space of nonempty closed subsets of $E$ equipped with the Fell topology. Indeed, admissibility of $\Phi$ implies that $\Vor(\Phi)$ (after forgetting the order of $(f_i)_{i \geq 1}$) satisfies this property. Note that the point process is not necessarily simple; i.e., two or more cells might be identical. See \cite[Section 3.6]{bookScWe08} for more detail on this topology. In addition, one might consider the pair $(\Phi,\Vor(\Phi))$ (together with the mapping between the nuclei and the cells, which is omitted in the notation) as a \dfn{point process of pointed closed subsets}; i.e., a point process in $\widehat{\partial E}\times \mathcal F(E)$ (see \Cref{prop:diagramConvergence}).
	
	Alternatively, as in~\cite{IPVT}, one might consider $\Vor(\Phi)$ as a random element of $(\mathcal F(E)\cup\{\emptyset\})^{\mathbb N}$ (while keeping the order of $(f_i)_{i\geq 1}$) and use the notion of convergence in the product topology on $(\mathcal F(E)\cup\{\emptyset\})^{\mathbb N}$. In this paper, we prefer to use the abovementioned topology because it enables directly automorphism-invariance (see \Cref{lem:autInv} below). Actually, under some natural conditions, the two topologies are identical; see \Cref{def:generic,cor:generic}.
	
	\begin{defn}[\textsc{IPVT}]
		\label{def:ipvt}
		Let $X^{(\lambda)}$ be the nuclei process as in \Cref{def:PPP}. Assume $\lambda_n\to 0$ is a sequence for which $\Vor(X^{(\lambda_n)})$ is convergent (as a point process of closed subsets). Then, the limit is called an \textbf{ideal Poisson--Voronoi tessellation (IPVT)} of $(E,\mu)$. If, in addition, $\lim_{\lambda\to 0} \Vor(X^{(\lambda)})$ exists, it is called \dfn{the IPVT} of $(E,\mu)$. 
	\end{defn}
	
	For example, in $\mathbb R^d$ (equipped with the Lebesgue measure), the IPVT exists and is trivial: It contains a single cell that occupies the whole $\mathbb R^d$.  {But this is not the case in hyperbolic spaces, as it is recalled later.}

	\begin{lem}[\textsc{Automorphism-Invariance}]
		\label{lem:autInv}
		Any IPVT is a random weak tessellation whose distribution is invariant under the automorphisms of $(E,\mu)$.
	\end{lem}
	\begin{proof}
		Under the topology described above, it can be seen that the set of weak tessellations is closed. It implies that any IPVT is a random weak tessellation. Also, the distribution of (the unordered version of) $\Vor(X^{(\lambda)})$ is invariant under the automorphisms of $(E,\mu)$, even if $\origin$ is not preserved. This implies that any weak limit of $\Vor(X^{(\lambda)})$ is also automorphism-invariant.
	\end{proof}

	\begin{remark}
		\label{rem:descendingpath}
		If $E$ is a geodesic space, $\theta$ is a horofunction and $x\in E$, one can show that there exists an infinite geodesic $\gamma$ starting from $x$ such that $\forall t\geq 0: \theta(\gamma(t))=\theta(x)-t$. However, $\gamma$ might not always converge to $\theta$. This geodesic might be helpful for proving or disproving non-degeneracy (see e.g., \Cref{lem:nondegenerate-cat0}). 
	\end{remark}

    \begin{remark}[\textsc{Star-Like Cells}]
		In every Voronoi diagram, the cells are star-like in the following sense: If $\theta$ is the center of a cell $C$ and $x\in C$, then $C$ contains the set $[x,\theta] \coloneqq  \{y: d(x,y)+d_\theta(y) = d_{\theta}(x)\}$.
		In particular, if $E$ is a graph or a geodesic space and $\theta\in\partial E$, then \Cref{rem:descendingpath} implies that $C$ contains an infinite geodesic in $[x,\theta]$ starting from $x$. However, $C$ might or might not contain a bi-infinite geodesic.
	\end{remark}
	\begin{remark}
		It is useful to mention some statements that do not necessarily hold for Voronoi diagrams, even if $E$ is a geodesic space: The cells are not necessarily geodesically convex not even connected (see the example in \Cref{subsec:Busemann}). Even if the cells are connected, the intersection of two cells might be disconnected. The closure of a cell might contain points in $\partial E$ other than the nucleus of the cell (e.g., when the cell contains a bi-infinite geodesic). Two cells might have unbounded intersection, which is a major feature used in~\cite{MiMe23}. 
	\end{remark}

	\subsection{Limits of the Nuclei Process}
	\label{subsec:nuclei}
	
	As mentioned in the introduction, Poisson--Voronoi diagrams may have a nontrivial limit despite that the nuclei escape to infinity. In this subsection, we describe how to \textit{shift} the distance functions, as in~\cite{IPVT,MiMe23}, to obtain a nontrivial limit of the nuclei process. {A detailed description of the topology is given in \Cref{subsec:precompact}}.
	
	For every $t\in\mathbb R$, consider the embedding $\iota_t:E\to \mathrm{SD}(E)$ defined by
	\[\iota_t(x)\coloneqq (y\mapsto d(x,y)-t).\] 
	It is clear that $\Vor(\iota_t(X^{(\lambda)}))= \Vor(X^{(\lambda)})$ for all $t$. So, if there exists some (possibly random) $t=t(\lambda)$ such that $\iota_{t(\lambda)}(X^{(\lambda)})$ converges weakly to a nontrivial point process on $\corona{E}$ (see \Cref{subsec:precompact} for the description of the topology), possibly along a subsequence, this limiting point process might be used as a way to study the limiting behavior of $\Vor(X^{(\lambda)})$.

    Note that adding a constant to $t(\lambda)$ does not affect the behavior of the tessellation, but it affects the limit of $\iota_{t(\lambda)}(X^{(\lambda)})$. 
	
	For the above goal, it is important to choose the shift suitably to keep track of the closest point of $X^{(\lambda)}$ to the origin. There are two natural choices for $t(\lambda)$:
	\begin{itemize}
		\item In~\cite{IPVT}, the choice in the general case is $t(\lambda)\coloneqq d(X^{(\lambda)}_1, \origin)$, where $X^{(\lambda)}_1$ is the nucleus nearest to $\origin$;
		\item In~\cite{MiMe23}, and in the present paper, $t(\lambda)$ is deterministic and is defined by $t(\lambda)\coloneqq F^{-1}(\lambda)\coloneqq  \inf\{r: \lambda \cdot F(r)\geq 1\}$.
	\end{itemize}
	The benefit of the first choice is that the delay of $X^{(\lambda)}_1$ is controlled (preventing it from escaping to $\pm \infty$). The second choice is also natural since $\iota_t(X^{(\lambda)})$ is a Poisson point process on $\lip_1$ with intensity measure $\lambda \iota_t^*\mu$, and hence, the measure of $\lip_1^{\leq 0}$ is controlled by the second choice. More importantly, since the second choice does not depend on $\origin$, the limiting distributions are automatically invariant under the automorphisms of $(E,\mu)$; see also \Cref{lem:autInv}. However, this choice does not a priori prevent the delays from escaping to $-\infty$ (which is the case in the Euclidean space).
	
	Nevertheless, it is easy to see that, under nice growth conditions (see \Cref{prop:precompact}), the two choices are not much different since $d(X^{(\lambda)}_1, \origin)- \inf\{r: \lambda \cdot F(r)\geq 1\}$ is tight. 
	In this case, precompactness of $\iota_{t(\lambda)}(X^{(\lambda)})$ holds in either definition (\Cref{prop:precompact}). Also, under the stronger assumptions of \Cref{thm.scfc},  full-convergence holds with both definitions, see also \Cref{lem:criterion2}.

	\begin{remark}
		There are some examples where $\Vor(X^{(\lambda)})$ converges while $\iota_{t(\lambda)}(X^{(\lambda)})$ cannot fully converge for any choice of $t(\lambda)$, see \Cref{thm.ioxi,thm:xi}.
	\end{remark}

    \subsection{On the Relation between Horofunctions, Busemann Functions and the Visual Boundary of CAT(0) Spaces}
	\label{subsec:Busemann}

    In any metric space $E$, if $\gamma$ is an infinite geodesic, then the function $x\mapsto \lim_{t\to\infty} d(x,\gamma(t))-t$ is well defined and is called the \dfn{Busemann function} corresponding to $\gamma$. If the limit is uniform on bounded subsets of $E$ (which holds in graphs and in CAT(0) spaces), then every Busemann function is a horofunction. 
	But not necessarily all horofunctions are Busemann functions, even if $E$ is a proper geodesic metric space. In the last case, for every $\theta\in\partial E$, there exists an infinite geodesic $\gamma$ such that $\gamma(0)=\origin$ and $d_{\theta}(\gamma(t))=-t, \forall t$, but $\gamma$ does not necessarily converge to $\theta$; e.g., when $E=\mathbb R^2$ is endowed with the $L^1$ metric.  Another example is given by $E=\cup_{n \geq 1} C_n\subseteq \mathbb R^2$ , where $C_n$ is the square whose opposite corners are the points $(0,0)$ and $(n,n)$, and the distance is given by the shortest-path metric. Clearly, there is no geodesic that converges to $\theta \coloneqq \lim_{n\to \infty} (n,n)\in\partial E$. In this case, $d_{\theta}$ is not a Busemann function. {Similar \textit{traps} exist in the Diestel--Leader graph as well (see \Cref{sec:dl}). Indeed, only the horofunctions $\rho_{\omega}$ and $\rho'_{\omega'}$ defined in \Cref{prop.gbdl-conv} are Busemann functions.}
	
	However, in a complete CAT(0) space, the two notions are equivalent: \cite[Theorem~II.8.13]{BH} shows indeed that the horofunctions are precisely the Busemann functions and  the horoboundary $\partial E$, defined in \Cref{sec:ipvt}, coincides with the \textit{visual boundary} of $E$. The latter is the set of equivalence classes of infinite geodesic rays $\gamma:[0,\infty)\to E$, where $\gamma$ is equivalent to $\gamma'$ if $\sup d(\gamma(t),\gamma'(t))<\infty$. See \Cref{subsec:cat0} for more on the IPVT of CAT(0) spaces.

\section{General Convergence Results}
	\label{sec:full}
	
	\subsection{Subsequential Convergence and Pre-Compactness}
	\label{subsec:precompact}
	
	As discussed in \Cref{subsec:nuclei}, the shifted nuclei process $\iota_{t(\lambda)}(X^{(\lambda)})$ is used to study the limiting tessellations.  
	As mentioned therein, we need to prevent the delays from escaping to $-\infty$. For this goal, we \textit{compactify $\lip_1\equiv \lip_1^0\times\mathbb R$ at delay $-\infty$} by replacing it with $\lip_1^0\times [-\infty,\infty)$. This is useful since every point process in $\lip_1^0\times[-\infty,\infty)$ (e.g., $\iota_{t(\lambda)}(X^{(\lambda)})$) is automatically admissible, where the latter is defined after \Cref{def:PPP}. Equivalently, \cite{IPVT,BB25} consider the map from $\lip_1^0\times\mathbb R$ to $\lip_1^0\times [0,\infty)$ defined by $(f,\delta)\mapsto (f,e^{\delta})$ and consider the vague topology for measures on $\lip_1^0\times [0,\infty)$.

	We always study the limits of the shifted nuclei processes when they are regarded as point processes on $\lip_1^0\times [-\infty,\infty)$. This is described further as follows. Let $\mathcal M$ be the space of Borel measures $\varphi$ on $\lip_1$ such that $\varphi(\lip_1^{\leq t})<\infty, \forall t$. Let $\mathcal M'$ be the space of Radon measures on $\lip_1^0\times [-\infty,+\infty)$. By considering the map from $\lip_1$‌ to $\lip_1^0\times [-\infty,+\infty)$, $\mathcal M$ can be identified with a subset of $\mathcal M'$ (those elements of $\mathcal M'$ that vanish on $\lip_1^0\times\{-\infty\}$).
	\begin{defn}[\textsc{Topology on $\mathcal{M}$}]
		We always equip $\mathcal M$ with the topology induced from the vague convergence on $\mathcal M'$. 
	\end{defn}
	Roughly speaking, this is the vague topology at $+\infty$ and the weak topology at $-\infty$.
	This allows one to define the weak convergence of probability measures on $\mathcal M$ or $\mathcal M'$ and weak convergence of point processes on $\lip_1^0\times[-\infty,\infty)$.

	\begin{prop}[\textsc{Converging Subsequences}]
		\label{prop:subseq}
		Assume that $t_n\in\mathbb R$ and $\lambda_n>0$ for all $n\in\mathbb N$ such that $\lambda_n\to 0$. Then, the following two statements are equivalent:
		\begin{enumerate}[label=(\roman*)]
			\item \label{prop:subseq:X} The shifted Poisson point process $\iota_{t_n}(X^{(\lambda_n)})$ converges weakly as $\lambda\downarrow 0$ to a point process on $\lip_1$ with respect to the topology of $\mathcal M$,  and its limit is nonempty with positive probability.
			\item \label{prop:subseq:mu} The sequence of measures $\lambda_n \iota_{t_n}^*\mu$‌ converges in $\mathcal M$ and its limit is nonzero.
		\end{enumerate}
		If these conditions hold, then:
		\begin{enumerate}[label=(\roman*)]
			\setcounter{enumi}{2}
			\item \label{prop:subseq:corona} $t_n\to\infty$ and the measure $\nu\coloneqq\lim_n \lambda_n \iota_{t_n}^*\mu$ is supported on $\corona{E}$.
			\item \label{prop:subseq:pp} The weak limit of $\iota_{t_n} X^{(\lambda_n)}$ is a (marked) Poisson point process on $\corona{E}$ with intensity measure $\nu$, namely $\Phi\coloneqq(\Theta_i,\Delta_i)_{i \geq 1}$.
			\item \label{prop:subseq:orderedpoints} \textsc{[Convergence of Ordered Points]} Assuming that either $E$ is a graph (with the graph-distance metric) or that the delays $(\Delta_i)_{i\geq 1}$ are distinct a.s., if the points $(\Theta_i,\Delta_i)_{i\geq 1}$ are ordered with the same rule as in \Cref{def:PPP}, then the joint distribution of the ordered sequence of points $\left(\iota_{t_n}X^{(\lambda_n)}_i\right)_{i\geq 1}$ converges to that of $(\Theta_i,\Delta_i)_{i\geq 1}$ as $n\to\infty$.
		\end{enumerate}
	\end{prop}
	
	\begin{proof}
		The equivalence of~\ref{prop:subseq:X} and~\ref{prop:subseq:mu} is implied by the fact that $\iota_{t_n}X^{(\lambda_n)}$ is a Poisson point process with intensity measure $\mu_n \coloneqq \lambda_n \iota_{t_n}^*\mu$. Assume that these conditions hold. 
		
		\ref{prop:subseq:corona}. If $t_n\not\to\infty$, one might pass to a subsequence and assume $t_n\leq t, \forall n$. For all $t'\in\mathbb R$, one has $\mu_n\left(\lip_1^{\leq t'}\right) = \lambda_n F(t_n+t')\leq \lambda_n F(t+t')$. This implies that $\nu=0$, which is a contradiction. 
		One also has $\mu_n(B_r(\origin)\times \mathbb R) = \lambda_n \mu(B_r(\origin))\to 0$ for every $r>0$. This implies that $\nu$ is supported on $\corona{E}$.
		
		\ref{prop:subseq:pp}. Since $\mu_n$ converges weakly to $\nu$, one obtains that $\mathrm{Poiss(\mu_n)}$ also converges weakly to $\mathrm{Poiss}(\nu)$, where $\mathrm{Poiss(\cdot)}$ denotes a Poisson point process with intensity measure $(\cdot)$. 
		
		\ref{prop:subseq:orderedpoints}.  The claim is implied by the fact that the marks of the points are distinct a.s.~and the ordering procedure is continuous on the event where the delays are distinct (note that, for the last continuity to hold, it is important that we have considered the vague topology on $\mathcal M'$ to prevent points from escaping to $-\infty$). If $E$ is a graph, the argument is similar and there is no need for having distinct delays (for the mentioned continuity, it is enough that the marks are distinct).
	\end{proof}

    \begin{prop}[\textsc{Convergence of Voronoi Diagrams}] 
		\label{prop:diagramConvergence}
		In \Cref{prop:subseq}, 
		\begin{enumerate}[label=(\roman*)]
			\item\label{prop:diagramConvergence:unordered} If either $E$ is a graph (with the graph-distance metric) or $\Vor(\Phi)$ is non-degenerate a.s., then $\Vor(X^{(\lambda_n)})$ converges weakly to $\Vor(\Phi)$. Moreover, the distribution of $(\iota_{t_n} X^{(\lambda_n)}, \Vor(X^{(\lambda_n)}))$ converges weakly to that of $(\Phi,\Vor(\Phi))$ as point processes of pointed closed subsets. 
			\item\label{prop:diagramConvergence:ordered} If in addition, either $E$ is a graph or the delays in $\Phi$ are distinct a.s., then the last convergence holds for the ordered versions as well; i.e., holds in $\left(\widehat{\partial E} \times \mathcal F(E)\right)^{\mathbb N}$.
			\item\label{prop:diagramConvergence:full} Also, similar claims hold assuming full convergence as $\lambda\downarrow 0$.
		\end{enumerate}
	\end{prop}
	Note that the joint distribution of $(\Phi,\Vor(\Phi))$
	is considered in this Proposition since the nuclei are not necessarily determined by the tessellation; e.g., in the trivial cases like the Canopy tree which have empty cell in the ordered version of the IPVT (see \Cref{ex:canopy}).
	
	\begin{proof}
		\ref{prop:diagramConvergence:ordered}.
		By Part~\ref{prop:subseq:orderedpoints} of \Cref{prop:subseq} and Skorokhod's representation theorem, there exists a coupling of $X^{(\lambda_1)}, X^{(\lambda_2)},\ldots$ and $(\Theta_i,\Delta_i)_i$ such that $\forall i: \iota_{t_n} X_i^{\lambda_n}\to (\Theta_i,\Delta_i)$ a.s. If $\Vor(\Phi)$ is nondegenerate a.s., the claim then follows by \cite[Theorem~2.3]{IPVT}. If $E$ is a graph, the claim is implied by the fact that the proof of~\cite[Theorem~2.3]{IPVT} does not require non-degeneracy in the case of graphs.

		\ref{prop:diagramConvergence:unordered}. 
		By Skorokhod's representation theorem, there exists a coupling of $X^{(\lambda_1)}, X^{(\lambda_2)},\ldots$ and $\Phi$ such that $\iota_{t_n} X^{\lambda_n}\to \Phi$ a.s. If $E$ is a graph, this implies the convergence of the ordered points, and the claim is implied by~\ref{prop:diagramConvergence:ordered}. If $\Vor(\Phi)$ is nondegenerate a.s., the claim is obtained by a slight modification of \cite[Theorem~2.3]{IPVT}. For this, it is enough to restrict attention to the first $N$ nuclei (given any $N$), and note that, for large enough $n$, there exists some permutation $\pi=\pi_{n,N}$ of $\{1,\ldots,N\}$ such that $\forall i\leq N: \lim_n \iota_{t_n}X^{\lambda_n}_{\pi(i)}=\Phi_i$.
		
		\ref{prop:diagramConvergence:full}. If the claim is false, then by precompactness, one can find a subsequence $\lambda_n\downarrow 0$ such that $(\iota_{t_n} X^{(\lambda_n)}, \Vor(X^{(\lambda_n)}))$ converges to something different from $(\Phi,\Vor(\Phi))$. This contradicts the previous parts.
	\end{proof}
	
	For the next proposition, we choose the natural shift mentioned in \Cref{subsec:nuclei}:
	\begin{equation}
		\label{eq:t}
		\mu_{\lambda}\coloneqq \lambda \iota_{t(\lambda)}^*\mu, \quad t(\lambda)\coloneqq  \inf\{r: \lambda \cdot F(r)\geq 1\}.
	\end{equation}
	This choice ensures that: 
	\begin{equation}
		\label{eq:F}
		\begin{array}{rcl}
			\mu_{\lambda}(\lip_1^{<0})=& \lambda F^-(t(\lambda))&\leq 1,\\
			\mu_{\lambda}(\lip_1^{\leq 0})=& \lambda F(t(\lambda))&\geq 1,
		\end{array}
	\end{equation}
	where $F^-$ denotes the left limit of $F$.
	Note also that $t(\lambda)$ is non-increasing and right-continuous in $\lambda$. Also, $t(F(r)^{-1})=r$ except if $F(r-\epsilon)=F(r)$‌ for some $\epsilon>0$.
	
	\begin{prop}[\textsc{Pre-Compactness}]
		\label{prop:precompact}
		Assume that  there exists a real number $ r_0>0$‌ such that $p_1\coloneqq \liminf_r F(r+r_0)/F(r)>1$ and $p_2\coloneqq \limsup_r F(r+r_0)/F(r)<\infty$. Then: 
		\begin{enumerate}[label=(\roman*)]
			\item \label{prop:precompact:precompact} If $t(\lambda)$ is defined by~\eqref{eq:t}, then the set of measures $\mu_{\lambda}\coloneqq \lambda\iota_{t(\lambda)}^*\mu \in\mathcal M$ is relatively compact.
			\item \label{prop:precompact:ineq} Every subsequential limit $\nu$ of $\mu_{\lambda}$, as $\lambda\to 0$, satisfies
			\begin{eqnarray*}
				p_1^{j}\leq \nu\left(\lip_1^{\leq jr_0}\right), &&\nu\left(\lip_1^{<jr_0}\right) \leq p_2^j,\\
				p_2^{-j}\leq \nu\left(\lip_1^{\leq -jr_0}\right), && \nu\left(\lip_1^{<-jr_0}\right) \leq p_1^{-j}.
			\end{eqnarray*}
			for all non negative $j\in \mathbb Z^{\geq 0}$. In particular, if $p\coloneqq \lim_r F(r+r_0)/F(r)$ exists, then 
			\begin{equation*}
				\forall j\in\mathbb Z: p^j\leq \nu\left(\lip_1^{\leq jr_0}\right)\leq p^{j+1}.
			\end{equation*}
			\item \label{prop:precompact:tessellation} The joint distributions of $(\iota_{t(\lambda)}X^{(\lambda)}, \Vor(X^{(\lambda)}))$, $\lambda>0$, are precompact,
			and hence, there exists an IPVT.
		\end{enumerate}
	\end{prop}

    It can be seen that the growth condition mentioned in \Cref{prop:precompact} holds for all non-amenable groups endowed with the Haar measure. 	
	\begin{remark}\label{rem:lb}
		Items \ref{prop:precompact:precompact} and \ref{prop:precompact:ineq} generalize~\cite[Proposition 3.3]{MiMe23} and are proved similarly. In fact, in \cite{MiMe23} (preprint version), the lower bound $p_1$ is not assumed. But this is clearly necessary; e.g., in $\mathbb R^n$. This issue has been fixed in the published version of~\cite[Proposition 3.3]{MiMe23} (personal communication with Mikolaj Fraczyk). Also, the specific topology of $\mathcal M$ (compactification at delay $-\infty$) does not seem to be considered in~\cite{MiMe23}. 
	\end{remark}

	\begin{remark}
		In item ~\ref{prop:precompact:tessellation}, if $((\Theta_i,\Delta_i)_{i\geq 1}, (C'_i)_{i\geq 1})$ is a subsequential weak limit of $(\iota_{t(\lambda)}X^{(\lambda)}, \Vor(X^{(\lambda)}))$, then $(C'_i)_{i\geq 1}$ is not necessarily equal to $\Vor((\Theta_i,\Delta_i)_{i\geq 1})\eqqcolon(C_i)_{i\geq 1}$. In the nondegenerate case, the equality is guaranteed in \Cref{prop:diagramConvergence}. In general, one can prove that $C_i^{\circ}\subseteq C'_i\subseteq C_i$.
	\end{remark}

    \begin{proof}[Proof of \Cref{prop:precompact}]
		One has 
		\begin{eqnarray*}
			\mu_{\lambda}(\lip_1^{<t}) = &\lambda F^-(t+t(\lambda))& \leq F^-(t+t(\lambda))/F^-(t(\lambda)),\\
			\mu_{\lambda}(\lip_1^{\leq t}) = &\lambda F(t+t(\lambda))& \geq F(t+t(\lambda))/F(t(\lambda)).
		\end{eqnarray*}
		Fix $\epsilon>0$.
		By letting $t=jr_0$ and $t=-jr_0$, one obtains for small enough $\lambda$ that
		\begin{eqnarray}
			\label{eq:prop:precompact:0}	
			(p_1-\epsilon)^{j}\leq \mu_{\lambda}(\lip_1^{\leq jr_0}), &&\mu_{\lambda}(\lip_1^{<jr_0}) \leq (p_2+\epsilon)^j,\\
			\nonumber(p_2+\epsilon)^{-j}\leq \mu_{\lambda}(\lip_1^{\leq -jr_0}), && \mu_{\lambda}(\lip_1^{<-jr_0}) \leq (p_1-\epsilon)^{-j}.
		\end{eqnarray}
		Recall that the topology on $\mathcal M$ is induced from the vague topology on $\mathcal M'$.
		Since $\lip_1^0\times [-\infty,t)$‌ is open in $\lip_1^0\times [-\infty,\infty)$, one can pass to the limit and deduce that $\nu(\lip_1^0\times [-\infty,jr_0)) \leq (p_2+\epsilon)^j$ and $\nu(\lip_1^0\times [-\infty,-jr_0)) \leq (p_1-\epsilon)^{-j}$. Since $\epsilon>0$ is arbitrary, one obtains
		\begin{equation}
			\label{eq:prop:precompact:1}
			\nu(\lip_1^0\times [-\infty,jr_0)) \leq p_2^j,\quad \nu(\lip_1^0\times [-\infty,-jr_0)) \leq p_1^{-j}.
		\end{equation}
		Similarly, since $\lip_1^0\times [-\infty,t]$ is closed, one obtains 
		$\nu(\lip_1^0\times [-\infty, jr_0]) \geq p_1^j$ and $\nu(\lip_1^0\times [-\infty, -jr_0]) \geq p_2^{-j}$. In addition, \eqref{eq:prop:precompact:1} implies that $\nu(\lip_1^0\times \{-\infty\})=0$. Hence,
		\begin{equation*}
			\label{eq:prop:precompact:2}
			\nu(\lip_1^{\leq jr_0}) \geq p_1^j,\quad \nu(\lip_1^{\leq -jr_0}) \geq p_2^{-j}.
		\end{equation*}
		In particular, $\nu$ is nonzero. 
		Now, by letting $\epsilon\downarrow 0$, \ref{prop:precompact:ineq} is proved.  
		
		For~\ref{prop:precompact:precompact}, note that the compact set  $\lip_1^0\times [-\infty, jr_0]$ has $\mu_{\lambda}$-measure at most $p_2^{j+1}$ uniformly in $\lambda$ (when $j\geq 0$). Since these sets exhaust $\lip_1^0\times [-\infty,\infty)$, the precompactness in $\mathcal M'$ is implied. By~\eqref{eq:prop:precompact:1}, every subsequential limit assigns zero measure to $\lip_1^0\times\{-\infty\}$; i.e., is in $\mathcal M$. This implies precompactness in $\mathcal M$‌ as well and~\ref{prop:precompact:precompact} is proved.

		\ref{prop:precompact:tessellation}. 
		Note that nondegeneracy is not assumed, and so, \Cref{prop:diagramConvergence} is not available.  Let $\Delta_i^{(\lambda_n)}\coloneqq d(X_i^{(\lambda_n)},\origin)-t(\lambda_n)$.  
		Let $A_{j,k}$ be the event that at least $k$ cells intersect $B_{jr_0}^\circ(\origin)$, where $B_{jr_0}^\circ(\cdot)$ denotes the open ball of radius $jr_0$. Since $A_{j,k}$ is open, for arbitrary $l\geq 0$, one has
		\begin{eqnarray*}
			\myprob{T\in A_{j,k}}&\leq&\liminf_n \myprob{\Vor(X^{(\lambda_n)})\in A_{j,k}}\\
			&\leq& \liminf_n \myprob{\Delta_k^{(\lambda_n)}<\Delta_0^{(\lambda_n)}+jr_0}\\
			&\leq& \liminf_n \myprob{\Delta_0^{(\lambda_n)}> lr_0} + \myprob{\Delta_k^{(\lambda_n)}<(j+l)r_0}\\
			&=& \liminf_n \myprob{\mathrm{Poiss}(\mu_{\lambda_n}(\lip_1^{\leq lr_0}))=0} + \myprob{\mathrm{Poiss}(\mu_{\lambda_n}(\lip_1^{<(j+k)r_0}))\geq k}.
		\end{eqnarray*} 
		By choosing $l=\ln\ln k$, one obtains from~\eqref{eq:prop:precompact:0} that $\lim_{k \to \infty} \myprob{T\in A_{j,k}}=0$, and hence, $\myprob{A_{j,\infty}}=0$. This implies the desired precompactness.
		\end{proof}

        According to the above results, the most interesting cases are the following:
	\begin{defn}[\textsc{Gentle spaces}]
		\label{def:generic}
		We say that $(E,\origin, \mu)$ is a \dfn{gentle space} if it satisfies the growth conditions of \Cref{prop:precompact} and, either $E$ is a graph, or $\lim_{r_0\to 0}\limsup_{r\to\infty} F(r+r_0)/F(r)=1$.
	\end{defn}
	In fact, more general than graphs, it is enough to assume that the points of $E$ have integer distances.
	\begin{cor}
		\label{cor:generic}
		Under the conditions of \Cref{def:generic}, the distribution of $(\iota_{t(\lambda)}X^{(\lambda)}, \Vor(X^{(\lambda)}))$ is precompact for both ordered and unordered versions (discussed in \Cref{subsubsec:ipvt}). 
		Moreover, if a subsequence $(\iota_{t(\lambda_n)}X^{(\lambda_n)}, \Vor(X^{(\lambda_n)}))$ converges in one sense, say to $(\Phi,T)$, then it converges to $(\Phi,T)$ in the other sense as well. If, in addition, either $E$ is a graph or $\Phi$ is nondegenerate a.s., then $T=\Vor(\Phi)$ a.s.
	\end{cor}
	\begin{proof}
		Precompactness of the unordered version was proved in \Cref{prop:precompact}. The assumptions of \Cref{def:generic} imply that, in any subsequential limit $(\Phi,T)$, the delays of $\Phi$ are distinct a.s., unless $E$ is a graph. Similarly to \Cref{prop:diagramConvergence}, one can deduce that convergence to $(\Phi,T)$ holds in the ordered version as well (nondegeneracy is not needed at this stage). Since any subsequence contains a convergent subsequence in the unordered sense, this implies precompactness of the ordered version as well, and that the subsequential limits are the same in both senses. The last claim is also proved in \Cref{prop:diagramConvergence}.
	\end{proof}

    \begin{remark}[\textsc{\textsc{Trivial Cases}}]
		If  $\lim_r F(r+r_0)/F(r)$ is equal to 1 or $\infty$, then no subsequence of $\lambda\iota^*_{t(\lambda)}$ can converge to a nonzero limit for any choice of $t(\lambda)$. More specifically:
		\begin{enumerate}[label=(\roman*)]
			\item If $\lim_r F(r+r_0)/F(r)=1$ (e.g., the Euclidean case), then the IPVT exists and has a single cell that covers the whole $E$.
			\item If $\limsup_r F(r+r_0)/F(r)=\infty$, then $\Vor(X^{(\lambda)})$ is not tight. Also, if $\lim_r F(r+r_0)/F(r)=\infty$, then $\Vor(X^{(\lambda)})$ has no converging subsequence as $\lambda \downarrow 0$.
		\end{enumerate}
		The other cases might have other behaviors; e.g., for graphs when $r_0<1$ (see \Cref{sec:graphs}).
	\end{remark}
	
	\subsection{Full Convergence as $\lambda\downarrow 0$}
	\label{subsec:fullConvergence}
	
	In this subsection, we prove \Cref{thm.scfc}, which provides a criterion for full-convergence of the shifted Poisson point process $\iota_{t(\lambda)}X^{(\lambda)}$ as $\lambda\downarrow 0$ (not along an unknown subsequence). This criterion builds on (a logarithmic version of) the theory of regularly varying functions (see \emph{e.g.}~\cite{bookFoKoZa11heavytail}).

    \begin{defn}[\textsc{Log-Regularly-Varying Functions}]
		\label{def:reg}
		A function $f:\mathbb N\to\mathbb R^+$ is \textbf{log-regularly-varying} if 
		$\lim_{n\to\infty} f(n+1)/f(n)$ exists and is positive. This clearly implies that $\forall r_0\in\mathbb Z: \lim_{n\to\infty} f(n+r_0)/f(n) = e^{br_0}$ for some $b\in\mathbb R$, which is called the \textbf{index} of $f$. We also use $p\coloneqq e^b$ in the formulas.\\
		Similarly, a measurable function $f:\mathbb R\to\mathbb R^+$ is called \textbf{log-regularly-varying} (at $+\infty$) if $f\circ\log$ is regularly-varying; i.e., for every $r_0\in\mathbb R$, the limit $\lim_{r\to\infty} f(r+r_0)/f(r)$ exists and is positive. 
		By well-known results about regularly-varying functions,  the limit is equal to $e^{b r_0}$‌ for some $b\in\mathbb R$, which is called the \textbf{index} of $f$. Hence, $f(r)=l(r) e^{br}$, where $l$ is a \textbf{long-tail} function (see \cite[Definition~2.14]{bookFoKoZa11heavytail}); i.e., $\lim_{r\to\infty} l(r+r_0)/l(r)=1, \forall r_0$.
		
	\end{defn}

    \begin{lem}
		\label{lem:criterion1}
		The following statements are equivalent:
		\begin{enumerate}[label=(\roman*)]
			\item \label{thm:criterion:X} There exists a function $t(\lambda)$ such that the shifted Poisson point process $\iota_{t(\lambda)}X^{(\lambda)}$ converges weakly to a nonempty point process on $\lip_1$ (with respect to the topology of $\mathcal M$) as $\lambda\downarrow 0$.
			\item \label{thm:criterion:mu} There exists a function $t(\lambda)$ such that $\mu_{\lambda}\coloneqq \lambda \iota_{t(\lambda)}^*\mu$‌ converges in $\mathcal M$ and its limit is nonzero.
			\item \label{thm:criterion:F} $F(r)$‌ is log-regularly-varying with index $b>0$ and the probability measure $F(r)^{-1} \restrict{\mu}{B_r(\origin)}$ (regarded as a probability measure on $\bar E$) converges weakly to a probability measure, namely $\nu_0$, on $\partial E$ as $r\to\infty$.
		\end{enumerate}
	\end{lem}

	\begin{proof}
		\ref{thm:criterion:X}$\Leftrightarrow$\ref{thm:criterion:mu}. Similarly to \Cref{prop:subseq}, since $\iota_{t(\lambda)} X^{(\lambda)}$ is a marked Poisson point process with intensity measure $\mu_{\lambda}$, the equivalence of~\ref{thm:criterion:X} and~\ref{thm:criterion:mu} is immediate. Note that the limit of $X^{(\lambda)}$‌ is nonempty with positive probability. Also, \eqref{eq:nu} (proved later) will imply that $\nu$ has total mass $\infty$, and hence, the limiting point process is nonempty a.s.
		
		\ref{thm:criterion:F}$\Rightarrow$\ref{thm:criterion:mu}. 
		We first show that, for every $r_0>0$, the normalized restriction of $\mu$ to the annulus $B_{r+r_0}(\origin)\setminus B_r(\origin)$ converges to $\nu_0$; i.e., 
		\begin{equation}
			\label{eq:annulus}
			(F(r+r_0)-F(r))^{-1} \restrict{\mu}{B_{r+r_0}(\origin)\setminus B_r(\origin)} \Rightarrow \nu_0.
		\end{equation}
		Let $\alpha_r$‌ and $\alpha'_r$ be the normalized restrictions of $\mu$ to $B_r(\origin)$ and $B_{r+r_0}(\origin)\setminus B_r(\origin)$. One has
		\[
		\alpha'_r = \frac{F(r+r_0)}{F(r+r_0)-F(r)} \alpha_{r+r_0} - \frac{F(r)}{F(r+r_0)-F(r)} \alpha_r.
		\]
		By assumption, $\alpha_r$‌ and $\alpha_{r+r_0}$ converge to $\nu_0$ as $r\to\infty$. Also, $\lim_r F(r+r_0)/F(r) = g(r_0)\coloneqq  e^{b r_0}$ for all $r_0\in\mathbb R$. Thus, the above equation implies that $\alpha'_r$‌ converges to $\nu_0$ and~\eqref{eq:annulus} is proved.
		
		Also, $\lim_{\epsilon\downarrow 0} \lim_r F(r-\epsilon)/F(r) = \lim_{\epsilon\downarrow 0} e^{-b \epsilon} = 1$. This implies that the jumps of $F$‌ converge to zero as $r\to \infty$. So, if $t(\lambda)$‌ is defined by~\eqref{eq:t}, then~\eqref{eq:F} implies that
		\begin{equation}
			\label{eq:criterion:1}
			\lim_{\lambda \downarrow 0}\lambda F(t(\lambda)) = 1.
		\end{equation}
		Let $\nu\coloneqq \nu_0\otimes \beta$ as in~\eqref{eq:nu}. To prove~\ref{thm:criterion:mu} and~\eqref{eq:nu}, one should prove $\nu=\lim_{\lambda} \mu_{\lambda}$. 
		For this, it is enough to show that for every $t\in\mathbb R$, the restriction of $\mu_{\lambda}$ to $\lip_1^0\times [-\infty,t]$ converges weakly to that of $\nu$. By Theorem~2.3 of~\cite{billingsley2013convergence} (see also Theorem~2.8 of~\cite{billingsley2013convergence}), it is enough to prove that
		\begin{equation}
			\label{eq:criterion:2}
			\lim_{\lambda \downarrow 0} \mu_{\lambda}(A\times [a_1,a_2]) = \nu_0(A) \beta([a_1,a_2])
		\end{equation}
		for every $\nu_0$-continuity set $A\subseteq \partial E$ and every $-\infty\leq a_1\leq a_2<\infty$. If $a_1>-\infty$, letting $r_0\coloneqq a_2-a_1$, one has
		\begin{eqnarray*}
			\lim_{\lambda \downarrow 0} \mu_{\lambda}(A\times (a_1,a_2]) &=& \lim_{\lambda \downarrow 0} \lambda\left(F(t(\lambda)+a_2)-F(t(\lambda)+a_1) \right) \alpha'_{t(\lambda)+a_1}(A)\\
			&=& \lim_{\lambda \downarrow 0} \frac{F(t(\lambda)+a_2)-F(t(\lambda)+a_1)}{F(t(\lambda))} \alpha'_{t(\lambda)+a_1}(A)\\
			&=& (g(a_2)-g(a_1))\nu_0(A) = \nu(A),
		\end{eqnarray*}
		where the third {equality holds} by~\eqref{eq:annulus}. In addition, since the jumps of $F$ converge to zero, one has $\lim_{\lambda \downarrow 0} \mu_{\lambda}(\lip_1^0\times \{a_1\}) = 0$. So, \eqref{eq:criterion:2} is proved when $a_1>-\infty$.
		
		Monotonicity of $\mu_{\lambda}(A\times [a_1,a_2])$ in $a_1$ implies that $\liminf_{\lambda \downarrow 0} \mu_{\lambda}(A\times [-\infty,a_2])\geq \nu(A\times [-\infty,a_2])$. For the other side, given $\epsilon>0$, choose $a_1$‌ small enough such that $g(a_1)\leq \epsilon/2$. One has $\mu_{\lambda}(\lip_1^0\times [-\infty,a_1]) = \lambda F(t(\lambda)+a_1) \to g(a_1)$. So, $\mu_{\lambda}(\lip_1^0\times [-\infty,a_1])<\epsilon$‌ for small enough $\lambda$. Therefore,
		\[
		\limsup_{\lambda \downarrow 0} \mu_{\lambda}(A\times [-\infty ,a_2]) \leq \epsilon + \limsup_{\lambda \downarrow 0} \mu_{\lambda}(A\times [a_1,a_2]) \leq \epsilon + \nu(A\times [-\infty,a_2]),
		\]
		where the last inequality is by the case $a_1>-\infty$ of \eqref{eq:criterion:2}. The claim is then implied by letting $\epsilon\downarrow 0$. Note that~\eqref{eq:nu} is also proved.
		
		\ref{thm:criterion:mu}$\Rightarrow$\ref{thm:criterion:F}.
		Assume that~\ref{thm:criterion:mu} holds and let $\nu\coloneqq\lim_{\lambda \downarrow 0} \mu_{\lambda}$. Let $g(t)\coloneqq\nu(\lip_1^{\leq t})$, which is a non-decreasing function. Note that, since $\lip_1^{\leq t}$ has compact closure in $\lip_1^0\times [-\infty,\infty)$, one has $g(\cdot)<\infty$. 
		Also, by the definition of convergence in $\mathcal M$, one has $\lim_{t\to-\infty} g(t) = 0$.
		We prove~\ref{thm:criterion:F} in the following steps. First, note that
		\begin{equation}
			\label{eq:criterion:g}
			\begin{array}{rcl}
				\limsup_{\lambda \downarrow 0}\mu_{\lambda}(\lip_1^{\leq t})&\leq& g(t),\\
				\liminf_{\lambda \downarrow 0}\mu_{\lambda}(\lip_1^{< t})&\geq& g^-(t).
			\end{array}
		\end{equation}
		This implies that, if $g(r_0)\neq 0$‌ and $g(r'_0)\neq 0$, then
		\begin{eqnarray*}
				\limsup_{\lambda \downarrow 0} \frac {F(t(\lambda)+r_0)}{F(t(\lambda)+r'_0)} &\leq & \frac{g(r_0)}{g^-(r'_0)},\\
				\liminf_{\lambda \downarrow 0} \frac {F(t(\lambda)+r_0)}{F(t(\lambda)+r'_0)} &\geq & \frac{g^-(r_0)}{g(r'_0)}.
		\end{eqnarray*}

		In addition, if $t$ is a continuity point of $g$, then $\lip_1^{\leq t}$ is a continuity set of $\nu$, and hence, the above inequalities are in fact equality. 
		
		There are three remaining issues for deducing that $F$‌ is log-regularly-varying: The zeros of $g$, the jumps of $g$‌ and the jumps of $t(\lambda)$. The next five claims deal with these issues by contradiction.
		
		\begin{claim}
			\label{claim:nonzero}
			$g(t)>0, \forall t$.
		\end{claim} 
		\begin{proof} 
			If the claim is false, there exists $R>0$ such that $g(-R)=0$‌ and $\delta\coloneqq g^-(R)>0$. So, \eqref{eq:criterion:g} implies that $\lambda \mu(B^{\circ}_{t(\lambda)+R})>\delta/2$ for small enough $\lambda$ and $\lim_{\lambda} \lambda \mu(B_{t(\lambda)-R})=0$. In particular, given $k\in\mathbb N$, one has $(\lambda/k) \mu(B_{t(\lambda/k)-R})<\delta/(2k)$ for small enough $\lambda$. These facts imply that $t(\lambda/k)-R\leq t(\lambda)+R$ for small enough $\lambda$. So,
			\[
			\frac{\delta}2<\frac{\lambda}k\mu\left(B^{\circ}_{t(\frac{\lambda}k)+R}\right)\leq \frac{\lambda}k \mu\left(B^{\circ}_{t(\lambda)+2R}\right)
			\]
			for small enough $\lambda$. By~\eqref{eq:criterion:g}, the limsup of the right hand side is at most $g(2R)/k$. This is a contradiction if $k>2g(2R)/\delta$ from the beginning. 
		\end{proof}
		
		\begin{claim}
			The jumps of $t(\lambda)$ are bounded; i.e., $\sup_{\lambda} \limsup_{\epsilon\to 0} \norm{t(\lambda+\epsilon)-t(\lambda)} <\infty$.	
		\end{claim}
		\begin{proof}
			Assume the claim is false. So, there exist $\lambda_i$ and $\lambda'_i$ converging to 0 such that $\lambda'_i/\lambda_i\to 1$ and such that $r_i\coloneqq t(\lambda_i)-t(\lambda'_i)\to\infty$. Fix $0<\epsilon<g^-(0)$, which is possible by~Claim~\ref{claim:nonzero}. By~\eqref{eq:criterion:g}, one has $\mu_{\lambda'_i}(\lip_1^{<0})>g^-(0)-\epsilon$ for large enough $i$. So, given arbitrary $R<\infty$, one has for large enough $i$
			\[
			\mu_{\lambda_i}(\lip_1^{\leq -R})\geq \mu_{\lambda_i}(\lip_1^{<-r_i})=\frac{\lambda_i}{\lambda'_i}\mu_{\lambda'_i}(\lip_1^{<0}) >\frac{\lambda_i}{\lambda'_i}(g^-(0)-\epsilon).
			\]
			Using~\eqref{eq:criterion:g} again, one obtains $g(-R)\geq g^-(0)-\epsilon$. This contradicts the fact $\lim_{t\to-\infty}g(t)=0$.
		\end{proof}
		
		\begin{claim}
			\label{claim:jumps0}
			The jumps of $t(\lambda)$ converge to zero as $\lambda\downarrow 0$.	
		\end{claim}
		\begin{proof}
			Assume the claim is false. So, claim 2 implies that there exists $\lambda_i$‌ and $\lambda'_i$ that converge to 0 such that $\lambda'_i/\lambda_i\to 1$ and $r_i\coloneqq t(\lambda_i)-t(\lambda'_i)\to R$ for some $R>0$. Consider the shift on $\lip_1^0\times [-\infty,\infty)$ defined by $\sigma_s(f,r)\coloneqq (f,r-s)$. 
			The definition of $\mu_{\lambda}$‌ gives $\mu_{\lambda_i}={\lambda_i}/{\lambda'_i}\sigma_{r_i}^* \mu_{\lambda'_i}$. By letting $i\uparrow\infty$, one obtains that $\nu=\sigma_{r_i}^*\nu$. It is straightforward to see that this is impossible unless $\nu=0$. So, the claim is proved.		
		\end{proof}

        \begin{claim}
			$F$ is log-regularly-varying with positive index.
		\end{claim}
		\begin{proof}
			Fix two real numbers $r_0$ and $r'_0$‌, and $n\in\mathbb N$.
			Claim~\ref{claim:jumps0} implies that the closed intervals $[t(\lambda)-1/(2n),t(\lambda)+1/(2n)]$ cover all sufficiently large numbers. Therefore, 
			\begin{eqnarray*}
				\limsup_{r\to\infty} \frac{F(r+r_0)}{F(r+r'_0)} &=&
				\limsup_{\lambda\to 0} \sup_s\left\{ \frac{F(t(\lambda)+s+r_0)}{F(t(\lambda)+s+r'_0)}: -\frac 1{2n}\leq s\leq \frac 1{2n}\right\}\\
				&\leq& \frac{g(r_0+\frac 1{2n})}{g^-(r_0-\frac 1{2n})},
			\end{eqnarray*}
			where the last inequality holds by~\eqref{eq:criterion:g}. Similarly, one can show that 
			\begin{eqnarray*}
				\liminf_{r\to\infty} \frac{F(r+r_0)}{F(r+r'_0)} 
				\geq \frac{g^-(r_0-\frac 1{2n})}{g(r_0+\frac 1{2n})}.
			\end{eqnarray*}
			Assume that $r_0$‌ and $r'_0$ are continuity points of $g$. Since $n$ is arbitrary, the above inequalities imply that $\lim_{r\to\infty} {F(r+r_0)}/{F(r+r'_0)} = g(r_0)/g(r'_0)$.
			Equivalently, for $s\coloneqq r_0-r'_0$, one has $\lim_r F(r+s)/F(r)$ exists and is positive. Since this holds for all but countably many $s\in \mathbb R$, it holds for all $s$ by standard results about regularly-varying functions (see \emph{e.g.}~\cite{BGT}). So, $F$ is log-regularly-varying. In particular, $g(r)=g(0)e^{b r}$ for some $b$. Since $\lim_{t\to-\infty} g(t) = 0$, one has $b>0$ and the claim is proved.	
		\end{proof}
		As a corollary, one has $F(r) = e^{b r} L(r)$, where $L$‌ is a long-tailed function. Therefore, $t(\lambda) = \log_b(\lambda)-\log_b(L(\log_b(\lambda)))+o(1)$.
		
		\begin{claim}
			The limit $\nu_0\coloneqq \lim_r F(r)^{-1} \restrict{\mu}{B_r(o)}$ exists and is supported on $\partial E$.
		\end{claim}
		\begin{proof}
			The previous claim shows that $\nu(\lip_1^0\times \{t\})=0, \forall t$. This implies that 
			$\lambda F(t(\lambda)) = \mu_{\lambda}(\lip_1^0\times [-\infty,0]) \to g(0)$ as $\lambda\to 0$. It also implies that
			$\mu'_{\lambda}\coloneqq \restrict{(\mu_{\lambda})}{\lip_1^0\times (-\infty,0]}$ converges to $\nu'\coloneqq \restrict{\nu}{\lip_1^0\times (-\infty,0]}$. Therefore, if $\pi:\lip_1^0\times [-\infty,\infty)$ is the projection onto the first coordinate, $\pi_* \mu'_{\lambda}$ converges to $\pi_*\nu'\coloneqq \nu_0$. The claim then follows from the facts $\pi_*\mu'_{\lambda}= \lambda \restrict{\mu}{B_{t(\lambda)}(\origin)}$ and $\lambda F(t(\lambda))\to g(0)$.
		\end{proof}
		Now, \ref{thm:criterion:F} is proved and the proof of \Cref{lem:criterion1} is completed.
	\end{proof}

\begin{lem}
		\label{lem:criterion2}
		If the conditions of \Cref{lem:criterion1} hold, then:
		\begin{enumerate}[label=(\roman*)]
			\item \label{thm:criterion:unique} $\nu\coloneqq  \lim_{\lambda\to 0}\mu_{\lambda}$ does not depend on the choice of $t(\lambda)$ up to multiplication by constants.
			\item \label{thm:criterion:nu} One necessarily has $t(\lambda) = \inf\{r: \lambda F(r)\geq 1\} + a + o(1)$ for some $a\in\mathbb R$. In this case, $\nu$ exists and 
			\begin{equation}
				\label{eq:nu}
				\nu= \lim_{\lambda\downarrow 0} \mu_{\lambda} = \nu_0\otimes (e^{ab}\beta),
			\end{equation}
			where $\beta$ is the measure on $\mathbb R$ defined by $\beta((-\infty,r])\coloneqq e^{b r}$. 
		\end{enumerate}
	\end{lem}
	\begin{proof}
		The claim is implicitly shown in the proof of \Cref{lem:criterion1}. Note that $\lim_{\lambda \downarrow 0} F(t(\lambda))=g(0)$‌ and let $a=\ln(g(0))/b$.
	\end{proof}

	We are now ready to prove \Cref{thm.scfc}.
	\begin{proof}[Proof of \Cref{thm.scfc}]
		\Cref{lem:criterion1} proves that $X^{(\lambda)}$ converges (after a suitable shift) if and only if \ref{thm.scfc-i} and~\ref{thm.scfc-ii} hold. If these conditions hold, then \Cref{prop:subseq} shows that the limit is a Poisson point process on $\widehat{\partial E}$ with intensity measure $\nu$. Then, \eqref{eq:nu} implies that $(\Theta_i)_{i\geq 1}$ are i.i.d. points with distribution $\nu_0$ and $(\Delta_i)_{i\geq 1}$ is a Poisson point process on $\mathbb R$ with intensity measure $e^{ab}\beta$. Finally, if the nondegeneracy condition holds, then the convergence of the Voronoi diagrams is implied by \Cref{prop:diagramConvergence}.
	\end{proof}

    \begin{example}[\textsc{Real Hyperbolic Space of Dimension $d\geq 2$}]
		In $\mathbb H^d$, the volume function $F(r)= \frac{2\pi^{d/2}}{\Gamma(d/2)}\int_0^r (\sinh \rho)^{d-1}d\rho$ is log-regularly-varying with index $b=(d-1)$, and the uniform probability measure on $B_R(\origin)$ converges to the harmonic measure on $\partial \mathbb H^d$. The nondegeneracy condition is also easy to verify. So, \Cref{thm.scfc} implies full convergence towards IPVT, thus providing an alternative proof to~\cite[Theorem~1.1]{IPVT}.
	\end{example}

    \section{On the IPVT of General Graphs}
	\label{sec:graphs}

    In this section, we extend and prove convergence criteria for vertex-measured and edge-measured graphs (\Cref{thm.ctfg,thm.ctfemg}). We also extend and prove the independence on $\xi$ mentioned in \Cref{thm.ioxi}.

    \subsection{Periodic Convergence}
	\label{subsec:periodic}

    In some examples, like graphs, $F$ cannot be log-regularly-varying and does not satisfy \Cref{thm.scfc}. As mentioned in the paragraph before \Cref{thm.ctfg},  in such cases, we are interested in considering only integer shifts of the delays, and we study the existence of the limit $\iota_{r_0n}(X^{(\lambda_n)})$ (note that if $\lim_{n \to \infty} \iota_{a+r_0n}(X^{(\lambda_n)})$ exists for some $a$, then it exists for all $a$ by continuity of the shift). 
	In \Cref{thm:critArithmetic} below, we obtain the decomposition \eqref{eq:thm:critArithmetic:sum} of the limiting measure, which is a weaker version of~\eqref{eq:nu}. {This implies that the IPVT depends on a natural parameter $\xi$ in a periodic way, which justifies the title of this subsection.}
	\Cref{thm:critArithmetic:converse} is a converse to \Cref{thm:critArithmetic} and provides a sufficient condition for convergence. These results are mainly used for graphs and edge-measured graphs in the next subsections (as another example, the direct product of a graph and a continuum space may also satisfy periodic convergence; e.g., $\mathbb Z\times\mathbb R$, see \Cref{subsec:product}).
	
	First, we need the following preparatory result, which extends \Cref{prop:precompact}.

	\begin{prop}
		\label{prop:precompact2}
		Assume $0\neq\nu=\lim_{n \to \infty} \lambda_n\iota_{t_n}^* \mu$, where $t_n\to\infty$.
		Assume the growth conditions of \Cref{prop:precompact} and that, for some $a\in\mathbb R$, one has $\forall j\in\mathbb Z: \nu(\lip_1^{a_j})=0$, where $a_j\coloneqq a+jr_0$. Then,
		for all $j\in\mathbb N$, 
		\begin{eqnarray}
			\label{eq:prop:precompact2-2}
			(p_1-1)p_1^{j-1}\leq &{\nu\left(\lip_1^{[a_{j-1},a_j]}\right)}/{\nu\left(\lip_1^{\leq a}\right)}&\leq (p_2-1)p_2^{j-1},\\
			\label{eq:prop:precompact2-1}
			p_1^{j}\leq &\nu\left(\lip_1^{\leq a_j}\right)/ \nu\left(\lip_1^{\leq a}\right)& \leq p_2^j.
		\end{eqnarray}
		In particular, if $p_1=p_2=p$, then $\forall j:\nu\left(\lip_1^{[a_{j-1},a_j]}\right) = (p-1)p^{j-1+\xi}$, where $p^{\xi}= \nu(\lip_1^{\leq a})$.
	\end{prop}
	\begin{proof}
		Let $\mu_n\coloneqq\lambda_n\iota_{t_n}^* \mu$.
		The assumption $\nu(\lip_1^{a_{j}})=0$ implies that $\nu(\lip_1^{\leq a_j}) = \lim_n \mu_n(\lip_1^{\leq a_j})=\lim_n \lambda_n F(t_n+a_j)$. 
		By dividing both sides by $\mu_n(\lip_1^{\leq a})$, one readily obtains~\eqref{eq:prop:precompact2-1}. Also, 
		\[
		\mu_n(\lip_1^{(a_j,a_{j+1}]})=\lambda_n(F(t_n+a_{j+1})-F(t_n+a_j)) = \frac{F(t_n+a_{j+1})-F(t_n+a_j)}{F(t_n+a_j)} \mu_n(\lip_1^{\leq a_j}).
		\]
		By letting $n\to\infty$, we get
		\[
		(p_1-1) \nu(\lip_1^{\leq a_j})\leq \nu(\lip_1^{[a_j,a_{j+1}]})\leq (p_2-1) \nu(\lip_1^{\leq a_j}).
		\]
		This implies~\eqref{eq:prop:precompact2-2} by induction on $j$. In addition, \eqref{eq:prop:precompact2-1} is obtained by summing~\eqref{eq:prop:precompact2-2} over $j$.
	\end{proof}

    In the next result, $\delta_t$ denotes the Dirac measure on $t$ (not to be confused with the delays).
	
	\begin{theorem}[\textsc{Periodic Convergence}]
		\label{thm:critArithmetic}
		Let $r >0$ and the sequence $(\lambda_n)_{n \geq 1}$ be given. 
		If the limit $\nu\coloneqq  \lim_{n \to \infty} \lambda_n \iota_{nr}^*\mu$ exists and is nonzero, then:
		\begin{enumerate}[label=(\roman*)]
			\item \label{thm:critArithmetic:exists} There exist $p>1$, $\xi\in\mathbb R$, a probability measure $\varphi$ on $(0,r]$ and a family of probability measures $\nu_t$ on $\partial E$ for $0<t\leq r$ such that the map $t\mapsto \nu_t$ is measurable (i.e., $\nu_{(\cdot)}$ is a Markov kernel from $(0,r]$ to $\partial E$) and 
			\begin{eqnarray}
				\label{eq:thm:critArithmetic:sum}
				\nu&=& \sum_{j\in\mathbb Z} (p-1)p^{j+\xi} \int_0^r \left(\nu_t\otimes\delta_{jr+t}\right) \varphi(dt),\\
				\label{eq:thm:critArithmetic:periodic}
				\iota_{ir}^* \nu &=& p^i \nu,\quad \forall i\in\mathbb Z,\\
				\label{eq:thm:critArithmetic:limsup}
				\limsup_{n \to \infty} \lambda_n F(nr) &\leq& \nu\left(\lip_1^{\leq 0}\right) = p^\xi,\\
				\label{eq:thm:critArithmetic:liminf}
				\liminf_{n \to \infty} \lambda_n F^-(nr) &\geq & \nu\left(\lip_1^{<0}\right) = p^\xi\left(1-\frac{p-1}{p} \varphi(\{r\})\right).
			\end{eqnarray}
			Also, $p, \xi$ and $\varphi$ are uniquely determined and $\nu_{(\cdot)}$ is unique up to modification on $\varphi$-null sets.
			\item \label{thm:critArithmetic:pp} The shifted Poisson point process $\iota_{nr} X^{(\lambda_n)}$ converges weakly as $n\to \infty$ to a Poisson point process $(\Theta_i,\Delta_i)_{i\geq 1}$ on $\corona{E}$ with intensity measure $\nu$. In other words, the delay process $(\Delta_i)_{i\geq 1}$ is a Poisson point process on $\mathbb R$ with intensity measure $\beta^{(\xi)}(A)\coloneqq \sum_{j\in \mathbb Z} (p-1)p^{\xi+j} \varphi(A-jr)$ for $A\subseteq \mathbb R$. In addition, conditionally on $(\Delta_i)_{i\geq 1}$, the points $(\Theta_i)_{i\geq 1}$ are independent points on $\partial E$ such that $\Theta_i$ has distribution $\nu_{\Delta'_i}$, where $\Delta'_i$ is the unique number in $(0,r]$ such that $\Delta_i-\Delta'_i\in r\mathbb Z$.
			\item \label{thm:critArithmetic:reg} For all except countably many $a\in\mathbb R$, the function $n\mapsto F(a+nr)$‌ is log-regularly-varying with index $b\coloneqq \ln p$ and, if $\mu_n$ is the normalized restriction of $\mu$ to  $B_{a+(n+1)r}(\origin )\setminus B_{a+ n r}(\origin)$, then $\iota_{n r}^*\mu_n$ converges weakly to the normalized restriction of $\nu$ to ${\lip_1^{(a,a+r]}}$.
		\end{enumerate}
	\end{theorem}

    \begin{proof}
		{Item} \ref{thm:critArithmetic:pp} is implied by~\eqref{eq:thm:critArithmetic:sum} and Part~\ref{prop:subseq:pp} of \Cref{prop:subseq}.
		We shall then prove \ref{thm:critArithmetic:exists} and~\ref{thm:critArithmetic:reg}. Let $g(t)\coloneqq \nu\left(\lip_1^{\leq t}\right)$. For all except countably many $a\in\mathbb R$, all of the points $a_j\coloneqq a+j r$, $j\in\mathbb Z$ are continuity points of $g$. Fix such an $a$. This implies that, for all $j\in\mathbb Z$,
		\begin{eqnarray}
			\label{eq:thm:criterArithmetic:F}
			\lambda_n F(a_{n+j}) = (\iota_{nr}^* \mu) \left(\lip_1^{\leq a_j} \right) &\underset{n\to \infty}{\to}& \nu \left(\lip_1^{\leq a_j} \right),\\
			\label{eq:thm:criterArithmetic:restrict}
			\lambda_n \restrict{\left(\iota_{nr}^* \mu \right)}{\lip_1^{\leq a_j}} &\underset{n\to \infty}{\Rightarrow}& \restrict{\nu}{\lip_1^{\leq a_j}}.
		\end{eqnarray}
		So, $\lambda_n F(a_{n+1+j})$ converges to $\nu \left(\lip_1^{\leq a_{j+1}} \right)$. On the other hand, $\lambda_n F(a_{n+1+j}) = \frac{\lambda_n}{\lambda_{n+1}}\lambda_{n+1} F(a_{n+1+j})$. By considering the limits of the two sides of this equation, one obtains
		\begin{equation*}
			\label{eq:thm:criterArithmetic:frac}
			\lim_{n\to \infty} \frac{\lambda_n}{\lambda_{n+1}} = \nu \left(\lip_1^{\leq a_{j+1}} \right)/\nu \left(\lip_1^{\leq a_{j}} \right),
		\end{equation*}
		assuming that at least one of $\nu \left(\lip_1^{\leq a_{j}} \right)$ or $\nu \left(\lip_1^{\leq a_{j+1}} \right)$ is nonzero. This is indeed the case for large enough $j$, and hence, $\lim_{n\to \infty} \frac{\lambda_n}{\lambda_{n+1}}$ exists and is positive. This in turn implies that $\forall j: \nu \left(\lip_1^{\leq a_{j}} \right)>0$ (otherwise, choose the largest $j$ such that it is zero, and obtain a contradiction). 
		Thus, \eqref{eq:thm:criterArithmetic:F} implies that $\lim_{n\to \infty}  F(a_{n+j+1})/F(a_{n+j}) = \nu \left(\lip_1^{\leq a_{j+1}} \right) / \nu \left(\lip_1^{\leq a_{j}} \right)>0$ for all $j$. Therefore, $n\mapsto F(a_n)$ is a log-regularly-varying function. So, there exists $p\geq 1$ such that $\lim_{n\to \infty}  F(a_{n+1})/F(a_n) = p$. Thus, $\forall j: \nu \left(\lip_1^{\leq a_{j+1}} \right) / \nu \left(\lip_1^{\leq a_{j}} \right) = p$. Since $\lim_{j\to-\infty} \nu\left(\lip_1^{\leq a_j}\right)=0$, one obtains that $p>1$. 
		
		Let now $\mu'_n$ be the (non-normalized) restriction of $\mu$ to  $B_{a+(n+1)r}(\origin )\setminus B_{a+ n r}(\origin)$. Equation~\eqref{eq:thm:criterArithmetic:restrict} implies that 
		\[
		\forall j: \lambda_n \iota_{nr}^* \mu'_{n+j} \underset{n\to \infty}{\Rightarrow} \restrict{\nu}{\lip_1^{(a_j,a_{j+1}]}}\coloneqq \nu'_j.
		\]
		This proves that $\iota_{nr}^* \mu_{n+j}$ also converges to the normalized version of $\nu'_j$.
		On the other hand, 
		\[\lambda_n \iota_{nr}^* \mu'_{n+j} =\frac{\lambda_n}{\lambda_{n+j}} \iota_{-jr}^* \left(\lambda_{n+j}\iota_{(n+j)r}^* \mu'_{n+j}\right)\underset{n\to \infty}{\Rightarrow} p^j\iota_{-jr}^* \nu'_0.\]
		Therefore, $\forall j: \nu'_j = p^j \iota_{-jr}^* \nu'_0$. This proves~\eqref{eq:thm:critArithmetic:periodic}. Define $\xi$ such that $p^{\xi}=\nu\left(\lip_1^{\leq 0}\right)$. This implies~\eqref{eq:thm:critArithmetic:limsup} and the inequality in~\eqref{eq:thm:critArithmetic:liminf}. It remains to prove~\eqref{eq:thm:critArithmetic:sum} for some $\varphi$ and $\nu_{(\cdot)}$. This is obtained from~\eqref{eq:thm:critArithmetic:periodic} by disintegration as follows.
		
		Let $\nu''_n\coloneqq \restrict{\nu}{\lip_1^{(nr,(n+1)r]}}$. Equation~\eqref{eq:thm:critArithmetic:periodic} implies that $\nu''_n = p^n\iota_{-nr}^* \nu''_0$. In addition, the total mass of $\nu''_0$ is $\nu\left(\lip_1^{\leq r}\right) - \nu\left(\lip_1^{\leq 0}\right) = p^{\xi+1}-p^{\xi}=(p-1)p^{\xi}$. Let $\varphi$ be the probability measure on $(0,r]$ obtained by projecting $\nu''_0$ under the map $(\theta,t)\mapsto t$ and normalizing it. By the disintegration theorem (see \emph{e.g.}~\cite[Theorem 3.4]{kallenberg}) there exists a Markov kernel $\nu_t$, $0<t\leq r$, such that $\nu''_0 = (p-1)p^{\xi}\int_0^r \left(\nu_t\otimes\delta_t\right) \varphi(dt)$. This implies~\eqref{eq:thm:critArithmetic:sum} and the claim is proved.
	\end{proof}

    \begin{prop}
		\label{thm:critArithmetic:converse}
		Conversely to \Cref{thm:critArithmetic}, if for some $a\in\mathbb R$, the function $n\mapsto F(a+nr)$‌ is log-regularly-varying with positive index $b$, $\lambda_n F(a+nr)$ converges to a positive constant when $n\to \infty$ and $\iota_{n r}^*\mu_n$ converges weakly to a probability measure on $\lip_1^{(a,a+r]}$ (in particular, no mass escapes to $\lip_1^a$), then $\lim_n \lambda_n \iota_{nr}^*\mu$ exists and is nonzero.
	\end{prop}
	\begin{proof}
		Let $\nu'_0 \coloneqq \lim_n \mu_n$. Let $a_n\coloneqq a+nr$ and $\mu'_n$ be the (non-normalized) restriction of $\mu$ to  $B_{a_{n+1}}(\origin )\setminus B_{a_n}(\origin)$. One has $\mu'_n = \big(F(a_{n+1})-F(a_n)\big) \mu_n$. By assumption, $F(a_{n+1})/F(a_n)\to p$ for some $p>1$. 
		Also, $\lambda_n F(a_n)\to p^{\xi}$ for some $\xi\in\mathbb R$. Therefore, $\lambda_n \iota_{nr}^*\mu'_n\Rightarrow (p-1)p^{\xi} \nu'_0$. Similarly, for every $j\in\mathbb Z$, the following limit exists:
		\begin{eqnarray*}
			\lim_{n \to \infty} \lambda_n\iota_{nr}^* \mu'_{n+j} &=& (p-1)p^{j+\xi} \lim_{n \to \infty}  \iota_{nr}^* \mu_{n+j}\\
			&=& (p-1)p^{j+\xi} \iota_{-jr}^* \lim_{n \to \infty}  \iota_{(n+j)r}^* \mu_{n+j}\\
			&=& (p-1)p^{j+\xi} \iota_{-jr}^* \nu'_0\coloneqq \nu'_j.
		\end{eqnarray*}
		Note that $\nu'_j$ is a finite measure on $\lip_1^{(jr,(j+1)r]}$ and $\mu=\sum_j \mu'_{n+j}$. Now, by summing over $j$, \Cref{lem:sum} implies that $\restrict{\left(\iota_{nr}^*\mu\right)}{\lip_1^{\leq m r}}$ converges weakly to $\sum_{j\leq m} \nu'_j$ for every $m\in\mathbb Z$. This implies that $\iota_{nr}^*\mu$ converges in $\mathcal M$ to $\sum_{j\in\mathbb Z} \nu'_j$ and the claim is proved.
	\end{proof}

    \subsection{Convergence Results for Vertex-Measured Graphs}
	\label{subsec:graphs}
	
	In this subsection, we prove the full-convergence result for graphs announced in \Cref{thm.ctfg} and obtain a decomposition of the intensity measure on $\widehat{\partial G}$ given in \eqref{eq:nu-graph}. Here, $G$ is the vertex set of a connected locally-finite graph, equipped with the graph-distance metric and a measure $\mu$ on the vertices. We call $(G,\mu)$ a \textbf{vertex-measured graph}. As before, we consider a Poisson point process on $G$ with intensity measure $\lambda\mu$, where $\lambda >0$. Alternatively, one may consider the Bernoulli point process on $G$ with probabilities $x\mapsto \min\{1,\lambda\mu(x)\}$ for $x\in G$. In \Cref{prop:bernoulli}, we will show that the two choices are generally equivalent in the limit.
	
	It is easy to see that, if $\lambda_n\iota^*_{t_n}\mu$ converges, then the fractional part of $t_n$ converges in $\mathbb R/\mathbb Z$. Since applying a constant shift does not affect the tessellation, we may thus assume $t_n\in\mathbb Z$ without loss of generality. The following corollary is a direct application of \Cref{prop:precompact2}:
	\begin{cor}
		\label{cor:precompactGraph}
		In \Cref{prop:precompact2}, if  $E=G$ is a vertex-measured graph, $\nu=\lim_n \lambda_n\iota^*_{t_n}\mu$ and  $\forall n: t_n\in\mathbb Z$
		then $(p_1-1)p_1^{j-1}\leq {\nu(\lip_1^j)}/{\nu(\lip_1^{\leq 0})}\leq (p_2-1)p_2^{j-1}$ for all $j\in\mathbb N$. In particular, if $p_1=p_2=p$, then $\forall j\in\mathbb Z:\ \nu(\lip_1^j) = (p-1)p^{j-1+\xi}$ for some $\xi\in\mathbb R$.
	\end{cor}

    We now proceed to the proof of \Cref{thm.ctfg}. By considering the measure of $\lip_1^{\leq 0}$, one obtains that $\lim_n \lambda_n\iota^*_n\mu$ should exist; i.e., we may assume $t_n=n$.
	
	\begin{lem}
		\label{lem.scfg1}
		Let $(G,\mu)$ be a vertex-measured graph and $\origin\in G$. 
		The following statements are equivalent:
		\begin{enumerate}[label=(\roman*)]
			\item \label{thm:criterion2:mu} There exists a sequence $(\lambda_n)_{n \geq 1}$ s.t.~$\lambda_n \iota_n^*\mu$‌ converges in $\mathcal M$ as $n\to \infty$ and its limit nonzero.
			\item \label{thm:criterion2:F} $\restrict{F}{\mathbb N}$‌ is log-regularly-varying with index $b>0$ and the probability measure $F(n)^{-1} \restrict{\mu}{B_n(\origin)}$ (as a probability measure on $\bar E$) converges weakly as $n\to\infty$ to a probability measure on $\partial E$.
			\item \label{thm:criterion2:Fsphere} $\restrict{F}{\mathbb N}$‌ is log-regularly-varying with index $b>0$ and the normalized restriction of $\mu$ to the sphere  $S_n\coloneqq B_{n}(\origin )\setminus B_{n-1}(\origin)$ converges weakly to a probability measure on $\partial E$ as $n\to\infty$.
		\end{enumerate}
	\end{lem}
	\begin{proof}
		The equivalence of~\ref{thm:criterion2:F} and~\ref{thm:criterion2:Fsphere} can be proved similarly to the arguments in the proof of \ref{thm:criterion:F}$\Rightarrow$\ref{thm:criterion:mu} of \Cref{lem:criterion1}. The rest of the claims are directly implied by \Cref{thm:critArithmetic,thm:critArithmetic:converse}.
	\end{proof}

    \begin{lem}
		\label{lem.scfg2}
		If the conditions of \Cref{lem.scfg1} hold, then the limiting measures in~\ref{thm:criterion2:F} and~\ref{thm:criterion2:Fsphere} are identical, namely $\nu_0$. In addition, letting $p\coloneqq e^b$,  one has $\lim_{n\to \infty} \lambda_{n} F(n)= p^{\xi}$ for some $\xi\in\mathbb R$ and the limiting measure can be decomposed as
		\begin{equation}
			\label{eq:nu-graph}
			\nu^{(\xi)}\coloneqq \lim_{n} \lambda_n \iota_n^*\mu = \nu_0\otimes \beta^{(\xi)},
		\end{equation}
		where $\beta^{(\xi)}$ is the measure on $\mathbb Z$ defined by $\beta^{(\xi)}(i)\coloneqq(p-1)p^{i-1+\xi}$ for $i\in \mathbb Z$. In particular, one has
		\begin{equation*}
			\label{eq:nu-graph2}
			\iota_i^*\nu^{(\xi)} = p^{i} \nu^{(\xi)},\quad \forall i\in\mathbb Z.
		\end{equation*} 
	\end{lem}
	\begin{proof}
		The claim is a direct corollary of \Cref{thm:critArithmetic}. The only additional argument is that, in the space of 1-Lipschitz \underline{integer-valued} functions on $G$, the set $\lip_1^{\leq 0}$ is clopen. So, $\lim_{n \to \infty} \lambda_{n} F(n) = \nu(\lip_1^{\leq 0})$.
	\end{proof}
	
	\begin{proof}[Proof of \Cref{thm.ctfg}]
		By \Cref{prop:subseq}, $\iota_{t_n}X^{(\lambda_n)}$ converges if and only if $\lambda_n\iota^*_{t_n}\mu$ does so. By considering the measure of $\lip_1^{\leq 0}$, one obtains that $t_n$ should be asymptotically equal to $n+c$ for some constant $c$. So, we may assume $t_n=n$. Now, the claims of \Cref{thm.ctfg} are implied by \Cref{lem.scfg1,lem.scfg2,prop:subseq}.
	\end{proof}

    \begin{example}[\textsc{Regular Tree and Canopy Tree}]
		\label{ex:canopy}
		If $G$ is the $(p+1)$-regular tree, then $\nu_0$ is the harmonic measure on $\partial G$ and $b=\ln p$ (see \Cref{subsec:regularTree}). If $G$ is the canopy tree, there are significant differences. Since every tree is CAT(0) and the canopy tree is one-ended, the discussion in \Cref{subsec:cat0} shows that 
		$\partial G$ consists of a singe point, call it $\theta$. Therefore, the uniform measure on the balls converge to the Dirac measure on $\theta$. Hence, the limiting Poisson point process satisfies $\Theta_i^{(\xi)}=\theta, \forall i$ and $(\Delta_i^{(\xi)})_{i\geq 1}$ is a Poisson point process on $\mathbb Z$‌ with intensity measure $e^{b\xi}\beta$. If $\Delta_1^{(\xi)}=\cdots = \Delta_k^{(\xi)}$ are the minimum values of $(\Delta_i)_{i\geq 1}$, then in the Voronoi tessellation, the first $k$ cells are the whole $G$ and the rest of the cells are empty. Roughly speaking, some cells \textit{escape to infinity}.
	\end{example}

    \begin{remark}[\textsc{Dependence on $\xi$}]
		\label{rem:dependenceOnXi}
		In both convergence results of \Cref{thm.scfc,thm.ctfg}, the limiting intensity measure is of the form $\nu^{(\xi)}\coloneqq \nu_0\otimes(p^{\xi}\beta^{(0)})$ (see~\eqref{eq:nu} and~\eqref{eq:nu-graph}). However, there is a fundamental difference: In \Cref{thm.scfc}, all the measures $\nu^{(\xi)}$ are equivalent under shift (in the sense that $\nu^{(\xi)}=\iota^*_t\nu^{(0)}$ for some $t\in\mathbb R$), and hence, they lead to the same IPVT. But in \Cref{thm.ctfg}, $\nu^{(\xi_1)}$ is equivalent under shift to $\nu^{(\xi_2)}$ if and only if $\xi_1-\xi_2\in\mathbb Z$. We guess that the IPVT indeed depends on the fractional part of $\xi$. This has been proved for regular trees in~\cite[Proposition~6.4]{IPVT} in the stronger sense that the fractional part of $\xi$ is a measurable function of the IPVT.
	\end{remark}
	
	\begin{lem}[\textsc{Poisson VS Bernoulli}]
		\label{prop:bernoulli}
		Let $(G,\mu)$ be a vertex-measured graph that satisfies the growth conditions of \Cref{prop:precompact}. For each $\lambda>0$, let $X^{(\lambda)}$ be a Poisson point process with intensity measure $\lambda{\mu}$ and let $Y^{(\lambda)}$ be a Bernoulli point process on $G$ with probabilities $x\mapsto \min\{1,\lambda\mu(x)\}$. Then, $d_P(\Vor(X^{(\lambda)}),\Vor(Y^{(\lambda)})) \to 0$ as $\lambda\to 0$ ($d_{P}$ being the Prokhorov distance), provided that $\lim_{n \to \infty} \mu^2(B_n(\origin))/\mu(B_n(\origin))^2 = 0$, where $\mu^2(\cdot)\coloneqq \left[\sum_{x\in \cdot} \mu(x)^2\right]$. As a result, $\Vor(X^{(\lambda)})$ and $\Vor(Y^{(\lambda)})$ have the same subsequential weak limits.
	\end{lem}
	Note that if $\mu$ is the counting measure on $G$, then the assumption of this lemma always holds (assuming the growth conditions of \Cref{prop:precompact}).
	\begin{proof}
		We construct a coupling of $X^{(\lambda)}$ and $Y^{(\lambda)}$ as follows. Let $(U(x))_{x\in G}$ be i.i.d. uniform random variables in the interval $[0,1]$. For each $x\in G$, let $x\in Y^{(\lambda)}$ if $U(x)> 1- \min\{1,\lambda\mu(x)\}$. Let the multiplicity of $x$ in $X^{(\lambda)}$ be zero if $U(x)<e^{-\lambda \mu(x)}$, be 1 if $e^{-\lambda\mu(x)}\leq U(x)<(1+\lambda\mu(x))e^{-\lambda\mu(x)}$, and be larger than 1 otherwise. It follows that the multiplicities of $x$ in $X^{(\lambda)}$ and $Y^{(\lambda)}$ are different with probability at most $c_1(\lambda \mu(x))^2$, for some constant $c_1$.
		
		Fix $\epsilon>0$ and let $A_{\epsilon}$ be the event that the restrictions of $\Vor(X^{(\lambda)})$ and $\Vor(Y^{(\lambda)})$ to the ball $B_{1/\epsilon}(\origin)$ are identical. We will prove that $\myprob{A_\epsilon}\geq 1-2\epsilon$ for small enough $\lambda$. This implies the claim. For this, it is enough to show that, with probability at least $1-2\epsilon$, the closest points of $X^{(\lambda)}$ and $Y^{(\lambda)}$ to $\origin$ are identical, namely at distance $\Delta_1$, and that $X^{(\lambda)}$ and $Y^{(\lambda)}$ are identical in $B_{\Delta_1+2/\epsilon}(\origin)$.
		
		Let $t_{\epsilon}(\lambda)\coloneqq \min\{n: \lambda F(n)\geq -\ln\epsilon \}$ and $t'_{\epsilon}(\lambda)\coloneqq t_{\epsilon}(\lambda)+2/\epsilon$. Assume for simplicity that the growth conditions of \Cref{prop:precompact} hold with $r_0=1$ (the general case is similar). One gets for small enough $\lambda$ that $\lambda F(t'_{\epsilon}(\lambda))\leq \lambda F(t_{\epsilon}(\lambda)) (p_2+\epsilon)^{2/\epsilon}\leq -(\ln \epsilon) (p_2+\epsilon)^{1+2/\epsilon}$. So, the assumption in Lemma \ref{prop:bernoulli} implies that $\lambda^2\mu^2(B_{t'_{\epsilon}(\lambda)})$ converges to zero as $\lambda\to 0$. Hence, the previous paragraph implies that $X^{(\lambda)}$ and $Y^{{(\lambda)}}$ are identical in $B_{t'_{\epsilon}(\lambda)}(\origin)$ with high probability. Also, one has $\myprob{\Delta_1> t_{\epsilon}(\lambda)}\leq e^{\ln\epsilon} = \epsilon$. This implies that $\myprob{A_{\epsilon}}\geq 1-2\epsilon$ for small enough $\lambda$, and the claim is proved.
	\end{proof}

    \subsection{Convergence Results for Weighted Edge-Measured Graphs}
	\label{subsec:edgeMeasured}

	In this subsection, we prove an extended version of \Cref{thm.ctfemg}, which gives a convergence criterion for \textit{weighted edge-measured graphs} (not necessarily bipartite). 
	Here is the formal definition of the latter. Let $G$ be a given undirected connected locally-finite graph in which every edge $e$ has a \dfn{weight} $w(e)\geq 0$. We allow multiple edges and self-loops here. For each edge $e$, consider a copy $I_e$‌ ofthe unit interval $[0,1]$. Let $E$ be the quotient of the disjoint union $\sqcup_e I_e$ when gluing an endpoint of $I_e$ with an endpoint of $I_{e'}$ if the corresponding vertices of $e$ and $e'$ are the same. 
	One can equip $E$ with the quotient metric. Let also $\mu$ be the measure on $E$‌ such that $\restrict{\mu}{I(e)} = w(e) \mathrm{Leb}_e$ for every edge $e$. Then $(E,\mu)$ is called the \dfn{weighted edge-measured graph} corresponding to $G$.

	To avoid ambiguity, let $B_r(x)\coloneqq B_r(E,x)$ be the ball in $E$ with radius $r$ centered at $x$ and define $B_r(G,x)$ similarly. As before, we let $F(r)\coloneqq \mu(B_r(\origin))$ be the volume function. Note that $F$ is linear on every interval $[n,n+1/2]$ and $[n+1/2,n+1]$, where $n\in \mathbb{N}^{*}$. 

    The above paragraph implies that $F$ cannot be log-regularly-varying, and hence, cannot satisfy the conditions of \Cref{thm.scfc}. Therefore, $\lim_{n \to \infty} \iota_{t(\lambda)}X^{(\lambda)}$ cannot exist for any function $t(\lambda)$. As in \Cref{subsec:graphs}, we are interested in integer shifts only. In this case, the results of periodic convergence in \Cref{subsec:periodic} are applicable, but we are able to say more by understanding the relation of $\partial E$‌ and $\partial G$, which is discussed below.

    \subsubsection{The Relation of $\partial E$ and $\partial G$ for Edge-Measured Graphs}
	\label{subsec:partialEgraphs}
	
	It can be seen that a sequence in  $G$ converges in $\overline G$‌ if and only if it converges in $\overline{E}$. This induces a topological embedding of $\partial G$ into $\partial E$. Assume $(x_ny_n)_{n\geq 1}$ is a sequence of edges such that $\iota(x_n)\to f\in \lip_1^0$ and $\iota(y_n)\to g\in \lip_1^0$. If $d(\origin, y_n)=d(\origin,x_n)\eqqcolon m_n$ (i.e., $x_ny_n$ is a \dfn{same-level edge}), then the image $\iota_{m_n}(x_ny_n)$ of this edge converges to the set of functions of the form $\edgezero{f,g,t}\in\lip_1$, $0\leq t\leq 1$, where 
	\begin{equation*}
		\edgezero{f,g,t}\coloneqq\left[(f+t)\wedge (g+1-t)\right]\in \lip_1^{t\wedge (1-t)} \; .
	\end{equation*}
	If $d(\origin,x_n)=d(\origin,y_n)-1\eqqcolon m_n$ (i.e., $x_ny_n$ is an \dfn{inter-level edge}), then $\iota_{m_n}(x_ny_n)$ converges to the set of functions of the form $\edgeone{f,g,t}$, where
	\begin{equation*}
		\edgeone{f,g,t} \coloneqq \edgezero{f,g+1,t} = \left[(f+t)\wedge (g+1+(1-t))\right]\in\lip_1^t \; .
	\end{equation*}
	\noindent	
	By regarding $\edge{0}{}$ and $\edge{1}{}$ as maps from $(\lip_1^0)^2\times [0,1]$ to $\lip_1$, we have obtained:
	
	\begin{lem}[\textsc{$\partial E$ vs $\partial G$}]
		\label{lem:edge-measured-partialE}
		Let $K_0\subseteq (\partial G)^2$ (resp.~$K_1\subseteq (\partial G)^2$) be the set of $(f,g)\in(\partial G)^2$ such that there is a sequence of same-level edges (resp. inter-level edges) $x_ny_n$ in which $\iota(x_n)\to f$ and $\iota(y_n)\to g$ as $n \to \infty$. Then, $K_0$ and $K_1$ are compact subsets of $(\partial G)^2$ and 
		\begin{equation}
			\label{eq:edge-measured-partialE}
			\partial E = \partial G \cup \pi \left(\edge{0}{} (K_0\times [0,1])\right) \cup \pi\left(\edge{1}{}(K_1\times [0,1])\right),
		\end{equation}
		where $\pi:\lip_1\to \lip_1^0$ is the projection $f\mapsto f-f(\origin)$.
		%
		%
	\end{lem}

    Note that $K_0$ is a subset of $\{(f,g)\in (\partial G)^2: \sup\mynorm{f(\cdot)-g(\cdot)}\leq 1 \}$ and $K_1$ is a subset of $\{(f,g): \sup\mynorm{f(\cdot)-(g(\cdot)+1)}\leq 1 \}$. However, these subsets can be smaller. 
	If $G$ is bipartite, then $K_0$ is empty and if $G$ is a tree, then $K_1$ is the diagonal of $(\partial G)^2$.
	
	We now study the limiting measures on $\partial E$. Let $\mathrm{edges}(x,y)$ denote the set of edges with endpoints $\{x,y\}$.
	Let $\alpha_1^n$ be the measure on $S_{n}(G,\origin)\times S_{n+1}(G,\origin)$ defined by
	\begin{equation}
		\label{eq:inter-level}
		\alpha_1^n(x,y)\coloneqq \sum_{e\in \mathrm{edges}(x,y)} w(e).
	\end{equation}
	Let also $\alpha_0^n$ be the measure on $S_{n}(G,\origin)^2$ defined by
	\begin{equation}
		\label{eq:same-level}
		\alpha_0^n(x,y)\coloneqq
		\begin{cases}
			\frac 12\sum_{e\in \mathrm{edges}(x,y)} w(e), & \text{$x\neq y$},\\
			\sum_{e\in \mathrm{edges}(x,x)} w(e), & \text{$x=y$.}
		\end{cases}
	\end{equation}
	
	By pushing forward under the map $\iota$, these measures induce measures on $(\overline{G})^2$.
	\begin{lem}
		\label{lem:edge-measured-measure}
		Let $E$ be the weighted edge-measured graph corresponding to a graph $G$. Assume $\nu=\lim_n \lambda_n\iota_{t_n}^*\mu$, where $\mathbb Z\ni t_n\to\infty$. 
		Then, by passing to a subsequence if necessary, the measures $\alpha_0\coloneqq \lim_n \lambda_n\iota^*\alpha_0^{t_n}$ and $\alpha_1 \coloneqq \lim_{n \to \infty} \lambda_n\iota^*\alpha_1^{t_n}$ are well defined as measures on $K_1$ and $K_1$ respectively, and 
		\begin{equation}
			\label{eq:edge-measured-decomposition}
			\restrict{\nu}{\lip_1^{[0,1]}} = \edge{0}{}{}^*\left(\alpha_0\otimes \mathrm{Leb}\right) + \edge{1}{}{}^*\left(\alpha_1\otimes \mathrm{Leb}\right).
		\end{equation}
	\end{lem}
	Roughly speaking, $\restrict{\nu}{\lip_1^{[0,1]}}$ is a combination of the Lebesgue measure on the intervals of the form $\{(f,g)\}\times [0,1]$, where $(f,g)$ is an element of $K_0$ or $K_1$. Also, unless in exotic examples, there is no need to pass to a subsequence.

    \begin{proof}
		By pre-compactness and passing to a subsequence, we may assume that $\alpha_0\coloneqq \lim_{n \to \infty} \lambda_n\iota^*\alpha_0^{t_n}$ and $\alpha_1\coloneqq \lim_n \lambda_n\iota^*\alpha_1^{t_n}$ exist.
		Since $\mu(\lip_1^{\leq t})$ is continuous and linear on the intervals $[n/2,(n+1)/2]$, so is $\nu$. Therefore, $\restrict{\nu}{\lip_1^{[0,1]}}$ is the limit of $\lambda_n\restrict{(\iota_{t_n}^*{\mu})}{\lip_1^{[0,1]}}$. The latter is the sum of $\lambda_n\edge{0}{}{}^*\left(\iota^*(\alpha_0^{t_n})\otimes {\mathrm{Leb}}\right)$ and $\lambda_n\edge{1}{}{}^*\left(\iota^*(\alpha_1^{t_n})\otimes {\mathrm{Leb}}\right)$.
		This implies the claim.
	\end{proof}

    \subsubsection{Convergence Criterion for Edge-Measured Graphs}
	
	Using the above descriptions of $\partial E$, one can refine \Cref{thm:critArithmetic,thm:critArithmetic:converse} for edge-measured graphs as follows. \Cref{thm:criterion-edge} below extends part of \Cref{thm.ctfemg} to weighted graphs (not necessarily bipartite). Its converse is studied in  \Cref{lem:criterion-edge-converse}.
	{Recall the measures $\alpha_0^n$ and $\alpha_1^n$ defined in~\eqref{eq:same-level} and~\eqref{eq:inter-level} respectively, and note that $F(n)= \sum_{i=0}^{n-1}(\mynorm{\alpha_0^i} + \mynorm{\alpha_1^{i}})$.}
	
	\begin{theorem}[\textsc{Convergence Criterion for Nuclei in Edge-Measured Graphs}]
		\label{thm:criterion-edge}
		Let $E$ be the weighted edge-measured graph corresponding to a weighted graph $G$, and let $\origin\in G$.
			If the following conditions hold for some sequence $(\lambda_n)_{n \geq 1}$, then the limit measure $\nu \coloneqq \lim_{n \to \infty}\lambda_n \iota_n^*\mu$‌ exists in $\mathcal M$.
			\begin{enumerate}[label=(\roman*)]
				\item \label{thm:criterion-edge:1} $\restrict{F}{\mathbb N}$ is log-regularly varying with index $b>0$, and
				\item \label{thm:criterion-edge:2} the measures $\lambda_n\alpha_0^n$ and $\lambda_n\alpha_1^n$, defined in~\eqref{eq:inter-level} and~\eqref{eq:same-level} (representing same-level and inter-level edges), regarded as measures on $\overline{G}^2$, converge weakly to finite measure $\alpha_0$ and $\alpha_1$ on $(\partial G)^2$ respectively.
			\end{enumerate}
			Moreover, under \ref{thm:criterion-edge:1} and~\ref{thm:criterion-edge:2}, $\nu$ exists and is nonzero if and only if $\lim_{n \to \infty} \lambda_n F(n) = p^{\xi}$ for some $\xi\in\mathbb R$, where $p=e^b$. 
			In this case, if the total mass of $\alpha_i$ is $c_i$ for $i=0,1$, then the claims of \Cref{thm:critArithmetic}  hold for $\varphi \coloneqq \frac{2 c_0+c_1}{c_0+c_1} \restrict{\mathrm{Leb}}{[0,1/2]} + \frac{c_1}{c_0+c_1} \restrict{\mathrm{Leb}}{[1/2,1]}$ and 
			\begin{equation}
				\label{eq:thm:criterion-edge:nu}
				\nu_t \coloneqq
				\begin{cases}
					\frac {2}{2 c_0+c_1} \left(\pi\circ\edge{0}{\cdot,\cdot,t}\right)^*(\alpha_0) + \frac {1}{2 c_0+c_1} \left(\pi\circ\edge{1}{\cdot,\cdot,t}\right)^*(\alpha_1), & 0\leq t\leq 1/2\\
					\frac {1}{c_1} \left(\pi\circ\edge{1}{\cdot,\cdot,t}\right)^*(\alpha_1), & 1/2<t\leq 1.
				\end{cases}
			\end{equation}
			
			%
	\end{theorem}
	\begin{proof}
		The claims follow from \Cref{thm:critArithmetic,thm:critArithmetic:converse,lem:edge-measured-measure}. Note that the factor 2 and the different behavior in the intervals $[0,1/2]$ and $[1/2,1]$ are due to the fact that $\edge{0}{f,g,t} = \edge{0}{g,f,1-t}$, which implies that the image of $\edge{0}{}$ is a subset of $\lip_1^{[0,1/2]}$.
	\end{proof}

    \begin{lem}
		\label{lem:criterion-edge-converse}
		Conversely to \Cref{thm:criterion-edge}, if $\nu \coloneqq \lim_{n \to \infty}\lambda_n \iota_n^*\mu$‌ exists and is nonzero and the measures $\alpha_0$ and $\alpha_1$ in the decomposition~\eqref{eq:edge-measured-decomposition} of $\nu$ are uniquely determined, then the assumptions of \Cref{thm:criterion-edge} hold.
	\end{lem}
	It seems difficult to construct an example where the decomposition~\eqref{eq:edge-measured-decomposition} is not unique. Note that, for uniqueness, it is sufficient that the restriction of $\edgezero{}{}$ to $K_0\times\{t\}$ and the restriction of $\edgeone{}{}$ to $K_1\times\{t\}$ are injective for some $\frac 1 2 <t<1$. We will also show (in the proof of \Cref{thm.ctfemg}) that bipartiteness is sufficient for this uniqueness.
	\begin{proof}
		Assume that $\nu$ exists and is nonzero. \Cref{thm:critArithmetic} implies that $\restrict{F}{\mathbb N}$ is log-regularly-varying and the restriction of $\lambda_n \iota_n^*\mu$ to $\lip_1^{[0,1]}$ converges to $\restrict{\nu}{\lip_1^{[0,1]}}$ as $n\to \infty$. Also, \Cref{lem:edge-measured-measure} gives a decomposition~\eqref{eq:edge-measured-decomposition} of $\nu$. If $(\lambda_n\alpha_0^n,\lambda_n\alpha_1^n)$ does not converge to $(\alpha_0,\alpha_1)$ when $n\to \infty$, then it has a subsequential weak limit other than $(\alpha_0,\alpha_1)$. Now, \Cref{lem:edge-measured-measure} gives another decomposition~\eqref{eq:edge-measured-decomposition} of $\nu$, which contradicts the assumption. So, the claim is proved.
        \end{proof}
	
	\begin{proof}[Proof of \Cref{thm.ctfemg}]
		Similarly to the proof of \Cref{thm.ctfg}, $X^{(\lambda_n)}$ converges (after a suitable shift) when $n\to \infty$ if and only if $\lambda_n\iota_n^*\mu$ converges to a nonzero measure. Assume that~\ref{thm.ctfemg-item1} to~\ref{thm.ctfemg-item3} hold. Bypartiteness implies that $\alpha_0^n=0$ and that $\alpha_1^n$ converges to a multiple of $\alpha$. So, \Cref{thm:criterion-edge} implies that $\lambda_n\iota_n^*\mu$ converges. The rest of the claims are also implied by \Cref{thm:criterion-edge}, noting that $c_0=0$ and $c_1=1$ here. Note that nondegeneracy holds for edge-measured graphs, and hence, the convergence of Voronoi diagrams is implied by \Cref{prop:diagramConvergence}.
		
		For the converse, by \Cref{lem:criterion-edge-converse}, it is enough to show that bipartiteness implies that the decomposition~\eqref{eq:edge-measured-decomposition} is unique. Assume $f,g\in \corona{G}$ are distinct and $\sup \mynorm{f(x)-g(x)}\leq 1$. Bipartiteness implies that $f(x)=f(y)\pm 1$ if $x$ is adjacent to $y$ and $f(x)-g(x)=\pm 1$ for all $x\in G$. In this case, given $t\in[0,1]\setminus\{0,1/2,1\}$, $f$ and $g$ can be recovered from $h\coloneqq (f+t)\wedge (g+1-t)$ because, for $x\in G$, $h(x)= n+t$ if $f(x)=g(x)-1=n$ and $h(x)=n-t$ if $f(x)=g(x)+1 = n$. It is straightforward to deduce that $\alpha_0$ and $\alpha_1$ can be recovered from $\nu$, and the claim is proved.
	\end{proof}

    \subsection{Independence on $\xi$}
	\label{subsec:xi}	
	
	In \Cref{thm:xi} below, we prove an extended version of \Cref{thm.ioxi}, which gives an affirmative answer to the research question stated just above Proposition 6.3 of \cite{IPVT}. {We also provide a second proof of \Cref{thm.ioxi} using a coupling technique. The interested reader might jump directly to this proof.}
	
	Let $E$ be the edge-measured graph corresponding the a graph $G$. More generally than \Cref{subsec:edgeMeasured}, fix a {diffuse} probability measure $\psi$ on $[0,1]$ that is symmetric with respect to $t\mapsto 1-t$, and let the measure on each edge $e$ be $w(e)\psi$. 
	Then, \Cref{thm:criterion-edge} can be extended to obtain a one-parameter family of limiting point processes $\Phi=(\Theta_i,\Delta_i)_{i\geq 1}$. {Note that ${\rm IPVT}_{\xi}(E) = \Vor(\Phi)$ by \Cref{thm:criterion-edge}.} The following theorem extends \Cref{thm.ioxi}.
	
	\begin{theorem}[\textsc{Independence on $\xi$ (Extended Version)}]
		\label{thm:xi}
		Let $E$ be an edge-measured graph corresponding to a graph $G$ and a measure $\psi$, as described above. Let $G^{1/2}$ be the set of vertices and midpoints of the edges of  $G$. If:
		\begin{itemize}
			\item the same assumptions of \Cref{thm:criterion-edge} hold,
			\item $\alpha_0$ is a multiple of $\alpha_1$ and they are both supported on the diagonal of $(\partial G)^2$, and \item the function $t\mapsto \nu(\lip_1^{\leq t})$ has a continuous derivative,
		\end{itemize}
		then the distribution of $\left((\Theta_i)_{i\geq 1}, \restrict{\Vor(\Phi)}{G^{1/2}} \right)$ does not depend on $\xi$. In addition, $\restrict{\Vor(X^{(\lambda)})}{G^{1/2}}$  converges weakly to $\restrict{\Vor(\Phi)}{G^{1/2}}$ as $\lambda\downarrow 0$.
		%
			%
	\end{theorem}
	
	In fact, the proof shows that $t\mapsto \nu(\lip_1^{\leq t})$ has a continuous derivative if and only if the map $t\mapsto \psi([0,t])$ has a continuous derivative and $\mynorm{\alpha_0}\frac d{dt} \psi([0,t])=0$ at $t=1/2$. This justifies why $\alpha_0$ was assumed to vanish in \Cref{thm.ioxi}.

    \begin{proof}
		It is clear that $\edge{0}{f,f,t} = f+(t\wedge (1-t))$ and $\edge{1}{f,f,t} = f+t$. This implies that $\pi(\edge{0}{f,f,t})=\pi(\edge{1}{f,f,t})=f$. 
		Hence, if $\alpha_0$ is a multiple of $\alpha_1$ and both are supported on the diagonal of $(\partial G)^2$, then the measures $\nu_t$ given by \Cref{thm:criterion-edge} do not depend on $t$. Indeed, \eqref{eq:thm:criterion-edge:nu} implies that $\nu_t=\frac {1}{c_1} \pi^*(\alpha_1)$. So, the decomposition~\eqref{eq:thm:critArithmetic:sum} gives 
		\begin{equation}
			\label{eq:thm:xi-beta}
			\nu=\nu_1\otimes \beta^{(\xi)}, \text{ where } \beta^{(\xi)} \coloneqq \sum_{j\in\mathbb Z} (p-1)p^{j+\xi} \int_0^1 \delta_{j+t} \tilde{\varphi}(t) dt,
		\end{equation}
		where $\tilde{\varphi}(t)=(p-1)^{-1}p^{-\xi} \frac d{dt} \nu(\lip_1^{\leq t})$ for $0\leq t\leq 1$. Similarly to \Cref{thm:criterion-edge}, one has 
		$$
		\tilde{\varphi}(t) = \left(\frac {2c_0+c_1}{c_0+c_1} 1_{\{t\leq 1/2\}} + \frac{c_1}{c_0+c_1}1_{\{t>1/2\}}\right) \frac d{dt} \psi([0,t]) \; .
		$$
		Note that this function is continuous and does not depend on $\xi$.
		By this decomposition of $\nu$, $(\Theta_i)_{i\geq 1}$ are i.i.d.~points on $\partial E$ with distribution $\nu_1$ and, independently of $(\Theta_i)_{i\geq 1}$, the delays $(\Delta_i)_{i\geq 1}$ form a Poisson point process on $\mathbb R$ with intensity measure $\beta^{(\xi)}$. Note that $\nu_1$ is supported on $\partial G$, and hence, $\Theta_i\in \partial G$ a.s.
		In addition, a point $x\in E$ is in the $i$'th cell if and only if $\forall j: d_{\Theta_i}(x)+\Delta_i \leq d_{\Theta_j}(x)+\Delta_j$. If $x\in G^{1/2}$, then $d_{\Theta_i}(x)-d_{\Theta_j}(x)\in \mathbb Z$. Therefore, $\restrict{\Vor(\Phi)}{G^{1/2}}$ is determined by $(\Theta_i)_{i\geq 1}$ and $(\floor{\Delta_i-\Delta_j})_{i,j\geq1}$. Thus, the claim is implied by \Cref{lem:xi} below.
		
		The last claim is implied by \Cref{thm:criterion-edge} as follows. Let $T$ have the same distribution as $\restrict{\Vor(\Phi)}{G^{1/2}}$ which does not depend on $\xi$. If the claim is false, then there exists a sequence $\lambda'_n\to 0$ such that $\restrict{\Vor(X^{(\lambda'_n)})}{G^{1/2}}$ does not converge weakly to $T$. By passing to a subsequence, one might assume that $\lim_n \lambda'_n F(t_n) = p^{\xi}$ for some $t_n\in\mathbb Z$ and $\xi\in[0,1]$. So, \Cref{thm:criterion-edge,prop:diagramConvergence} imply that $\iota_{t_n}X^{(\lambda_n)}$ converges to $\Phi^{(\xi)}$ and that $\Vor(X^{(\lambda'_n)})$ converges to $\Vor(\Phi^{(\xi)})$. Since the points of $G^{(1/2)}$ are not on the boundaries of the cells of $\Vor(\Phi^{(\xi)})$ a.s., one can deduce that $\restrict{\Vor(X^{(\lambda'_n)})}{G^{1/2}}$ converges weakly to $\restrict{\Vor(\Phi^{(\xi)})}{G^{1/2}}$, which is a contradiction. So, the claim is proved.
	\end{proof}
	
	We now provide the main ingredient needed to complete the proof of \Cref{thm.ioxi}. 
	The key is that the terms depending on $\xi$ are cancelled in a 
	``magic telescopic sum'' (see~\eqref{eq.mts} below), which explains the independence on $\xi$ under our hypotheses.

    \begin{lem}
		\label{lem:xi}
		Let $\tilde{\varphi}$ be a continuous function such that $\int_0^1 \tilde{\varphi}(s)ds = 1$.
		Let $(\Delta_i)_{i\geq 1}$ be a Poisson point process on $\mathbb R$ with intensity measure $\beta^{(\xi)}$, where the latter is defined in~\eqref{eq:thm:xi-beta}. Then, the distribution of $(\floor{\Delta_i-\Delta_j})_{i,j\geq 1}$ does not depend on $\xi$.
	\end{lem}
	
	\begin{proof}
		We first prove that the distribution of {the difference of the delays w.r.t.~the first point}, that is, $\bar\Delta_1\coloneqq(\floor{\Delta_i-\Delta_1})_{i\geq 1}$, does not depend on $\xi$. The latter is determined by the probability of the events of the form
		\[
		A\coloneqq \left\{\forall 0\leq k < m: \mynorm{(\Delta_i)_{i\geq 1}\cap (\Delta_1+k, \Delta_1+k+1)} = n_k \right\},
		\]
		where $m\in\mathbb N$ and $n_0,\ldots,n_{m-1}\in\mathbb Z^{\geq 0}$ are given integers. By conditioning on $\Delta_1$, one has
		\begin{eqnarray*}
			\myprob{A} &=& \int \myprob{(\Delta_i)_i\cap (-\infty,t)=\emptyset}\times \myprob{\mynorm{(\Delta_i)_i\cap (t,t+dt)}=1}\\
			&&\quad\times \prod_k \myprob{\mynorm{(\Delta_i)_i\cap (t+k,t+k+1)}=n_k}\\
			&=& \int e^{-\beta^{(\xi)}(-\infty,t)} \beta^{(\xi)}(t,t+dt) \prod_{0\leq k<m} e^{-\beta^{(\xi)}(t+k,t+k+1)} \frac{\left(\beta^{(\xi)}(t+k,t+k+1)\right)^{n_k}}{n_k!}\\
			&=& c_1 \int \beta^{(\xi)}(t,t+dt) e^{-\beta^{(\xi)}(-\infty,t+m)}\prod_{0\leq k<m} \left(\beta^{(\xi)}(t+k,t+k+1)\right)^{n_k},
		\end{eqnarray*}
		where $c_1\coloneqq \prod_k 1/{n_k!}$. In the following lines, each time a variable $c_i$ is used, it means a constant that does not depend on $\xi$ (but depends on the other variables).
		If $t=j+s$, where $j\in \mathbb Z$‌ and $0\leq s<1$, then $\beta^{(\xi)}(t,t+dt) = (p-1)p^{j+\xi}\tilde{\varphi}(s)ds$ and $\beta^{(\xi)}(-\infty,t+k) = p^{j+k+\xi}(1+(p-1)H(s))$, where $H(s)\coloneqq \int_0^s \tilde{\varphi}(u)du$. So, $\beta^{(\xi)}(t+k,t+k+1) = (p-1)p^{j+k+\xi}(1+(p-1)H(s))$, and thus,
		\begin{eqnarray*}
			\myprob{A} &=& c_2 \sum_{j\in\mathbb Z}\int_0^1
			p^{j+\xi}
			e^{-p^{j+m+\xi}(1+(p-1)H(s))} 
			\prod_{0\leq k<m} \left(p^{j+k+\xi}(1+(p-1)H(s))\right)^{n_k}\tilde{\varphi}(s)ds\\
			&=& c_3 \sum_{j\in\mathbb Z}\int_0^1
			\left(p^{j+\xi}(1+(p-1)H(s))\right)^n e^{-p^{j+m+\xi}(1+(p-1)H(s))} p^{j+\xi}\tilde{\varphi}(s)ds,
		\end{eqnarray*}
		where $n\coloneqq \sum_k n_k$.  
		Use the change of variables $u\coloneqq p^{j+m+\xi}(1+(p-1)H(s))$, which is valid since $H$ has a continuous derivative by assumption. By $du =(p-1) p^{j+m+\xi} \tilde{\varphi}(s)ds$, one gets
		\begin{eqnarray}\label{eq.mts}
			\myprob{A}&=& c_4 \sum_{j\in\mathbb Z}\int_{p^{j+m+\xi}}^{p^{j+m+1+\xi}} u^n e^{-u} du = c_4\int_{0}^{+\infty} u^ne^{-u} du= c_{4} \cdot (n!).
		\end{eqnarray}
		This clearly does not depend on $\xi$. So, it is proved that the distribution of $\bar\Delta_1$ does not depend on $\xi$.
		
		To complete the proof, it remains to show that the distribution of $\bar\Delta\coloneqq (\floor{\Delta_i-\Delta_j})_{i,j\geq 1}$ conditional on $\bar\Delta_1$ does not depend on $\xi$. First, we study the conditional distribution of $\bar\Delta$ given $\Delta_1$ and $\bar\Delta_1$. Assume $\Delta_1= j+s$, where $j=\floor{\Delta_1}$, and $\forall i: z_i\coloneqq \floor{\Delta_i-\Delta_1}$. 
		Note that, given this information, one has $\forall i: \Delta_i\in [z_i+j+s, z_i+j+1+s)$. By the definition of $\beta^{(\xi)}$, the restriction of $\beta^{(\xi)}$ to this interval is a multiple of the measure $\beta_s$ shifted by $z_i+j$, where $\beta_s$ is the measure on $[s,s+1)$ with density function $\tilde{\varphi}(t) 1_{[s,1]}(t) + p \tilde{\varphi}(t-1)1_{[1,s+1)}(t)$ (note that $\beta_s$ does not depend on $\xi$). Now, construct $s_2,s_3,\ldots\in [s,s+1)$ i.i.d. with distribution $\beta_s$. For each $m\geq 0$, shuffle $\{s_i: z_i=m\}$ by sorting them {in increasing order,} and let $s'_2,s'_3,\ldots,$ be the result. Finally, {define} $\Delta_i\coloneqq \Delta_1+z_i+s'_i$ for every $i\geq 2$. It can be seen that this constructs $\Delta_2,\Delta_3,\ldots,$ with the right distribution conditionally on $\Delta_1$ and $\bar\Delta$. Unfortunately, the distribution of $(s'_2,s'_3,\ldots)$ depends on $s$ and the distribution of $s=\{\Delta_1\}$ depends on $\xi$. However, note that the order of $s'_2,s'_3,\ldots$ (from smallest to largest) determines $\bar\Delta$. Since $(s_i)_{i\geq 1}$ are i.i.d.~with a diffuse distribution, one obtains that, for every $M$, the order of $(s'_i)_{\{i: z_i\leq M\}}$ is uniform among all possible orders (given $\Delta_1$ and $(z_i)_{i\geq 1}$), and this distribution does not depend on the triplet $(j,s,\xi)$. This implies that the distribution of $\bar\Delta$ conditionally on $\bar\Delta_1$ does not depend on $\xi$, and the proof is completed.
	\end{proof}

    \begin{proof}[Proof of \Cref{thm.ioxi} (First Proof)]
		In the setting of \Cref{thm:xi}, if $\mu$ is a multiple of the Lebesgue measure on each edge, then $\psi=\restrict{\mathrm{Leb}}{[0,1]}$. Also, by~\eqref{eq:thm:xi-beta} and \Cref{thm:criterion-edge}, one has $\tilde{\varphi}(t) = \frac{2 c_0+c_1}{c_0+c_1}$ for $0<t<\frac 1 2$ and $\tilde{\varphi}(t)= \frac{c_1}{c_0+c_1}$ for $\frac 1 2 <t<1$. In this case, continuity of $\tilde{\varphi}$ is equivalent to $\alpha_0=0$. Now, the claim is implied by \Cref{thm:xi}.
	\end{proof}

    \begin{proof}[Proof of \Cref{thm.ioxi} (Second Proof)]
		Let $(\Theta^{(\xi)}_i,\Delta^{(\xi)}_i)_{i\geq 1}$ be distributed as ${\rm IPVT}_{\xi}(E)$, where $\Delta^{(\xi)}_1<\Delta^{(\xi)}_2<\cdots$. Since $\alpha_0=0$ and $\alpha_1$ is supported on the diagonal, the decomposition~\eqref{eq:thm:xi-beta} is valid. Thus, $(\Theta^{(\xi)}_i)_{i\geq 1}$ are i.i.d. points on $\partial G$ with distribution $\nu_1$ and, independently, $(\Delta^{(\xi)}_i)_{i\geq 1}$ form a Poisson point process on $\mathbb R$ with intensity measure $\beta^{(\xi)}$. We couple these processes as follows. First, construct $\Theta_i:=\Theta^{(0)}_i$ and $\Delta_i:=\Delta^{(0)}_i$ for all $i\geq 1$. Then, for any $\xi\in\mathbb R$, let $\Theta^{(\xi)}_i:=\Theta_i$ and $\Delta^{(\xi)}_i:=u(\Delta_i,\xi)$, where the function $u$ is defined as follows: $u(\cdot,\xi):=g_{\xi}^{-1}\circ g_0$, where $g_{\xi}(t):=\beta^{(\xi)}((-\infty,t)) = p^{\xi+\lfloor t\rfloor}(1+(p-1)(t-\lfloor t\rfloor))$. {Since $u(\cdot,\xi)^* \beta^{(0)}=\beta^{(\xi)}$}, the Poisson mapping theorem implies that $(\Delta^{(\xi)})_{i\geq 1}$ is indeed a Poisson point process with intensity measure $\beta^{(\xi)}$, and the coupling is constructed. {Since $g_{\xi+1}(t) = g_{\xi}(t+1)$, one obtains that $\forall i: \Delta^{(\xi+1)}_i=\Delta^{(\xi)}_i-1$. Therefore, ${\rm IPVT}_{\xi+1}(E)={\rm IPVT}_{\xi}(E)$ a.s. in this coupling. Also, since $\Delta^{(\xi)}_i$ is continuous in $\xi$, the intersection of two cells moves continuously until hitting a vertex. Below, we will prove that the intersection does not hit any point of $G^{1/2}$ a.s. This implies the continuity of ${\rm IPVT}_{\xi}(E)$.}
		
		We now prove that, in this coupling, the restrictions of ${\rm IPVT}_{\xi}(E)$ and ${\rm IPVT}_0(E)$ to $G^{1/2}$ are identical. 
		Since $g_{\xi}(t+n) = p^n g_{\xi}(t)$ for $n\in\mathbb Z$, one obtains that $u(t+n,\xi) = u(t,\xi)+n$.
		Let $f^{(\xi)}_i:= \Delta^{(\xi)}_i+d_{\Theta_i}(x)$ denote the horofunction corresponding to $(\Theta_i,\Delta^{(\xi)}_i)$. If $x$ is a vertex, then $d_{\Theta_i}(x)$ is an integer (since $\Theta_i\in\partial G$), and hence, $f^{(\xi)}_i(x) = u(f^{(0)}_i(x),\xi)$.
		Fixing $i, j$, and an edge $xy$, consider the condition that, in ${\rm IPVT}_{0}(E)$, the cells of $(\Theta_i,\Delta_i)$ and $(\Theta_j,\Delta_j)$ intersect on the edge $xy$ at a point closer to $x$ than to $y$. Assuming $f^{(0)}_i(x)>f^{(0)}_j(x)$, this condition is equivalent to the occurence of all of the following:  
		$d_{\Theta_i}(y) = d_{\Theta_i}(x)-1$, $d_{\Theta_j}(y) \geq d_{\Theta_j}(x)$ and $f^{(0)}_i(x)<f^{(0)}_j(x)+1$.
		By changing $f^{(0)}_\cdot$ to $f^{(\xi)}_\cdot$, the values of $d_{\Theta_i}$ and $d_{\Theta_j}$ are not changed. Also, one has $f^{(\xi)}_i(x)>f^{(\xi)}_j(x)$ by monotonicity.
		Hence, to prove that the two cells in ${\rm IPVT}_{\xi}(E)$ also intersect on the same edge $xy$,  
		it is enough to prove that the condition $f^{(0)}_i(x)<f^{(0)}_j(x)+1$ is preserved under the mentioned change. This is proved as follows: 
		\begin{eqnarray*}
			f^{(\xi)}_i(x) = u(f^{(0)}_i(x),\xi)< u(f^{(0)}_j(x)+1,\xi) = u(f^{(0)}_j(x),\xi)+1 = f^{(\xi)}_j(x)+1.
		\end{eqnarray*}
		So, the claim is proved. 
	\end{proof}

	\begin{remark}
		\label{rem:off-diagonal}
		Here is the problem with the off-diagonal points of $K_0$ and $K_1$. Assume $f,g\in \corona{G}$ such that $\sup\mynorm{f(x)-g(x)}\leq 1$, $f\neq g$, $f\neq g+1$ and $f\neq g-1$. So, there exist adjacent vertices $x$ and $y$ such that $f(x)-g(x)\neq f(y)-g(y)$. In this case, it can be seen that, for $h_t\coloneqq (f+t)\wedge (g+(1-t))$, the function $t\mapsto h_t(x)-h_t(y)$ is not constant, and hence, $h_t$ depends on $t$. So, if $\alpha_0$ and $\alpha_1$ are not supported on the diagonal, we cannot deduce that $\nu_t$ is independent of $t$ in general, which is essential in the proof of \Cref{thm:xi}.
	\end{remark}

    The next lemma gives an intuitive criterion for proving that $\alpha_1$ is supported on the diagonal. In the statement, by a \dfn{$1\times n$ rectangle}, we mean vertices $(x_1,y_1,x_2,y_2)$ such that $d(x_1,y_1)=d(x_2,y_2)=1$, $d(x_1,x_2)=d(y_1,y_2)=n$ and $d(x_1,y_2)=d(x_2,y_1)=n+1$.
		
		\begin{lem}[\textsc{Thin Rectangles}]
			\label{lem:rectangle}
			In the setting of \Cref{thm:criterion-edge}, assuming $\alpha_0=0$, $\alpha_1$ is supported on the diagonal if and only if for every edge $e$, the number of $1\times n$ rectangles with side $e$, divided by $F(n)$, converges to zero as $n\to\infty$. 
		\end{lem}
		\begin{proof}
			For simplicity of the proof, we prove the claim in the bipartite case.
			For every edge $e=(x,y)$, let $R_n(e)$ be the set of edges $(x',y')$ that form a $1\times n$ rectangles together with $e$, and let $R(e)$ be the closure of $\cup_n R_n(e)$. It can be seen that the set of off-diagonal elements of $K_1$ can be written as $\cup_e R(e)$. So, $\alpha_1$ is supported on the diagonal if and only if $\forall e: \alpha_1(R(e))=0$. Since $R(e)$ is clopen, the latter is equivalent to $\forall e: \lim_n \alpha_1^n(R(e))=0$. This easily implies the claim.
		\end{proof}

        \section{Examples and Applications of the General Results}
	\label{sec:ex}
	
	{
		In this section, we apply the general results of \Cref{sec:full,sec:graphs} to prove convergence towards IPVT in various settings. We also study in \Cref{subsec:regularTree} further properties of the IPVT of regular trees and their products. The case of Diestel--Leader graphs will be studied in \Cref{sec:dl}.
		
		To use the general theorems, usually the more difficult part is proving that the normalized restriction of $\mu$ to a large ball converges. This task is highly model-dependent. For this goal, we will use the following lemma frequently, which helps proving the convergence by partitioning the ball into some \textit{layers} or \textit{regions}:
	}
	
	\begin{lem}[\textsc{{Weak Convergence via Linear Decomposition}}]
		\label{lem:sum}
		\ 
		\begin{enumerate}[label=(\roman*)]
			\item 
			For $n\in\mathbb N$, let $\mu_n$ be a probability measure on a separable metric space $E_1$. Assume $\mu_n = \sum_i c_{n,i}\mu_{n,i}$, where $c_{n,i}\geq 0$ and $\mu_{n,i}$ is a probability measure. Assume also that, for all $i$, $\lim_{n\to \infty} c_{n,i}=c'_i$ and $\lim_n \mu_{n,i} = \mu'_i$ in the weak topology. Then 
			$$
			\lim_{n\to \infty} \mu_n = \mu'\coloneqq \sum_i c'_i\mu'_i \; .
			$$
			\item  \label{lem:sum:cont}  In the previous part, assume $\mu_n = \int k_n(y,\cdot)\alpha_n(dy)$, where $\alpha_n$ is a probability measure on a measurable space $E_2$ and $k_n$ is a Markov kernel from $E_2$ to $E_1$. Assume also that $\tv{\alpha_n-\alpha}\to 0$ and $\lim_n k_n(y,\cdot)=k(y,\cdot)$ (in the topology of weak convergence) for all $y\in E_2$. Then, $$
			\lim_{n\to \infty} \mu_n = \mu'\coloneqq \int k(y,\cdot)\alpha(dy) \; .
			$$
		\end{enumerate}
	\end{lem}

    \begin{proof}
		The first statement is a special case of the second one.
		We use the Prokhorov metric $d_P$ for weak convergence.
		Fix $\epsilon>0$. For Borel sets $A\subseteq E_1$, let $A_{\epsilon}\coloneqq \cup_{x\in A} B_{\epsilon}(x)$. It is enough to show that, for large enough $n$, one has $\forall A: \mu'(A_{\epsilon})\leq \mu_n(A)+\epsilon$. Let $\epsilon_{n}(y)\coloneqq d_P(k_n(y,\cdot),k(y,\cdot))$. One has
		\begin{eqnarray*}
			\mu'(A_{\epsilon}) = \int k(y,A_{\epsilon})\alpha(dy) \leq  \int \left( k_n(y,A) + \epsilon_{n}(y)\right) \alpha(dy).
		\end{eqnarray*}
		Since $\epsilon_n(y)\leq 1$, bounded convergence implies that $\int \epsilon_n(y) \alpha(dy)\to 0$. For the other term, since $k_n(y,A)\leq 1$, one has
		\begin{eqnarray*}
			\int k_n(y,A)\alpha(dy) &\leq & \int k_n(y,A) \alpha_n(y) + \tv{\alpha_n-\alpha} = \mu_n(A) + \tv{\alpha_n-\alpha}.
		\end{eqnarray*}
		This implies the claim.
	\end{proof}

	\begin{remark}
		One can extend \Cref{lem:sum} by allowing $\alpha_n$ converge weakly to $\alpha$, assuming that the kernels $k$ and $k_n$ are continuous in $y$, and that $k_n$ converges uniformly to $k$ (the latter always holds when $E_2$ is compact).
	\end{remark}

	\subsection{CAT(0) Spaces}
	\label{subsec:cat0}
	
	Recall from \Cref{subsec:Busemann} that, in complete CAT(0) spaces, the horoboundary coincides with the visual boundary. We now obtain the following useful criterion for non-degeneracy of Voronoi diagrams. In this result, a metric space is called \dfn{non-branching} if every geodesic has at most a unique prolongation, see \emph{e.g.}~\cite[Definition 2.1]{AR14}.
	
	\begin{lem}[\textsc{Non-degeneracy in CAT(0) Spaces}]
		\label{lem:nondegenerate-cat0}
		Let $E$ be a proper non-branching CAT(0) space (in particular, any CAT(0) Riemannian manifold). If $(f_i)_{i\geq 1}$ is an admissible sequence of distinct horofunctions, then $\Vor((f_i)_{i\geq 1})$ is non-degenerate.
	\end{lem}
	\begin{proof}
		If $x\in C_i$, let $\gamma_i$ be the unique geodesic starting from $x$ such that $f_i(\gamma_i(t)) = f_i(x)-t$. The triangle inequality implies that, if $f_i\neq f_j$, then $f_j(\gamma_i(t))\geq f_j(x)-t\geq f_i(x)-t$. If equality holds for some $t>0$, then the geodesic $\restrict{\gamma}{[0,t]}$ has two different extensions, which violates the assumption. This implies that $\gamma(t)\in C_i^\circ$ for every $t>0$ and the claim is proved.
	\end{proof}

	\begin{cor}[\textsc{Convergence Towards IPVT in CAT(0) Spaces}]
		\label{cor:cat0}
		Let $E$ be a proper non-branching CAT(0) space (in particular, any CAT(0) Riemannian manifold). Under the assumptions of \Cref{prop:subseq}, if $\nu$ is diffuse, then $\Vor(\Phi)$ is non-degenerate a.s. Hence, $\lim_n \Vor(X^{(\lambda_n)})$ exists and is equal to $\Vor(\Phi)$.
	\end{cor}

    \begin{proof}
		If $\nu$ is diffuse, then the points of $\Phi$ are distinct a.s. So, the proof is implied by \Cref{lem:nondegenerate-cat0,prop:diagramConvergence}.
	\end{proof}
	
	\begin{remark} 
		\label{rem:cat0-diffuse}
		Under the assumptions of the full-converge result (\Cref{thm.scfc}), $\nu$ is always diffuse. But otherwise, one can cook up an example where $\nu$ has atoms and the convergence does not hold, for instance, this happens if $E$ is the hyperbolic plane and $\mu$ is a suitable measure on the union of circles of integer radius about $\origin$.
	\end{remark}
	
	\subsection{Symmetric Spaces}
	\label{subsec:symmetric}
	
	We now prove \Cref{prop:symmetric} using the results of \Cref{subsec:cat0}.
	
	\begin{proof}[Proof of \Cref{prop:symmetric}]
		In Theorem~3.6 of~\cite{MiMe23}, it is proved that $\lim_{t\to\infty} (F(t))^{-1}\iota_t^*\mu$ has a unique subsequential limit. So, by \Cref{lem:criterion1}, the conditions of the full convergence theorem (\Cref{thm.scfc}) are satisfied (note also that Lemma~4.4 of~\cite{MiMe23} shows that $F(r)$ is asymptotically of the form $r^a p^r$ for some explicit $a$ and $p$). 
		Hence, as shown in \Cref{rem:cat0-diffuse}, the limiting measure is diffuse. Now, since the space is a CAT(0) Riemannian manifold, the claim is implied by \Cref{cor:cat0}.
	\end{proof}

	\subsection{On the IPVT of Trees}
	\label{subsec:regularTree}
	
	In this subsection, we apply the results of \Cref{sec:graphs} to regular trees and provide new streamlined proofs of some of the results of \cite{bhupatiraju} and of \cite[Section 6]{IPVT}. 
	We prove that cells are one-ended in all unimodular random trees.
	As a simple byproduct, we prove the ``folklore'' result that the IPVT cells have pairwise finite intersection. We also provide a description of the zero cell by an edge-percolation which works for an arbitrary base tree, and which appears to be new. {Note that trees are CAT(0). This implies that the horocompactification is identical to the end compactification (see \Cref{subsec:cat0}).}
	
	Given a natural number $p\geq 2$, denote by $\rtree{p}$ the infinite $(p+1)$-regular tree,\footnote{Note that in the literature, $\rtree{p}$ usually denotes the $p$-regular tree. We use the $(p+1)$-regular tree for having nicer formulas.} 
	endowed with the graph distance $\treedist{p}$. Fix an arbitrary vertex $\origin \in \rtree{p}$ as the origin. 
	For the vertex-measured tree $\rtree{p}$ and $n>0$, one can calculate the volume function as $F(n)= \frac{p+1}{p-1} p^n -\frac{2}{p-1}$. Hence, $\restrict{F}{\mathbb N}$ is log-regularly-varying with index $b=\ln p$. Also, it is a classical fact {(see e.g.~\cite[Theorem 20.3]{woess2000})} that the uniform measure on $B_n(\origin)$ converges to the \textit{harmonic measure} (seen from $\origin$) on $\partial\rtree{p}$, namely $\nu_0$. So, the full-convergence in \Cref{thm.ctfg} holds and one obtains a one-parameter family of IPVTs.
	
	For the edge-measured version of $\rtree{p}$ (which is the one studied in \cite{IPVT,bhupatiraju}), one similarly gets the volume function $F(n)= \frac{p+1}{p-1} p^n -\frac{p+1}{p-1}$, which has the same index $b=\ln p$. 
	In the setting of \Cref{thm.ctfemg}, one can show that $\alpha$ exists and is supported on the diagonal of $(\partial\rtree{p})^2$. More precisely, $\alpha$ is the push-forward of the harmonic measure $\nu_0$ under the diagonal map $\mathrm{diag}(\theta):=(\theta,\theta)$. Therefore the full-convergence in \Cref{thm.ctfemg} holds and one obtains a one-parameter family of IPVTs again. Moreover, the distribution of delays provided in \Cref{thm.ctfemg} gives~\cite[Proposition~6.1]{IPVT}. In addition, the conditions of \Cref{thm:xi} hold, and hence, the independence on $\xi$ mentioned in \Cref{thm:xi} holds. This implies Theorem~6.2 of~\cite{IPVT}. 

    It is also proved in~\cite[Proposition~6.3]{IPVT} that the cells of ${\rm IPVT}(\rtree{p})$ are one-ended. We extend this result to all unimodular random trees by providing a proof based solely on the mass transport principle:
	
	\begin{lem}[\cite{IPVT}]
		\label{lem:one-ended2}
		Let $[\bs G, \bs o]$ be a unimodular random tree that satisfies the growth conditions of \Cref{prop:precompact}. Then, in any IPVT of $[\bs G, \bs o]$, the cells are one-ended trees a.s.
	\end{lem}
	\begin{remark}
		\label{rem:unimodular}
		In \Cref{lem:one-ended2}, one should use the space $\mathcal G_*$ of tuples $[G,o;\varphi]$, where $G$ is a connected locally finite graph, $o\in G$, and $\varphi$ is a discrete set (in the Fell topology) of nonempty (pointed) closed subsets of $G$. Also, isomorphic tuples are considered the same. Using the framework of~\cite{Kh19generalization}, one can define a suitable topology on $\mathcal G_*$, which is skipped for brevity. Then, an IPVT of $[\bs G, \bs o]$ should be defined as a subsequential weak limit of $[\bs G, \bs o; \Vor(X^{(\lambda)})]$, considered as a random element of $\mathcal G_*$. Similarly to \Cref{lem:autInv}, one can show that unimodularity is preserved by adding the IPVT. This can also be extended to \textit{unimodular random measured metric spaces} using the framework of~\cite{Kh23unimodularspaces}.
	\end{remark}
	\begin{proof}[Proof of \Cref{lem:one-ended2}]
		{Let $(C_i)_i$ be an IPVT of $[\bs G, \bs o]$ as described in \Cref{rem:unimodular} above. By passing to a subsequence if necessary, we may assume that the collection of pointed cells also converge. Denote by $\bs T=(f_i, C_i)_i$ the resulting (unordered) collection, where $f_i\in\partial G$ and $C_i$ is the corresponding cell in $\Vor((f_i)_i)$. As mentioned in \Cref{rem:unimodular}, $[\bs G, \bs o; \bs T]$ is unimodular. 
			By a \textit{typical cell} $C$, we mean a random cell that contains $\origin$, biased by $\mynorm{\{i: \origin\in C_i\}}$.}
		By verifying the mass transport principle, one can show that $C$ is a unimodular random tree. Also, it is enough to prove that $C$ is one-ended a.s. 
		One may regard $C$ as a family tree (called and \textit{eternal family tree} in~\cite{eft}) and consider its level-sets. By the general classification of unimodular eternal family trees in~\cite{eft}, if $C$ is one-ended, then all of its level-sets are either infinite or empty. Otherwise, $C$ is two-ended and all of its level-sets are finite and nonempty. So, it is enough to prove that $C$ has a level-set that is infinite or empty a.s.
		
		Let $L_i \coloneqq \{x\in \bs G: f_i(x)=0\}$ and call it a \textit{zero-horosphere}. By the previous paragraph, it is enough to show that $L_i\cap C_i$ is infinite or empty, for all $i$, a.s. Consider a typical zero-horosphere as follows: Let $A\coloneqq \{(f_i,L_i,C_i): \origin\in L_i\}$. Bias the probability measure by $\mynorm{A}$, and then, let $(f',L',C')$ be a random element of $A$. By unimodularity, it is enough to prove that $L'\cap C'$ is infinite or empty a.s. By verifying the mass transport principle, it is straightforward to see that $[L', \origin]$ is unimodular. More generally, unimodularity of $[L', \bs o]$ still holds if one keeps the rest of the structure as a decoration (see \textit{local unimodularity} in \cite{H20}, or unimodular measured graphs in \cite{Kh23unimodularspaces}). By the growth conditions, $[\bs G, \bs o]$ is a unimodular graph with infinitely many ends and no isolated ends (see~\cite{processes}). This implies that $L'$ is infinite a.s. Finally, $L'\cap C'$ is an equivariant subset of $L'$. The claim is then implied by the fact that an infinite unimodular graph cannot have a finite nonempty equivariant subset (see Lemma~2.9 of~\cite{eft}).
	\end{proof}

    We now study the intersection of cells. In the edge-measured version, it is clear that any two adjacent cells intersect in a single non-vertex point (since the delays are distinct) a.s. In the vertex-measured case, note that if the difference of delays $\Delta_i-\Delta_j$ is odd, then the corresponding cells cannot intersect. More generally:
	\begin{cor}
		In any IPVT of the vertex-measured regular tree, for any two cells $C_i$ and $C_j$ and for all $r\geq 0$, $N_r(C_i)\cap C_j$ is finite, where $N_r(C_i)$ is the $r$-neighborhood of $C_i$.
	\end{cor}
	\begin{proof}
		If $C_i\cap C_j=\emptyset$, the claim follows from the fact that $C_i$ and $C_j$ are subtrees of $\rtree{p}$. Otherwise, $C_i\cap C_j$ is also a subtree of $\rtree{p}$. If the latter is infinite, one finds that $C_i$ has more than one end, which is a contradiction. This proves the claim.
	\end{proof}

	In general trees, one can describe the first cell $C_1$ (i.e. the cell of $\origin$) by an edge-percolation as follows (recall that if more than one cell contains $\origin$, then a uniform random order is chosen on those cells).
	
	\begin{prop}[\textsc{Description of the Zero Cell}]
		Let $G$ be any tree and let $\nu$ be any nonzero subsequential limit of $\lambda_n \iota_{t_n}^*\mu$. Consider the corresponding ${\rm IPVT}(G)$ and let $\gamma$ be the unique path from $\origin$ to $\Theta_1$. Then, conditionally on $(\Theta_1,\Delta_1)$, the first cell $C_1$ is the connected component of $\origin$ in an edge percolation on $G$ such that each edge $(x,y)$ is open independently from the other edges (but not with equal probabilities). 
	\end{prop}
	\begin{proof}
		Assume $(\Theta_1,\Delta_1)$ is known.
		If $(x,y)\in \gamma$, then let it be open with probability 1. Otherwise, assume that $d(x,\gamma)=n>0$ and $d(y,\gamma)=n-1$. Let $A(x)$ be the set of $\theta\in\partial G$ such that $x$ separates $\theta$ and $\Theta_1$. Let $B(x)\coloneqq \{(\theta,\delta): \theta\in A(x), \Delta_1+2n-2\leq \delta < \Delta_1+2n\}$. Note that the sets $B(x)$ are all disjoint. Also, $x\in C_1$ if and only if $y\in C_1$ and $B(x)\cap (\Theta_i,\Delta_i)_{i\geq 1}=\emptyset$. 
		Note that, conditionally on $(\Theta_1,\Delta_1)$, the rest of the process; i.e., $(\Theta_i,\Delta_i)_{i\geq 2}$, is a Poisson point process whose intensity measure is the restriction of $\nu$ to ${\lip_1^{\geq \Delta_1}}$.
		Therefore, the claim is implied by letting $(x,y)$ be open with probability $\exp(-\nu(B(x)))$.
	\end{proof}
	
	In the special case of the $(p+1)$-regular infinite tree $\rtree{p}$, {one can use the last result to describe the level sets of $C_1$ as follows, whose proof is left to the reader:}
	\begin{lem}[\textsc{Distinguishable Level-Sets}]
		In ${\rm IPVT}(\rtree{p})$, the level-sets of the first cell $C_1$ are distinguishable. Indeed, the \emph{asymptotic density} of the level-set of $x\in C_1$ is $\exp(-\beta^{(\xi)}(-\infty,f(x)))$, {where $f$ is the horofunction corresponding to $(\Theta_1,\Delta_1)$.}
	\end{lem}

    \subsection{Products}
	\label{subsec:product}
	
	{The study of the IPVT of products is initiated in \cite{MiMe23}, where it is proved that, in the product of two regular trees equipped with the $L^2$ metric, every pair of IPVT cells have nonempty intersection or adjacency. This is used to prove Gaboriau's fixed price problem in the case of products of regular trees. Also, in~\cite{AGKRW25}, a similar property is used to show that the direct product of two regular trees with the same degrees has the following \textit{sparse unique infinite cluster property}: In the Bernoulli percolation on the set of Poisson--Voronoi cells with intensity $\lambda$, the uniqueness threshold converges to 0 as $\lambda\to 0$.
		
		In this section, we study convergence towards IPVT for products of two spaces equipped with the $L^1$ or $L^{\infty}$ metrics. See \Cref{prob:l2} regarding the $L^2$ metric.}	
	Let $(E,\origin,\mu)$ and $(E',\origin',\mu')$ be rooted measured metric spaces as in \Cref{sec:ipvt}. Let $E''\coloneqq E\times E'$, $\origin''\coloneqq (\origin,\origin')$ and $\mu'\coloneqq \mu\otimes \mu'$. 
	We assume that either both of $E$ and $E'$ satisfy the assumptions of \Cref{thm.scfc}, or that both are graphs that satisfy the assumptions of \Cref{thm.ctfg}.
	First, the following lemma describes the horoboundary of $E''$. 
	
	\begin{lem}[\textsc{Boundary of Products}]
		\label{lem:productBoundary}
		Let $E''=E\times E'$ as above.
		\begin{enumerate}[label=(\roman*)]
			\item \label{lem:productBoundary:1} Under the $L^1$ metric, one has $\overline{E''}\equiv \overline{E}\times\overline{E'}$, and hence, $\partial E''\equiv (\partial E\times\partial E')\sqcup (E\times\partial E')\sqcup (\partial E\times E')$. More precisely, $\partial E''$ consists of the following functions: $(x,x')\mapsto d_{x_0}(x)+d_{x'_0}(x')$, where $(x_0,x'_0)\in (\overline{E}\times\overline{E'})\setminus (E\times E')$.
			\item \label{lem:productBoundary:inf} Under the $L^{\infty}$ metric, one can write $\partial E''\equiv (\partial E\times \partial E' \times I) \sqcup \partial E\sqcup \partial E'$, where $I=\mathbb Z$ if both $E$ and $E'$ are graphs, and $I=\mathbb R$ otherwise. More precisely, $\partial E''$ consists of the following functions: $(x,x')\mapsto \max\{d_{\theta}(x), d_{\theta'}(x')+ c\}$, $(x,x')\mapsto d_{\theta}(x)$ and $(x,x')\mapsto d_{\theta'}(x')$, where $\theta\in\partial E$, $\theta'\in\partial E'$ and $c\in I$.
		\end{enumerate}
	\end{lem}
	\begin{proof}
		\ref{lem:productBoundary:1}. If $y_n\to x_0\in \overline{E}$ and $y'_n\to x'_0\in\overline{E'}$, one obtains that $d_{(y_n,y'_n)}$ converges to the function $(x,x')\mapsto d_{x_0}(x)+d_{x'_0}(x')$. This implies the claim.
		
		\ref{lem:productBoundary:inf}. For $(y_n,y'_n)\in E''$, one has $d_{(y_n,y'_n)}(x,x') = \max\{d_{y_n}(x), d_{y'_n}(x')+c_n\} - \max\{0,c_n\}$, where $c_n = d(y'_n,\origin')-d(y_n,\origin)$. This easily implies the claim.
	\end{proof}
	
	We now consider the IPVT in each metric separately. Let $F, F'$ and $F''$ be th volume functions of $E,E'$ and $E''$. Let $b=\ln p$ and $b'=\ln p'$ be the growth indices of $E$ and $E'$ respectively (see \Cref{def:reg}). Let $\nu_0$ and $\nu'_0$ be the limits of the measures on balls, as assumed in \Cref{thm.scfc} or \Cref{thm.ctfg}. Define $\nu''_0$ similarly on $\partial E''$, if it exists.

	\begin{prop}[\textsc{Full Convergence in the $L^{\infty}$ Metric}]
		\label{prop:product-Linf}
		Consider $E''=E\times E'$ and $\mu''=\mu\otimes\mu'$ defined above.
		\begin{enumerate}[label=(\roman*)]
			\item \label{prop:product-Linf-1} Under the $L^{\infty}$ metric, $E''$ always satisfies the full convergence in \Cref{thm.scfc} or \Cref{thm.ctfg} (the latter if $E''$ is a graph). In addition, $F''$ has index $b+b'$. Also, $\nu''_0$ is supported on $\partial E\times \partial E'\times I\subseteq\partial E''$ and is equal to $\nu_0\times\nu'_0\times \gamma$, where $\gamma$ is a probability measure on $I$.
			
			\item \label{prop:product-Linf-2} If we modify the conditions on $E'$ by letting $b'=0$ (still assuming the existence of $\nu'_0$), then the last claim holds, except that $\nu''_0=\nu_0$, which is supported on $\partial E\subseteq\partial E''$.
		\end{enumerate}
	\end{prop}
	\begin{proof}[Proof of \Cref{prop:product-Linf} for Graphs]
		One has $B_n(\origin'')=B_n(\origin)\times B_n(\origin')$. Hence, $F''(n)=F(n) F'(n)$, which implies that $F''$ is log-regularly-varying with index $b+b'$. 
		
		\ref{prop:product-Linf-1}.
		Let $\mu''_n$ be the normalized restriction of $\mu''$ to $S_n(\origin'')=B_n(\origin'')\setminus B_{n-1}(\origin'')$. To use \Cref{lem:sum}, we partition $B_n(\origin'')$ as $\cup_{i=-n}^n T_{n,i}$, where $T_{n,i}:= S_{n-i}(\origin)\times S_{n}(\origin')$ when $i\geq 0$, and $T_{n,i}:=S_{n}(\origin)\times S_{n+i}(\origin')$ when $i\leq 0$. Let $c_{n,i}:=\mu''(T_{n,i})/(F''_n-F''_{n-1})$ and let $\mu''_{n,i}$ be the normalized restriction of $\mu''$ to $T_{n,i}$. So, $\mu''_n = \sum_{i} c_{n,i} \mu''_{n,i}$. For $i\geq 0$, the assumptions imply that 
		\begin{eqnarray*}
			c_{n,i} = &{\mu(S_{n-i}(\origin)) \mu'(S_{n}(\origin'))}/{(F''_n-F''_{n-1})} &\to c p^{-i},\\
			\mu''_{n,i} = &(\restrict{\mu}{S_{n-i}(\origin)})\otimes (\restrict{\mu'}{S_{n}(\origin')}) &\to \nu_0\otimes\nu'_0\otimes \delta_{i},
		\end{eqnarray*}
		as $n\to\infty$, where $c:= {(p-1)(p'-1)}/{(pp'-1)}$. Similarly, for $i<0$, one can obtain that $c''_{n,i}\to c(p')^i$ and $\mu''_{n,i}\to \nu_0\otimes\nu'_0\otimes \delta_{i}$. Thus, \Cref{lem:sum} implies that $\mu''_n$ converges to $\nu_0\otimes\nu'_0\otimes\gamma$, where $\gamma:=\sum_{i\geq 0} cp^{-i}\delta_{i} + \sum_{i<0} c(p')^{i} \delta_{i}$.
		
		\ref{prop:product-Linf-2}. 
		In this case, for most points $(x,x')\in S_n(\origin'')$, the difference $d(x,\origin)-d(x',\origin')$ is very large. This is formalized as follows.
		Let $(m_n)_{n\geq 0}$ be an increasing sequence that will be determined later.
		In the above proof, partition $S_n(\origin'')$ by sets $U_{n,1}:=S_n(\origin)\times B_{n-m_n}(\origin')$ and $U_{n,2}:=S_n(\origin'')\setminus U_{n,1}$. For $i=1,2$, let $C_{n,i}:=\mu''(U_{n,i})$ and let $\mu''_{n,i}$ be the normalized restriction of $\mu''$ to $U_{n,i}$. If $m_n\to\infty$, then $\mu''_{n,1}$ converges to $\nu_0$, where $\nu_0$ is regarded as a measure on $\partial E\subseteq\partial E''$. Also, the calculations in the previous case show that, if $m_n$ converges slow enough to infinity, then $c_{n,1}\to 1$ and $c_{n,2}\to 0$. So, \Cref{lem:sum} implies that $\mu''_n$ converges to $\nu_0$, and the claim is proved.
	\end{proof}
	
	For the non-graph case of \Cref{prop:product-Linf}, if we assume that the \textit{conditioning} of $\mu$ and $\mu'$ to spheres (defined via disintegration) converge to $\nu$ and $\nu'$ respectively, then we can mimic the above proof by using the continuous-case of \Cref{lem:sum}. But this assumption is not always satisfied (e.g., when $E=\mathbb Z^2$ with the Euclidean metric). We need the convergence of the uniform measure on a thin annulus as follows:
	
	\begin{lem}
		\label{lem:annulus}
		If $(E,\origin,\mu)$ satisfies the assumptions of \Cref{thm.scfc}, then there exists a non-increasing function $s=s(r)\geq 0$ such that $\lim_{r \to \infty} s(r) = 0$ and the normalized restriction of $\mu$ to $B_{r+s(r)}(\origin)\setminus B_r(\origin)$ converges to $\nu$ as $r\to\infty$.
	\end{lem}
	\begin{proof}
		Let $\mu_r$ be the normalized restriction of $\mu$ to $B_r(\origin)$, which converges weakly to $\nu$ by assumption. Fixing an arbitrary metrization of $\overline{E}$, denote by $d_{BL}$ the bounded-Lipschitz distance of probability measures on $\overline{E}$. Let $\epsilon_r \coloneqq\sup\{d_{BL}(\mu_{r'},\nu): r'\geq r\}$. So, $\lim_r \epsilon_r = 0$. For a function $s=s(r)$, to be determined later, let $\mu'_r$ be the normalized restriction of $\mu$ to $B_{r+s(r)}(\origin)\setminus B_r(\origin)$, and let $\Delta F(r)\coloneqq F(r+s(r))-F(r)$. For any 1-bounded 1-Lipschitz function $g$ on $\overline{E}$, one has
		\begin{eqnarray*}
			\int g d\mu'_r - \int g d\nu &=& \frac{F(r+s(r))}{\Delta F(r)}\left(\int g d\mu_{r+s(r)}-\int g d\nu \right) - \frac{F(r)}{\Delta F(r)} \left(\int g d\mu_{r}-\int g d\nu \right)\\
			&\leq& \epsilon_r\frac{F(r+s(r))+F(r)}{\Delta F(r)}.
		\end{eqnarray*}
		Choose the function $s(r)$ such that $F(r)/{\Delta F(r)}$ converges to infinity but slower than $\epsilon_r^{-1}$. Then, the last formula converges to 0 as $r\to\infty$, and hence, $d_{BL}(\mu'_r,\nu)\to 0$. Also, $s(r)$ also converges to 0. So, the claim is proved.
	\end{proof}
	
	\begin{proof}[Proof of \Cref{prop:product-Linf}  (General Case)]
		If $b'=0$, the claim can be proved similarly to the graph case. So, assume $b,b'>0$. 
		Consider a function $s(r)$ that satisfies the claim of \Cref{lem:annulus} for both $\mu$ and $\mu'$. Let $\Delta B_r(\origin):=B_{r+s(r)}(\origin)\setminus B_r(\origin)$ and $\Delta F_r:=F({r+s(r)})-F(r)$.
		Let $\mu''_r$ be the normalized restriction of $\mu''$ to $\Delta B_{r}(\origin'')$. It is enough to prove that $\mu''_r$ converges to the desired distribution. Partition the last annulus as $\cup_{i=-\lceil r\rceil}^{\lceil r\rceil} T_{r,i}$, where $T_{r,i}:= \Delta B_{(r-i s(r))}(\origin)\times \Delta B_{r}(\origin')$ when $i\geq 0$, and $T_{r,i}:=\Delta B_{r}(\origin)\times \Delta B_{(r+i s(r))}(\origin')$ when $i\leq 0$. We parametrize these cubes continuously as follows:
		For $t\geq 0$, let $i_r(t)$ be the largest $i\geq 0$ that satisfies $\mu''(\cup_{j=0}^i T_{r,j})/\Delta F''_r \leq t$. For $t<0$, let $i_r(t)$ be the largest $i<0$ such that $\mu''(\cup_{j=i}^{-1}T_{r,j})/\Delta F''_r >-t$. Note that $i_r(\cdot)$ is well defined only on an interval containing $0$.
		Let $\mu''_{r,t}$ be the normalized restriction of $\mu''$ to $T_{r,i_r(t)}$. Now, one has $\mu''_r = \int c_{r,t}\mu''_{r,t} dt$, where $c_{r,t}=1$ if $i(t)$ is well defined, and $c_{r,t}=0$ otherwise. As $r\to\infty$, since $s(r)\to 0$, one can show that $i_r(t)$ converges pointwise. Indeed, the limit is the unique function $i(t)$ that satisfies $\int_0^{i(t)} p^{-t}dt = ct$ for $t\geq 0$ and $\int_0^{-i(t)} (p')^{-t}dt = -ct$ for $t<0$, where $c$ is some constant. Also, since the cube $T_{r,i_r(t)}$ is thin, one can show that $\mu''_{r,t}$ converges to $\nu_0\otimes\nu'_0\otimes\delta_{i(t)}$ when $t$ is fixed and $r\to\infty$. So, the claim is implied by part~\ref{lem:sum:cont} of \Cref{lem:sum}.
	\end{proof}

	\begin{prop}[\textsc{Full Convergence in the $L^1$ Metric}]
		\label{prop:product-L1}
		Under the $L^1$ metric, {in the following cases,} $E''$ satisfies the full convergence in \Cref{thm.scfc} or \Cref{thm.ctfg} (the latter if $E''$ is a graph):
		\begin{enumerate}[label=(\roman*)]
			\item \label{prop:product-L1:=} When $b=b'$ {and $\int_0^{\infty} p^{-s} dF(s) = \int_0^{\infty}p^{-s}dF'(s)=\infty$. In this case,} $E''$ has the same growth index $b$ and $\nu''_0=\nu_0\otimes \nu'_0$, which is supported on $(\partial E\times \partial E')\subseteq\partial E''$.
			\item \label{prop:product-L1:<} When $b>b'$. In this case, $E''$ has growth index $b$ and $\nu''_0=\nu_0\otimes \tilde{\mu}'$, where $\tilde{\mu}'$ is the probability measure on $E'$ such that $d\tilde{\mu}'(x')= c\times p^{-d(x',\origin')}d\mu'(x')$ for some constant $c$. Note that $\nu''_0$ is supported on $(\partial E\times E')\subseteq\partial E''$.
			\item \label{prop:product-L1:1}The last claim still holds if one relaxes the conditions of $E'$ by allowing $b'=0$ (assuming the existence of $\nu'_0$); e.g., when $E'=\mathbb Z^k$ equipped with the counting measure, {or when $E'=\mathbb R^k$}.
		\end{enumerate}
	\end{prop}
	\begin{proof}
		We assume that both $E$ and $E'$ satisfy the conditions of \Cref{thm.scfc}. {The case where both are graphs is similar (and easier) and is skipped.}
		One has $F''(r) = \int_0^r F(r-s)dF'(s) = \int_0^r F'(r-s)dF(s)$.
		We first prove~\ref{prop:product-L1:<} and~\ref{prop:product-L1:1}, which are easier than~\ref{prop:product-L1:=}.
		
		\noindent
		\ref{prop:product-L1:<} and~\ref{prop:product-L1:1}. To show that $F''$ is log-regularly-varying, we first show that 
		$$
		\lim_{r\to\infty}F''(r)/F(r)=\int_0^{\infty} p^{-s}dF'(s) \; .
		$$ 
		For every $s$, one has $F'(r-s)/F(r)\to p^{-s}$. To integrate over $s$, we use dominated convergence. Let $\alpha_r$ be the measure on $(0,r)$ defined by $\alpha_r(ds) = F(r-s)/F(r)dF'(s)$. Choose $p''$ such that $p>p''>p'$ and $M$ such that, for all $r\geq M$, one has $F(r)/F(r-1)>p''$. Then $F(r-s)/F(r)<C(p'')^{-s}$ for all $M\leq s\leq r$, where $C$ is a suitable constant ($C$ can be arbitrarily close to 1, this is just Potter's bound (see \emph{e.g.}~\cite[Theorem 1.5.6]{BGT}). Hence, $F(r-s)/F(r)<C(p'')^{-s+M}$ for all $r\geq M$ and all $0\leq s\leq r$. Therefore, the measures $\alpha_r$ are bounded from above by the measure $\alpha$ on $(0,\infty)$ defined by $\alpha(ds)\coloneqq C(p'')^{-s+M}dF'(s)$. Since $\alpha$ has finite mass, dominated convergence implies that $\lim_{r\to\infty}F''(r)/F(r)=\int_0^{\infty} p^{-s}dF'(s)$. Since $F$ is log-regularly-varying, the last convergence implies that $F''$ is also log-regularly-varying with the same index.
		
		It remains to prove that the measure $\mu''_r\coloneqq  F''(r)^{-1}\restrict{\mu''}{B_r(\origin'')}$ converges to $\nu_0\otimes\nu'_0$. We prove this by decomposing the ball $B_r(\origin'')$ into horizontal thin rectangular areas as follows. Fix $\epsilon>0$ {for the moment (for the graph case, it is enough to consider $\epsilon=1$).} For $n\geq 1$, let $S_{n\epsilon}(\origin')\coloneqq B_{n\epsilon}(\origin')\setminus B_{(n-1)\epsilon}(\origin')$.  One has $\mu''_r\geq \alpha_r\coloneqq  \sum_{n=1}^{\floor{ r/\epsilon}}\alpha_{r,n}$, where 
		\begin{eqnarray*}
			\alpha_{r,n}&\coloneqq & F''(r)^{-1}\left(\restrict{\mu}{B_{r-n\epsilon}(\origin)}\right)\otimes\left(\restrict{\mu'}{S_{n\epsilon}(\origin')}\right)\\
			&=& \left(F(r-n\epsilon)^{-1}\restrict{\mu}{B_{r-n\epsilon}(\origin)}\right)\otimes\left(\frac{F(r-n\epsilon)}{F''(r)}\restrict{\mu'}{S_{n\epsilon}(\origin')}\right).
		\end{eqnarray*}
		For fixed $n$, the last line converges to $\nu_0\otimes \left(cp^{-n\epsilon}\restrict{\mu'}{S_{n\epsilon}(\origin')} \right)$, where $c\coloneqq \lim_r F(r)/F''(r)$. So, {by \Cref{lem:sum},} $\alpha_r$ converges to $\nu_0\otimes \left(\sum_n cp^{-n\epsilon}\restrict{\mu'}{S_{n\epsilon}(\origin')}\right)$. Similarly, 
		one can show that $\mu''_r$ is bounded from above by a measure $\alpha'_r$ that converges to $\nu_0\otimes \left(\sum_n cp^{-(n-1)\epsilon}\restrict{\mu'}{S_{n\epsilon}(\origin')}\right)$. If $\epsilon$ is small enough, the limiting upper and lower bounds are arbitrarily close to $\nu_0\otimes \tilde{\mu}'$. This proves the claim.
		
		\ref{prop:product-L1:=}. Since $b=b'$, the last Potter-type argument is not available for controlling the middle slices. However, since $F''(r)\geq \int_0^{k} F(r-s)dF'(s)$ for fixed $k$ and for $r\geq k$, one can similarly deduce that $\liminf_r F''(r)/F(r)\geq \int_0^k p^{-s}dF'(s)$. So, the assumptions of \ref{prop:product-L1:=} imply that $\lim_r F''(r)/F(r) = \lim_r F''(r)/F'(r)=\infty$. This does not imply that $F''$ is log-regularly-varying. So, we prove this directly as follows. Fix $k\geq 0$ and $\epsilon>0$. By uniform convergence for regularly-varying functions, there exists $M<\infty$ such that $\forall r\geq M: F(r+r_0) = p^{r_0} F(r) (1\pm \epsilon)$. So, for $r\geq M$, one has
		\begin{eqnarray}
			\nonumber
			\int_0^{r-M} F(r+k-s) dF'(s) &=& \int_0^{r-M} p^k F(r-s) (1\pm \epsilon)dF'(s)\\
			\label{eq:product:1}
			&=& (1\pm \epsilon)p^k \int_0^{r-M} F(r-s) dF'(s).
		\end{eqnarray}
		We show that the remaining terms are negligible. Integration by parts gives that 
		\[
		\int_{r-M}^{r+k} F(r+k-s)dF'(s) \leq \int_0^{M+k} F'(r+k-s)dF(s).
		\]
		The arguments in the beginning of the proof show that the right hand side is negligible compared to $F''(r+k)$. Similarly, $\int_{r-M}^r F(r-s)dF'(s)$ is negligible compared to $F''(r)$. So, \eqref{eq:product:1} implies that $F''(r+k) = (1\pm \epsilon+o(1)) p^k F''(r)$. By letting $\epsilon$ converge to zero, one obtains that $\lim_r F''(r+k)/F''(r) = p^k$. So, $F''$ is log-regularly-varying with index $b=\ln p$.
		
		It remains to prove that $\mu''_r$ converges to $\nu''_0:=\nu_0\otimes\nu'_0$. We prove this by showing that the contribution of the \textit{first} horizontal and vertical slices of $B_r(\origin'')$ is negligible. Fix $\epsilon>0$ and let $\mu''_{i,j}$ be the normalized restriction of $\mu''$ to the square-like set $T_{i,j}:=S_{i\epsilon}(\origin)\times S_{j\epsilon}(\origin')$. Let $c_{i,j}:=\mu''(T_{i,j})$. 
		By assumption, $\mu''_{i,j}$ converges to $\nu''_0$ as $i,j\to\infty$. So, there exists $M<\infty$ such that $d_{BL}(\mu''_{i,j}, \nu''_0)\leq \epsilon$ for all $i,j\geq M$, where $d_{BL}$ denotes the bounded-Lipschitz distance (given an arbitrary metrization of $\overline{E''}$). Now, for every non-negative 1-bounded 1-Lipschitz function $g$ on $E''$, one has
		\begin{eqnarray*}
			\int g d\mu''_r &\leq& \sum_{i,j} \frac{c_{i,j}}{F''(r)} \int g\mu_{i,j}\\
			&=& \sum_{i,j} \frac{c_{i,j}}{F''(r)}\int g d(\mu_{i,j}-\nu''_0) + (\frac{\sum_{i,j} c_{i,j}}{F''(r)})\int g d\nu''_0\\
			&\leq & \sum_{i<M\text{ or }j<M} \frac{c_{i,j}}{F''(r)} + \sum_{i\geq M, j\geq M} \frac{\epsilon c_{i,j}}{F''(r)} + \left(\frac{\sum_{i,j} c_{i,j}}{F''(r)}\right)\int g d\nu''_0,
		\end{eqnarray*}
		where the sum is over all $0\leq i\leq\lceil \frac{r}{\epsilon}\rceil$ and all $1\leq j\leq \lceil \frac{r}{\epsilon}\rceil-i+1$. 
		In the last line, the first term converges to zero as $r\to\infty$. The second term is bounded by $\epsilon F''(r+2\epsilon)/F''(r)$. Also, the last term is at most $F''(r+2\epsilon)/F''(r) \int g d\nu''_0$. So, by choosing $\epsilon$ small enough, $\limsup_r \int g d\mu''_r-\int g d\nu''_0$ is arbitrarily small. By a similar lower bound, one deduces that $\lim \int g d\mu''_r = \int g d\nu''_0$. This implies that $d_{BL}(\mu''_r,\nu''_0)\to 0$, and the claim is proved.
	\end{proof}

    As an example, we apply \Cref{prop:product-L1} to products of trees as follows:

    \begin{proof}[Proof of \Cref{cor:xi-examples} for Products of Regular Trees]
			Consider $\rtree{p}\times\rtree{q}$, where $p\geq q$. Let $\origin''=(\origin,\origin')$. Any point $x'':=(x,x')\in S_n(\origin'')$ has exactly 2 neighbors in $S_{n-1}(\origin'')$ and $p+q$ neighbors in $S_{n+1}(\origin'')$, except when $x=\origin$ or $x'=\origin'$. Assuming $p=q>1$, these exceptions are negligible (see \Cref{prop:product-L1}). This implies that $F''_r$ is log-regularly-varying with index $\ln p$. Also, if $\bs U_n$ is a random element of $S_n(\origin'')$ and $\bs V_n$ is a random neighbor of $\bs U_n$ in $S_{n+1}(\origin'')$, then the distribution of $(\bs U_n,\bs V_n)$ is approximately (in total variation distance) the normalized version of $\alpha_1^n$ defined in~\eqref{eq:inter-level}. One can deduce that $\alpha_1^n$ converges to $\mathrm{diag}(\nu_0)$, where $\nu_0$ is the harmonic measure on $\partial\rtree{p}$ (use \Cref{prop:product-L1} and Skorohod's representation theorem to assume that $\bs U_n$ converges a.s. to a random point of $\partial\rtree{p}\times\partial\rtree{q}$ and note that $\bs V_n$ converges to the same limit).
			
			If $p>q\geq 1$, the exceptions $x'=\origin'$ are not neglibigle, but only a biasing is needed on these exceptions. One can deduce similarly that $F''_r$ is log-regularly-varying with index $\ln p$. Also, it is enough to study the limit of $(\bs U_n,\bs V_n)$. By \Cref{prop:product-L1}, $\bs U_n$ converges to a random element of $\partial\rtree{p}\times \rtree{q}$, namely, to $(\Theta,\bs X')$. Using Skorokhod's representation theorem, we may assume that the last convergence holds a.s. In this case, the second coordinate of $\bs V_n$ converges to a random (non-uniform) element of $B_1(\bs X')$. So, by \Cref{lem:productBoundary}, the limit of $\bs V_n$ is not equal to that of $\bs U_n$ with positive probability. Hence, the limit of $\alpha_1^n$ is not supported on the diagonal, and independence on $\xi$ does not hold.
	\end{proof}
	
	\begin{example}
		{Equip $\rtree{p}\times \rtree{q}$ with the weighted $L^1$ metric $d((x,x'),(y,y')):= d(x,x')+c d(y,y')$, where $c=\ln p/\ln q$. Similarly to the last proof, one can use \Cref{prop:product-L1} and show that the assumptions of \Cref{thm.ctfg} are satisfied, and that the limiting measure is supported on $\partial\rtree{p}\times\partial\rtree{q}$. In fact, one can see that the resulting IPVT is identical to the IPVT with the $L^2$ metric studied in~\cite{MiMe23}. See also \Cref{prob:l2}.}
	\end{example}

    \begin{example} 
          $\HH^2 \times \HH^2$ equipped with the $L^1$ metric has been studied in \cite{IPVTH2H2}.  \Cref{prop:product-L1} applied to this case provides full convergence towards ${\rm IPVT}$, and also that the corona measure is concentrated on $\partial \HH^2 \times \partial \HH^2$. This provides an alternative proof to \cite[Theorem 2.1]{IPVTH2H2}. 
    \end{example}

    \section{On the IPVT of Diestel--Leader Graphs}
	\label{sec:dl}
	
	In this section, we construct the IPVT of Diestel--Leader graphs and prove \Cref{thm.dl-convergence} and the surrounding results announced in the introduction.
	
	Let $p\geq 2$. As in \Cref{subsec:regularTree}, let $T\coloneqq \rtree{p}$ be the $(p+1)$-regular tree. Fix a vertex $\origin\in T$ and an end $\omega_0\in\partial T$. The \dfn{confluent} of two vertices $u,v\in T$ with respect to $\omega_{0}$ is the intersection of the three geodesics $uv, u\omega_0$ and $v\omega_0$ and is denoted by $u\wedge v$ (the dependence on the chosen end $\omega_{0}$ being clear from the context). {We regard $h\coloneqq d_{\omega_0}$ as} the \dfn{height function} on $T$ with respect to $\omega_{0}$; i.e., $h(u)\coloneqq  d(u,u\wedge\origin)-d(\origin,u\wedge\origin)$. 
	Given $q\geq 2$, consider $T'\coloneqq \rtree{q}$, $\origin'\in T'$ and $\omega'_0\in\partial T'$ similarly, and let $h'\coloneqq d_{\omega'_0}$ be the analogous height function on $T'$.
	
	We recall the following definition from~\cite{DL}.

	\begin{defn}[\textsc{Diestel--Leader Graph}]
		\label{def:dl}
		Given natural numbers $p,q\geq 2$, consider $(T\coloneqq \rtree{p},\origin,\omega_0)$ and $(T'\coloneqq \rtree{q},\origin',\omega'_0)$ as above. The \dfn{Diestel--Leader graph} of order $(p,q)$, denoted by ${\rm DL}(p,q)$, is the {directed} graph whose vertex set $V$ is
		$$
		V \overset{\rm def}{=} \{ (v,v') \in T\times T' \stdp  h(v)+h'(v')=0 \}
		$$
		and whose edge set is defined by
		$$
		(v,v') \sim_{\rm DL}(p,q) (w,w') \iff v\sim_{T} w  \quad \text{and} \quad v'\sim_{T'} w' \; .
		$$
		The direction of the edges are defined by adding the condition $h(w)=h(v)-1$ {(i.e., $h'(w')=h'(v')+1$).}
	\end{defn}
	
	As mentioned in \Cref{cor:xi-examples}, ${\rm DL}(q,q)$ is unimodular, transitive and a Cayley graph of the lamplighter group $\mathbb{Z}_{p} \wr \mathbb{Z}$, see Woess~\cite{woess2005}.

\begin{lem}[\textsc{Metric on ${\rm DL}(p,q)$, \cite{DL-randomwalk}}]
		\label{lem:DL-metric}
		One has
		\begin{eqnarray*}
			d((x,x'),(y,y')) &=& - 2 h(x\wedge y) - 2 h'(x' \wedge y') - \mynorm{h(x)- h(y)}\\
			&=& d_T(x,y) + d_{T'}(x',y') -\mynorm{h(x)-h(y)}.	\end{eqnarray*} 
	\end{lem}
	
	\subsection{The Horoboundary of ${\rm DL}(p,q)$}

	\begin{figure}[!hbtp]
		\hspace{45pt}
		\begin{overpic}[width=.9\textwidth]{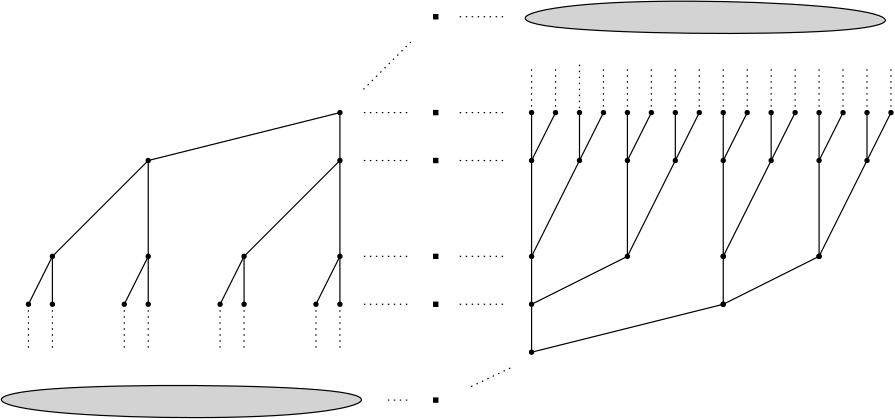}
			\put (15,-2) {\small{$\rho_{\omega}, \, \omega \neq \omega_{0}$}}
			\put (73,48) {\small{$\rho'_{\omega'}, \, \omega' \neq \omega'_{0}$}}
			\put (49,0) {\small{$\rho_{\omega_{0}}$}}
			\put (49,11) {\small{$\rho_{1}$}}
			\put (49,16) {\small{$\rho_{0}$}}
			\put (49,27) {\small{$\rho_{-1}$}}
			\put (49,32) {\small{$\rho_{-2}$}}
			\put (49,47) {\small{$\rho'_{\omega'_{0}}$}}
			\put (1,18) {\small{$\rho_{x_{0}}$}}
			\put (93,18) {\small{$\rho'_{x'_{0}}$}}
		\end{overpic}
		\vspace{5pt}
		\caption{Portrait of the horoboundary of ${\rm DL}(2,2)$ {described in \Cref{prop.gbdl,prop.gbdl-conv,prop:DL-top}}.}
		\label{fig:DLboundary}
	\end{figure}
	
	To study the IPVT of Diestel--Leader graphs, we first identify their horoboundary as follows.

	\begin{prop}{\textsc{(Horoboundary of ${\rm DL}(p,q)$)}}
		\label{prop.gbdl}
		The horoboundary of ${\rm DL}(p,q)$ consists of the following horofunctions:
		\begin{eqnarray*}
			\rho_{\omega}(y,y') &\coloneqq & d(y,\omega) \text{ given } w\in \partial T,\\
			\rho'_{\omega'}(y,y') &\coloneqq & d(y',\omega') \text{ given } w'\in \partial T',\\
			\rho_{x_0}(y,y') &\coloneqq & -2h(y\wedge x_0) - \mynorm{h(y)-h(x_0)} = d(y,x_0)-2\max\{h(y),h(x_0)\} \text{ given } x_0\in T,\\
			\rho'_{x'_0}(y,y') &\coloneqq & -2h'(y'\wedge x'_0) - \mynorm{h'(y')-h'(x'_0)} = d(y',x'_0)-2\max\{h'(y'),h'(x'_0)\} \text{ given } x'_0\in T',\\
			\rho_i(y,y') &\coloneqq & -\mynorm{h(y)-i} = -\mynorm{h'(y')+i} \text{ given } i\in\mathbb Z.
		\end{eqnarray*}
	\end{prop}

    As sets, one can write 
	$$\partial {\rm DL}(p,q) = \overline T\sqcup\overline {T'}\sqcup\mathbb Z,$$ but it should be noted that the induced topologies on $\overline{T}$ and $\overline{T'}$ are different from the inherent topologies of $\overline{T}$ and $\overline{T'}$ (see \Cref{prop:DL-top}). The last proposition is implied by the following:
	
	\begin{prop}[\textsc{Convergence Towards $\partial {\rm DL}(p,q)$}]
		\label{prop.gbdl-conv}
		For a sequence $(x_n,x'_n)\in {\rm DL}(p,q)$,
		\begin{itemize}
			\item $(x_n,x'_n)\to \rho_{\omega}$ if $x_n\to \omega\neq\omega_0$, where $\omega\in\partial T$ (and hence $x'_n\to \omega'_0$).
			\item $(x_n,x'_n)\to \rho'_{\omega'}$ if $x'_n\to \omega'\neq\omega'_0$, where $\omega'\in\partial T'$ (and hence $x_n\to \omega_0$).
			\item $(x_n,x'_n)\to \rho_{\omega_0}$ if $x_n\to\omega_0$, $x'_n\to\omega'_0$ and $h(x_n)\to+\infty$. 
			\item $(x_n,x'_n)\to \rho'_{\omega'_0}$ if $x_n\to\omega_0$, $x'_n\to\omega'_0$ and $h'(x'_n)\to+\infty$.
			\item $(x_n,x'_n)\to\rho_{x_0}$ if $x_n=x_0$ and $d(x'_n,\origin')\to\infty$ (s.th. $h'(x'_n)=-h(x_0)$).
			\item $(x_n,x'_n)\to\rho'_{x'_0}$ if $x'_n=x'_0$ and $d(x_n,\origin)\to\infty$ (s.th. $h(x_n)=-h'(x'_0)$).
			\item $(x_n,x'_n)\to \rho_i$ if $d(x_n,\origin)\to\infty$ and $d(x'_n,\origin')\to \infty$ s.th. $h(x_n)=-h'(x'_n)=i$.
		\end{itemize}
	\end{prop}
	\begin{proof}
		Fix $(y,y')\in {\rm DL}(p,q)$.
		By \Cref{lem:DL-metric}, $d((x_n,x'_n),(y,y'))-d((x_n,x'_n),(\origin,\origin'))$ is equal to $2 I_n+2 I'_n+J_n$, where $I_n = h(x_n\wedge \origin)-h(x_n\wedge y)$, $I'_n= h'(x'_n\wedge\origin')-h'(x'_n\wedge y')$ and $J_n = \mynorm{h(x_n)}-\mynorm{h(x_n)-h(y)}$. The limit of $I_n$ is as follows: $I_n\to 0$ if $x_n\to\omega_0$, $I_n\to h(x_0\wedge\origin)-h(x_0\wedge y)$ if $x_n=x_0$, and $I_n\to h(\omega\wedge\origin)-h(\omega\wedge y)$ if $x_n\to \omega\neq\omega_0$. Similar limits hold for $I'_n$. The limit of $J_n$ is as follows: $J_n\to h(y)$ if $h(x_n)\to\infty$, $J_n\to h'(y')$ if $h(x_n)\to-\infty$ and $J_n\to \mynorm{h(x_0)}-\mynorm{h(x_0)-h(y)}$ if $h(x_n)=h(x_0)$. The claim follows by considering the different cases of the limit of $(I_n,I'_n,J_n)$.
	\end{proof}
	\begin{proof}[Proof of \Cref{prop.gbdl}]
		It is straightforward to see that every sequence of distinct points in ${\rm DL}(p,q)$ contains a subsequence of one of the seven types mentioned in \Cref{prop.gbdl-conv}. This implies the claim.
	\end{proof}

    \begin{prop}[\textsc{Topology of $\partial {\rm DL}(p,q)$}]
		\label{prop:DL-top}
		The topology of $\partial {\rm DL}(p,q)$ is as follows
		\begin{itemize}
			\item $\rho_{x_0}$ and $\rho'_{x'_0}$ are isolated points. 
			\item $\rho_i\to \rho_{\omega_0}$ if $i\to\infty$, and $\rho_i\to \rho'_{\omega'_0}$ if $i\to-\infty$.
			\item $\rho_{\omega_n}\to\rho_{\omega}$ if $\omega_n\to\omega$, and $\rho'_{\omega'_n}\to\rho'_{\omega'}$  if $\omega'_n\to\omega'$.
			\item $\rho_{x_0}\to\rho_{\omega}$ for $\omega\neq \omega_0$ if $x_0\to\omega$.
			\item $\rho_{x_0}\to\rho_{\omega_0}$ if $x_0\to\omega_0$ and $h(x_0)\to-\infty$.
			\item $\rho_{x_0}\to\rho'_{\omega'_0}$ if $x_0\to\omega_0$ and $h(x_0)\to+\infty$.
			\item $\rho_{x_0}\to \rho_i$ if $x_0\to\infty$ such that $h(x_0)=i$.
		\end{itemize}
		and similarly for $\rho'_{x'_0}$. Also, every sequence of distinct points in $\partial {\rm DL}(p,q)$ has a subsequence of one of these types. See Figure~\ref{fig:DLboundary} for a portrait in the case $(p,q)=(2,2)$.
	\end{prop}
	{Notice the case $\rho_{x_0}\to\rho'_{\omega'_0}$ despited that $x_0\to \omega_0$ (see also Figure~\ref{fig:DLboundary}), which shows that the topology of $\overline{T}$ in $\partial {\rm DL}(p,q)$ is different from the inherent topology of $\overline{T}$.}

    \subsection{Convergence to ${\rm IPVT}({\rm DL}(p,q))$}
	\label{subsec:dl-conv}
	
	Assume $p\geq q$ and let $\delta\coloneqq \ln q/\ln p$. So, $q=p^{\delta}$ and $\delta\leq 1$. We first study the IPVT of the vertex-measured version of ${\rm DL}(p,q)$. The edge measured will be studied at the end of this subsection.
	To prove convergence and to obtain the distribution of the IPVT, by \Cref{thm.ctfg}, we need to estimate the growth of the volume function $F(r)$ and also obtain the limit of the uniform measure on the sphere $S_r(\origin)$. We decompose $S_r$ into the following parts:
	\begin{eqnarray*}
		S_r^{i,j,0}&\coloneqq & {\{ (x,x')\in S_r:\ h(x)= 0,\ h(x\wedge\origin) = -i,\ h'(x'\wedge\origin')=-j\}},\\
		S_r^{i,j,+}&\coloneqq & \{ (x,x')\in S_r:\ h(x)> 0,\ h(x\wedge\origin) = -i,\ h'(x'\wedge\origin')=-j\},\\
		S_r^{i,j,-}&\coloneqq & \{ (x,x')\in S_r:\ h(x)< 0,\ h(x\wedge\origin) = -i,\ h'(x'\wedge\origin')=-j\}.
	\end{eqnarray*}
	{Also, define $S_r^0, S_r^+$ and $S_r^-$ by taking union over all cases of $i$ and $j$.
		If $(x,x')\in S_r^{{i,j,0}}$, then $r = d((x,x'),(\origin,\origin'))= 2i+2j$. 
		In this case,
		\begin{eqnarray}
			\nonumber
			\mynorm{S_r^{i,j,0}} &=& \frac{p-1_{\{i>0\}}}{p} p^{i} \cdot \frac{q-1_{\{j>0\}}}{q} q^{j}\leq p^{(1-\delta)i+ \delta r/2}.
		\end{eqnarray}
		If $r$ is odd, then $S_r^0$ is empty. If $r$ is even, by summing over $i$, one obtains $\mynorm{S_r^0} = O(p^{r/2})$.}
	
	If $(x,x')\in S_r^{i,j,+}$, then $r=d((x,x'),(\origin,\origin')) = 2i + 2j - h(x)$. Therefore, $h(x)=2i+2j-r$. The region of validity of $(i,j)$ is determined by $i=-h(x\wedge\origin)\geq 0$, $2i+2j-r = h(x)> 0$ and $-h(x)=h'(x')\geq h'(x'\wedge\origin')$, which implies that $2i+j\leq r$ (see Figure~\ref{fig.valreg}). 
	{Given $h(x)$, then number of cases for $x$ that satisfy $h(x\wedge\origin)=-i$ is $\frac{p-1}{p}p^{i+h(x)}$, provided that $i>0$ and $i+h(x)>0$. Otherwise, it is $p^{i+h(x)}$.} 
	In the validity region, $i>0$ implies $i+h(x)>0$, and $j+h'(x')>0$ implies $j>0$. So, 
	\begin{eqnarray}
		\nonumber
		\mynorm{S_r^{i,j,+}} &=& \frac{p-1_{\{i>0\}}}{p} p^{i+h(x)} \frac{q-1_{\{{j+h'(x')>0}\}}}{q} q^{j+h'(x')}\\
		\nonumber &=& \frac{p-1_{\{i>0\}}}{p}\cdot \frac{q-1_{\{2i+j<r\}}}{q} p^{3i+2j-r} q^{-2i-j+r}\\
		\label{eq:S+} &=& \frac{p-1_{\{i>0\}}}{p}\cdot \frac{q-1_{\{2i+j<r\}}}{q} p^{(3-2\delta)i+(2-\delta)j+(\delta-1)r}.
	\end{eqnarray}
    
	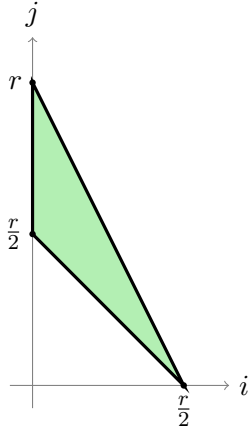
\begin{figure}
		\centering
		\begin{tikzpicture}[scale=1]
			\def\r{4}
			
			\draw[->, gray] (-0.3,0) -- (\r/2+0.6,0) node[right, black] {$i$};
			\draw[->, gray] (0,-0.3) -- (0,\r+0.6) node[above, black] {$j$};
			
			\coordinate (A) at (0,\r/2);
			\coordinate (B) at (0,\r);
			\coordinate (C) at (\r/2,0);
			
			\filldraw[
			fill=green!80!black,
			fill opacity=0.3,
			draw=black,
			very thick
			] (A) -- (B) -- (C) -- cycle;
			
			\fill (A) circle (1.2pt);
			\fill (B) circle (1.2pt);
			\fill (C) circle (1.2pt);
			
			\node[left]  at (A) {$\frac{r}{2}$};
			\node[left]  at (B) {$r$};
			\node[below] at (C) {$\frac{r}{2}$};
		\end{tikzpicture}
		\caption{{Validity region defined before~\eqref{eq:S+}}.}
		\label{fig.valreg}
	\end{figure}
    
	We claim that $\mynorm{S_r^+}=\Theta(p^r)$. The maximum exponent of $p$ in~\eqref{eq:S+} happens at one of the corners of the region of validity. At $(i,j)=(0,r/2)$ (which might not be an integer point), the exponent is $\delta r/2$. At $(i,j)=(0,r)$, the exponent is $r$ and at $(i,j)=(r/2,0)$, the exponent is $r/2$. 
	So the maximum happens at the corner $(0,r)$. By a change of variables, we choose this corner as the origin and map the triangle to the positive cone: Let $(i,j)= (0,r)+ a(0,-1)+b(1,-2)=(b, r-a-2b)$, which implies that $a,b\geq 0$. One obtains
	\[
	p^{-r}\mynorm{S_r^{i,j,+}} = \frac{p-1_{\{b>0\}}}{p}\cdot \frac{q-1_{\{a>0\}}}{q} p^{-(2-\delta) a - b}.
	\]
	When $r\to\infty$, the sum over all $(a,b)$ in the validity triangle converges to the sum over all $a,b\geq 0$. So, 
	\begin{eqnarray}
		\nonumber p^{-r}\mynorm{S_r^+} &=& \sum_{a\geq 0}\sum_{b\geq 0} \frac{p-1_{\{b>0\}}}{p}\cdot \frac{q-1_{\{a>0\}}}{q} p^{-(2-\delta) a - b} + o(1)\\
		\nonumber &=& \left[\sum_{b\geq 0} \frac{p-1_{\{b>0\}}}{p}p^{-b} \right] \cdot \left[\sum_{a\geq 0} \frac{q-1_{\{a>0\}}}{q} p^{-(2-\delta) a}\right] +o(1)\\
		\nonumber &=& \left[1+ \frac 1 p\right] \cdot \left[1+  \frac{q-1}{(p^{2}-q)}\right] + o(1)\\
		\label{eq:S+:2}&=& \frac{(p+1)(p^2-1)}{p(p^2-q)} + o(1).
	\end{eqnarray}
	
	We now estimate $\mynorm{S_r^-}$. If $\delta=1$, then $p=q$ and $\mynorm{S_r^-} =\mynorm{S_r^+}$ by symmetry. Now, assume $\delta<1$.
	If $(x,x')\in S_r^{i,j,-}$, then $r=d((x,x'),(\origin,\origin')) = 2i + 2j + h(x)$. Therefore, $h(x)=r-2i-2j$. The region of validity of $(i,j)$ is determined by $j=-h'(x'\wedge\origin')\geq 0$, $2i+2j-r = -h(x)> 0$ and $h(x)\geq h(x\wedge\origin)$, which implies that $i+2j\leq r$. In this triangle, one has {$i>0$ and $j+h'(x')>0$. So}
	\begin{eqnarray*}
		\nonumber
		\mynorm{S_r^{i,j,-}} &=& \frac{p-1_{\{{i+h(x)}>0\}}}{p} p^{i+h(x)} \frac{q-1_{\{j>0\}}}{q} q^{j+h'(x')}\\
		\nonumber &=& \frac{p-1_{\{i+2j<r\}}}{p}\cdot \frac{q-1_{\{j>0\}}}{q} p^{r-i-2j} q^{2i+3j-r}\\
		\label{eq:S-}&=& \frac{p-1_{\{i+2j<r\}}}{p}\cdot \frac{q-1_{\{j>0\}}}{q} p^{(2\delta-1)i+(3\delta-2)j+(1-\delta)r}.
	\end{eqnarray*}
	At the three corners $(0, r/2)$, $(r/2,0)$ and $(r,0)$ of the validity region, the exponent of $p$ in the last equation is $\delta r/2$, $r/2$ and $\delta r$ respectively. So, the maximum exponent in the validity region is $\max\{\delta,1/2\}r$. By a change of exponent similarly to the previous case and mapping the validity region to the positive cone, one obtains that $\mynorm{S_r^-} = o(p^r)$. In fact, if $\delta\neq\frac 12$, then it is $\Theta(p^{\max\{\delta,1/2\}r})$, and if $\delta=1/2$, then it is $\Theta(rp^{r/2})$. We are now ready to prove \Cref{thm.dl-convergence}.

	\begin{proof}[Proof of \Cref{thm.dl-convergence} (when $p>q$)]
		The estimates of $\mynorm{S_r^0}$, $\mynorm{S_r^-}$ and $\mynorm{S_r^+}$, mentioned above, imply
		\begin{equation}
			\label{eq:dl-growth1}
			p^{-r}\mynorm{S_r} = \frac{(p+1)(p^2-1)}{p(p^2-q)} + o(1).
		\end{equation}
		This proves that $\restrict{F}{\mathbb N}$ is log-regularly-varying with index $\ln p$. Also, these estimates show that $S_r^+$ occupies asymptotically almost all of $S_r$. For finding the limit of the uniform measure on $S_r$, we use \Cref{lem:sum}. According to the change of variables mentioned after~\eqref{eq:S+}, let $U_r^{a,b}\coloneqq  S_r^{b,r-a-2b,+}$ be the set of $(x,x')\in S_r$ such that $h(x\wedge\origin)=b$ and $h'(x')-h'(x'\wedge\origin')=a$. Let $\mu_r^{a,b}$ be the uniform probability measure on $U_r^{a,b}$ and $c_r^{a,b}\coloneqq \mynorm{U_r^{a,b}}/\mynorm{S_r}$. Letting $\eta$ be the harmonic measure on $\partial T$, it is clear that, as $r\to\infty$, $\mu_r^{a,b}$ converges to $\eta|V_b$ (which reads as $\eta$ conditioned on $V_b$), where $$
		V_b\coloneqq \{\omega\in \partial T: h(w\wedge w_0) = -b\}\subseteq\partial T\subseteq \partial {\rm DL}(p,q) \; .
		$$ Also, the calculation in~\eqref{eq:S+:2} shows that $\sum_a c_r^{a,b}= \frac{p-1_{\{b>0\}}}{p}p^{-b} / (\frac{p+1}p + o(1))$, and the latter converges to $\eta(V_b)$ as $r\to\infty$. So, \Cref{lem:sum} implies that the uniform measure on $S_r$ converges to $\eta$. Therefore, the conditions of \Cref{thm.ctfg} are satisfied and $\nu_0=\eta$.
		
		Now, fix a parameter $\xi$ and let $\Phi\coloneqq(\Theta_i,\Delta_i)_{i\geq 1}$ be the nuclei process given in \Cref{thm.ctfg}. By \Cref{prop:diagramConvergence}, the corresponding ${\rm IPVT}_{\xi}$ is just $\Vor(\Phi)$. By the previous paragraph, $\forall i: \Theta_i\in\partial T$. In addition, the distribution of $\Phi$ is identical to the distribution of the nuclei process for the tree $T$ with parameter {$\xi'=\xi+c$, where $c$ is determined by the ratio of the volume of balls in ${\rm DL}(p,q)$ and $T$}. By \Cref{prop.gbdl}, the horofunctions corresponding to $\Theta_1,\Theta_2,\ldots$ depend only on the first coordinates. This implies that the Voronoi diagram of $\Phi$ in ${\rm DL}(p,q)$ is just the inverse image under $\pi$ of the Voronoi diagram of $\Phi$ in $T$. So, the claims are proved for the vertex-measured case.
		
		For the edge-measured version, since ${\rm DL}(p,q)$ is bipartite, we apply \Cref{thm.ctfemg}. Let $S'_r$ be the set of edges $(z_1,z_2)$ of ${\rm DL}(p,q)$ such that $z_1\in S_r$ and $z_2\in S_{r-1}$. Let $W_r^{a,b}$ be the set of edges $(z_1,z_2)\in S'_r$ such that $z_1\in U_r^{a,b}$. Every vertex in $U_r^{a,b}$ is adjacent to exactly $p$ edges in $W_r^{a,b}$ if $a>0$, and otherwise, it is adjacent to a unique edge in $W_r^{a,b}$. So, $\mynorm{W_r^{a,b}} = p\mynorm{U_r^{a,b}}$ if $a>0$ and $\mynorm{W_r^{0,b}}=\mynorm{U_r^{0,b}}$. Since the rest of the points of $S_r$ have degree at most $p+q$, one obtains that the sets $W_r^{a,b}$ occupy almost all of $S'_r$. Similarly to~\eqref{eq:S+:2}, one can obtain 
		\begin{equation}
			\label{eq:dl-growth1edge}
			p^{-r}\mynorm{S'_r} = \frac{(p+1)(p-1)(p+q)}{p(p^2-q)}+o(1) \;.
		\end{equation}
		Also, one can see that, as $r\to\infty$, the uniform measure on $U_r^{a,b}$ converges to the push forward of $\eta|V_b$ under the diagonal map $\mathrm{diag}(z)\coloneqq(z,z)$. So, by \Cref{lem:sum} again, the uniform measure on $S'_r$ converges to $\mathrm{diag}^*\eta$. Finally, since this measure is supported on the diagonal, as in the proof of \Cref{thm:xi}, the measure $\nu_t$ in \Cref{thm.ctfemg} does not depend on $t$ and is equal to $\eta$, and the claim is proved.
	\end{proof}

    \begin{proof}[Proof of \Cref{thm.dl-convergence} (when $p=q$)]
		We modify the above proof of the case $p>q$ as follows.
		By~\eqref{eq:S+:2} and $\mynorm{S_r^+}=\mynorm{S_r^-}$, one obtains 
		\begin{equation}
			\label{eq:dl-growth2}
			p^{-r}\mynorm{S_r} = \frac{2(p+1)(p^2-1)}{p(p^2-q)} + o(1) = \frac{2(p+1)^2}{p^2}+o(1).
		\end{equation}
		This implies that $\restrict{F}{\mathbb N}$ is log-regularly-varying with index $p$. The same proof as that of the case $p>q$ shows that the uniform measure on $S_r^+$ converges to the harmonic measure $\eta$ on $\partial T$. By symmetry, the uniform measure on $S_r^-$ converges to the harmonic measure $\eta'$ on $\partial T'$. Hence, the uniform measure on $S_r$ converges to $(\eta+\eta')/2$. So, the conditions of \Cref{thm.ctfg} are satisfied and $\nu_0=(\eta+\eta')/2$.
		
		For the edge-measured version, let $S'_{r,\pm}$ be the set of edges in $S'_r$ that have one end point in $S_r^{\pm}$. Define $S'_{r,0}$ similarly. By the same arguments as the case $p>q$, and by symmetry, one obtains that $S'_{r,0}$ is negligible and 
		\begin{equation}
			\label{eq:dl-growth2edge}
			p^{-r}\mynorm{S'_r} = \frac{2(p+1)(p-1)(p+q)}{p(p^2-q)}+o(1) = \frac{4(p+1)}{p} + o(1) \; .
		\end{equation}
		Also, the same arguments imply that the uniform measure on $S'_{r,+}$ converges to $\mathrm{diag}^*\eta$. By symmetry, the uniform measure on $S'_{r,-}$ converges to $\mathrm{diag}^*\eta'$. Hence, the uniform measure on $S'_r$ converges to $(\mathrm{diag}^*\eta + \mathrm{diag}^*\eta')/2$. Since this measure is supported on the diagonal, one obtains again that $\nu_t$ does not depend on $t$ and is equal to $(\eta+\eta')/2$. So, the claim is proved.
	\end{proof}

    	\begin{proof}[Proof of \Cref{cor:xi-examples} for ${\rm DL}(p,q)$]
		Since ${\rm DL}(p,q)$ is bipartite, one has $\alpha_0=0$. \Cref{thm.dl-convergence} and its proof show that $\alpha$ is
		the push forward of $\nu_0$ under the diagonal map $\mathrm{diag}(z)\coloneqq(z,z)$, which
		is supported on the diagonal of $(\partial {\rm DL}(p,q))^2$. So, the claim is implied by \Cref{thm.ioxi}.
	\end{proof}

    \subsection{Geometric Properties of the Cells}\label{sec.ipvtdlpqgeop}
	
	As in the proof of \Cref{thm.dl-convergence}, let $\eta$ and $\eta'$ be the harmonic measures on $\partial T$ and $\partial T'$ respectively.
	\Cref{thm.dl-convergence} and \Cref{prop:diagramConvergence} imply that:
	\begin{cor}[\textsc{IPVT of ${\rm DL}(p,q)$}]
		\label{cor:DL}
		Every IPVT of ${\rm DL}(p,q)$ is of the form $\Vor(\Phi)$, where $\Phi=(\Theta_i,\Delta_i)_{i\geq 1}$ is the nuclei process described as follows: $(\Theta_i)_{i\geq 1}$ are i.i.d. points on $\partial {\rm DL}(p,q)$ with distribution $\nu_0$, where $\nu_0=\eta$ if $p>q$ and $\nu_0=(\eta+\eta')/2$ if $p=q$. Also, independently from $(\Theta_i)_{i\geq 1}$, the delays $(\Delta_i)_{i\geq 1}$ form a Poisson point process on $\mathbb R$ with intensity measure $\beta^{(\xi)}$ described in \Cref{thm.ctfg}, where $\xi\in [0,1)$ is a fixed parameter. For the edge-measured version, the same statement holds but $\beta^{(\xi)}$ is the measure described in \Cref{thm.ctfemg}.
	\end{cor}
	
	We say that a cell is of the \dfn{first type} if its nucleus belongs to $\partial T\subseteq\partial {\rm DL}(p,q)$, and is of the \dfn{second type} if its nucleus belongs to $\partial T'$.	
	No automorphism of the directed version of ${\rm DL}(p,p)$ can swap $\partial T$ and $\partial T'$. Therefore, using the notion of \textit{indistinguishability} defined in \cite{LySc99}, one obtains: 
	\begin{prop}[\textsc{Distinguishable Cells}]
		\label{prop:distinguishable}
		In the directed version of (the vertex-measured or edge-measured) ${\rm DL}(p,p)$, the cells of an IPVT are distinguishable by automorphism-invariant events. As a result, in the corresponding Cayley graph of $\mathbb{Z}_{p} \wr \mathbb{Z}$, the cells of an IPVT are distinguishable by group-invariant events.
	\end{prop}
	{Note that, since we have used the directed version of ${\rm DL}(p,q)$,  $\mathrm{Aut}({\rm DL}(p,p))$ is a subgroup of $\mathrm{Aut}(T)\times \mathrm{Aut}(T')$ and is homomorphic to $\mathbb{Z}_{p} \wr \mathbb{Z}$ (see~\cite{woess2005}). This holds in the undirected version only if $p\neq q$.}
	
	\begin{remark}[\textsc{Indistinguishable Cases}]
		In fact, one can prove by the method of~\cite{M26} that there are exactly two indistinguishability classes for the cells. 
		Also, when $p\neq q$ or when ${\rm DL}(p,q)$ is considered as an undirected graph, the cells are indistinguishable.
	\end{remark}
	
	\begin{prop}[\textsc{Intersection of Cells}]
		\label{prop:intersection}
		The pairwise intersections of the cells of ${\rm IPVT}({\rm DL}(p,q))$ satisfy the following almost surely:
		\begin{enumerate}[label=(\roman*)]
			\item \label{prop:intersection:diff} Two cells of different type have finite or empty intersection (resp.~adjacency). 
			\item \label{prop:intersection:same} Two cells of the same type have infinite or empty intersection (resp.~adjacency).
			\item \label{prop:intersection:all} In addition, every cell attains each of these intersection/adjacency types infinitely often.
		\end{enumerate}
	\end{prop}
	\begin{proof}
		Let $\pi:{\rm DL}(p,q)\to \rtree{p}$ and $\pi':{\rm DL}(p,q)\to\rtree{q}$ denote the projection maps onto the first and second coordinates respectively.
		If $p> q$, the claim is implied by the fact ${\rm IPVT}({\rm DL}(p,q)) = \pi^{-1}({\rm IPVT}(T))$ (see \Cref{thm.dl-convergence}). So, assume $p=q$. Fix the parameter $\xi\in\mathbb R$. By \Cref{cor:DL}, the nuclei process can be written as $\Psi\cup\Psi'$, where $\Psi=(\Theta_i,\Delta_i)_{i\geq 1}$ is a Poisson point process on $\widehat{\partial T}$ with intensity measure $\frac 12 \beta^{(\xi)}\eta$, and $\Psi'=(\Theta'_i,\Delta'_i)_{i\geq 1}$ is a Poisson point process on $\widehat{\partial T'}$ with intensity measure $\frac 12 \beta^{(\xi)}\eta'$ independent from $\Psi$ (note that $\frac 12\beta^{(\xi)} = \beta^{(\xi-\log_p(2))}$). Fix $(\theta',\delta')\in\widehat{\partial T'}$ such that $\theta'\neq \omega'_0$ and let $\tilde\Psi'\coloneqq\Psi'\cup\{(\theta',\delta')\}$. By Mecke's theorem, it is enough to prove that the cell of $(\theta',\delta')$ in $\Vor(\Psi\cup \tilde\Psi')$ satisfies the claims a.s.
		
		\ref{prop:intersection:diff}. 
		Fix $(\theta,\delta)\in\widehat{\partial T}$ such that $\theta\neq\omega_0$ and let $\tilde\Psi\coloneqq\Psi\cup\{(\theta,\delta)\}$. To prove~\ref{prop:intersection:diff}, by Mecke's theorem, it is enough to show that the cells of $(\theta,\delta)$ and $(\theta',\delta')$ in $\Vor(\tilde\Psi\cup\tilde\Psi')$ have either empty or finite intersection a.s.
		Let $C$ be the cell corresponding to $(\theta,\delta)$ in $\Vor(\tilde\Psi)$. Since $\theta\neq\omega_0$, we may let $l>-\infty$ be the minimum of $h$ on $C$. Define $C'$ and $l'$ similarly.
		Clearly, the corresponding two cells in $\Vor(\tilde\Psi\cup\tilde\Psi')$ are included in $\pi^{-1}(C)$ and $(\pi')^{-1}(C')$ respectively. So, their intersection is included in $\pi^{-1}(C)\cap (\pi')^{-1}(C')$. The last set is finite since each $(x,x')$ in this set satisfies $l_1\leq h(x)\leq -l_2$ and $l_2\leq h(x')\leq -l_1$, and the set of points in $C\times C'$ with these properties are finite. This implies the claim of~\ref{prop:intersection:diff}.
		
		\ref{prop:intersection:same}. Let $\psi'$ be a sample of $\Psi'$ and fix $f_1,f_2\in\psi'$. Let $C'_i$ (resp. $\hat C'_i$) be the cell of $f_i$ in $\Vor(\psi')$ (resp. in $\Vor(\Psi\cup \psi')$). Clearly, $\hat C'_i\subseteq (\pi')^{-1}(C'_i)$. So, if $C'_1\cap C'_2=\emptyset$, then $\hat C'_1\cap \hat C'_2=\emptyset$. This shows that $\hat C'_1$ avoids infinitely many cells of the same type.
		
		Now, assume that $C'_1\cap C'_2\neq\emptyset$. We prove that $\hat C'_1\cap \hat C'_2$ is infinite (the case where $C'_1$ is adjacent to $C'_2$ is similar). Choose $x'\in C'_1\cap C'_2$ arbitrarily and let $L\coloneqq\{x\in T: h(x)=-h'(x')\}$. The points $(x,x')$, where $x\in L$ are candidate points in $\hat C'_1\cap\hat C'_2$, but we should compare the distance of $(x,x')$ to the other nuclei as well.	
		In \Cref{lem:ergodic}, we will show that $L$ is an ergodic unimodular discrete space (where $\Psi$ is kept as a decoration). The set $W\coloneqq \{x\in L: (x,x')\in \hat C'_1\cap\hat C'_2\}$ is an equivariant subset of $L$. This implies that $W$ is either empty a.s. or infinite a.s. (see Lemma~2.9 of~\cite{eft}). Any point of $L$ has a positive probability of being in $W$ (given the sample $\psi'$ and given $f'_1,f'_2\in\psi'$). This implies that $W$ is infinite a.s.
		
		\ref{prop:intersection:all}.
		The claim for cells of the same type is already proved in~\ref{prop:intersection:same}. So, consider the claim for cells of different types. Using the notations of the proof of~\ref{prop:intersection:diff}, fix a realization $\psi'$ of $\Psi'$ and an element $f'=(\theta',\delta')$ of $\widehat{\partial T'}$. Let $\tilde\psi'\coloneqq \psi'\cup\{f'\}$. We will use unimodularity and ergodicity on a subset of $\Psi$.
		Let $\Gamma\coloneqq \{\gamma\in \mathrm{Aut}(T): h\circ\gamma=h\}$ be the stabilizer of $h$. By~\Cref{lem:autInv}, the action of $\Gamma$ preserves the distribution of $\Psi$. The orbits of the action of $\Gamma$ on $\widehat{\partial T}$ are $\tau^{-1}(k)$ for $k\in\mathbb Z$, where $\tau(f)\coloneqq  \min\{h(x): x\in f^{-1}(0)\}$ for $f\in\widehat{\partial T}$. Fixing $k\in\mathbb Z$, let $\Psi_k\coloneqq \Psi\cap\tau^{-1}(k)$. Let $A$ be the set of $f\in\widehat{\partial T}$ such that the geodesic connecting $\omega_0$ and $f$ passes through $\origin$. Bias the probability measure by $\mynorm{A\cap \Psi_k}$ and let $\bs f$ be a random element of $A\cap \Psi_k$.
		
		\textbf{Claim}. $\Psi_k$, after biasing and choosing $\bs f$ as the root, is a unimodular  discrete space.
		
		To prove this claim, note that $\tau^{-1}(k)$ is a limit of the level set $h^{-1}(n)$ of $h$ as $n\to+\infty$, after an appropriate shift and an appropriate scaling of the counting measure. So, $\tau^{-1}(k)$ is a unimodular continuum space (in fact, a transitive space with a unimodular automorphism group); see~\cite{Kh23unimodularspaces} for further discussion (note that a suitable $\Gamma$-invariant metric should be chosen on $\tau^{-1}(k)$; e.g., $d(f,g)\coloneqq 2^{-h(f\wedge g)}$, where $f\wedge g$ is the confluence of the ideal points corresponding to $f$ and $g$. This metric can be extended to $\{f\in \mathrm{SD}(T): \tau(f)=k\}$ similarly). Since $\Psi_k$ is a Poisson point process on $\tau^{-1}(k)$, the above claim is implied. {It should be noted that $\Psi$ is not itself unimodular since, intuitively, $\Psi_k$ is \textit{coarser} than $\Psi_{k+1}$ for every $k$.}
		
		In addition to the last claim, it holds that $\Psi_k$ is ergodic. Also, this claim holds when keeping the rest of $\Psi$ as a decoration. 
		
		Finally, if $k-\delta'$ is even, let $U$ (resp. $V$) be the set of $f\in \Psi_k$ such that the cells of $f$ and $f'$ in $\Vor(\tilde\Psi\cup\tilde\psi')$ intersect (resp. do not intersect). If $k$ is odd, define the same but replace intersection with adjacency. Noting that $\tilde\psi'$ is fixed, if $k$ is a large negative number (depending on the cell of $f'$ in $\Vor(\tilde\psi')$), then each of $U$ and $V$ is an equivariant subset of $\Psi_k$ and contains the root $\bs f$ with positive probability. So, the last ergodicity implies that $U$ and $V$ are nonempty and infinite a.s. 
	\end{proof}

    \begin{lem}
		\label{lem:ergodic}
		Let $L$ be a level-set of $h$ in $T=\rtree{p}$, choose $x_0\in L$ arbitrarily, and let $\Psi$ be an IPVT of $T$. Then, $[L,x_0;\Psi]$ is an ergodic unimodular discrete space.
	\end{lem}
	\begin{proof}
		The claim is implied by the fact that the action of $\mathrm{Aut}(\rtree{p})$ on ${\rm IPVT}(\rtree{p})$ is mixing\footnote{Personal communication with Sam Mellick.} (since the nuclei form a Poisson point process on $\widehat{\partial \rtree{p}}$), but we provide a simple direct argument similarly to the proof of the fact that the Poisson point process on $\mathbb R^d$ is mixing. 
		For this goal, it is enough to construct a sequence of compact sets $K_1\subseteq K_2\subseteq \cdots$ that exhaust $\widehat{\partial \rtree{p}}$ such that for each $n$, for all $\gamma$ outside some compact subset of $\mathrm{Aut}(\rtree{p})$, one has $\gamma K_n\cap K_n=\emptyset$. This is satisfied by letting $K_n$ be the set of horofunctions $f$ such that $f(\origin)\leq n$ and $h(f\wedge\origin)\geq -n$. This proves the claim
	\end{proof}
	
	\begin{prop}
		\label{prop:dl-biinfinite}
		No cell of ${\rm IPVT}({\rm DL}(p,q))$ contains a bi-infinite geodesic.
	\end{prop}
	\begin{proof}
		Similarly to the proof of \Cref{prop:intersection}, let $\Phi=\Psi\cup\Psi'$ be an IPVT of ${\rm DL}(p,q)$.
		Let $C'$ be a cell with nucleus $(\theta',\delta')$, where $\theta'\in\partial T'\setminus\{\omega'_0\}\subseteq \partial {\rm DL}(p,q)$. First, we show that $C'$ does not contain any geodesic between $\theta'$ and $\theta'_1\in\partial T'\setminus\{\omega'_0\}$. If it does, let $x'\coloneqq \theta'\wedge\theta'_1$. It can be seen that every point of $L\coloneqq \{(x,x'): x\in T, h(x)=-h'(x')\}$ is on a geodesic between $\theta'$ and $\theta'_1$, which implies that $L\subseteq C'$. This gives a contradiction as follows: As in the proof of \Cref{prop:intersection}, when $\Psi',\theta'$ and $x'$ are fixed, $L$ is an ergodic unimodular discrete space (while keeping $\Phi$ as a decoration). Every point of $L$ has a positive probability of not being in $C'$. This implies that $L\not\subseteq C'$ a.s.
		
		Second, we prove that $C'$ does not contain any geodesic that connects $\theta'$ to $\partial T\setminus\{\omega_0\}\subseteq \partial {\rm DL}(p,q)$. 
		Define the \dfn{trunk} of $C'$ as the set of $(x,x')\in C'$ that have a unique geodesic towards $\theta'$; i.e. $x'\in [\theta' \omega'_0]$. These geodesics make a family tree structure in the trunk. If we choose $C'$ as a typical cell that contains $(\origin,\origin')$ from the beginning (as described in the proof of \Cref{prop:intersection}), then the trunk becomes a unimodular random family tree. We should prove that the trunk is one-ended. By the general classification of unimodular family trees provided in~\cite{eft}, it is enough to show that the level-sets of the trunk are infinite a.s. This is also implied by the stationarity of $\Psi$ under the action of the stabilizer of $h$ similarly to the proof of \Cref{prop:intersection}.
		
		Finally, if $C'$ contains a bi-infinite geodesic between two other points of $\partial {\rm DL}(p,q)$, one can deduce that $C'$ contains a geodesic between $\theta'$ and another ideal point. So, by the previous arguments, no such geodesics exist.
	\end{proof}

    \section{Problems and Future Directions}		
		\label{sec:problems}
		
		\begin{question}
			Is there a transitive graph (or a Cayley graph) with log-regularly-varying growth, where full convergence of \Cref{thm.ctfg} does not hold? Recall that the volume function of $\mathbb Z_2 * \mathbb Z_3$ is not log-regularly-varying, see \cite[Figure 2.7]{bhupatiraju}.
		\end{question}
		
		
		\begin{question}[\textsc{{Other Point Processes of Nuclei}}]
			{Is it possible to build a non-trivial random ideal Voronoi diagram via low-intensity limits (in some sense) of Voronoi tessellations of other point processes different from the IPVT? A natural candidate to consider in dimension 2 is the $\alpha$-hyperbolic ensemble considered in \cite{BGZ22} in the limit $\alpha \downarrow 0$.}
		\end{question}
		
		\begin{question}[\textsc{{Product with $L^1$ Metric More General than \Cref{prop:product-L1}}}] 
			Without the condition in item~\ref{prop:product-L1:=} of \Cref{prop:product-L1}, what can be said about full convergence towards ${\rm IPVT}$? It seems that the limiting behavior of $F(r)/F'(r)$ would play an important role to address this question.
		\end{question}
		
		\begin{question}[\textsc{{Product with $L^2$ Metric}}] 
			\label{prob:l2}
			Under what conditions the produc of two metric spaces endowed with the $L^2$ metric satisfies \Cref{thm.scfc} or \Cref{thm.ctfg}? And what would be the limiting measure in this case? We believe that under some mild conditions the limit is identical to the limit with the weighted $L^1$ metric discussed in \Cref{prop:product-L1}.
		\end{question}
		\begin{question}[\textsc{Edge-Measured Product Graphs}]
			When does the edge-weighted version of the product of two graphs $G\times G'$ satisfy full convergence of \Cref{thm:criterion-edge}? It seems that the limit of the uniform measure on large balls in both vertex measured and edge measured versions play a role.
		\end{question}
		
		\begin{question}[\textsc{Types of IPVT Cells for Groups}]\label{q:mf}
				In the IPVT of Diestel--Leader graphs, one observes finitely many (but not one) types of cells, see \Cref{sec.ipvtdlpqgeop}. For IPVTs of groups, can one observe infinitely many distinct types of cells, or a finite number greater than two? Natural candidates to consider are Baumslag-Solitar groups or (general) lamplighters.
		\end{question}

\bibliographystyle{alpha}
		
		\bibliography{biblio}
\end{document}